\documentclass[12pt]{amsart}
\usepackage{graphicx}
\graphicspath{ {./images/} }
\usepackage{float}
\usepackage{amsmath, amsthm, amsfonts, amssymb} 
\usepackage{mathrsfs} 
\usepackage{bbm} 
\usepackage{enumerate}
\usepackage{comment}
\usepackage{hyperref}
\usepackage{xcolor} 
\usepackage{chngcntr} 
\counterwithin{table}{section} 
\newtheorem{thm}{Theorem}[section]
\newtheorem*{claim}{Claim}
\newtheorem{cor}[thm]{Corollary}
\newtheorem{lem}[thm]{Lemma}
\newtheorem{prop}[thm]{Proposition}
\newtheorem*{prob*}{Problem}
\newtheorem*{ques*}{Question}
\newtheorem*{thm*}{Theorem}
\newtheorem{quest}[thm]{Question}

\theoremstyle{definition}
\newtheorem{defn}[thm]{Definition}
\newtheorem{example}[thm]{Example}
\newtheorem{conj}[thm]{Conjecture}
\newtheorem*{defn*}{Definition}
\newtheorem{rem}[thm]{Remark}

\newtheorem{rem*}[thm]{Remark}
\numberwithin{equation}{section}

\newcommand{\R}{\mathbb R}
\newcommand{\Z}{\mathbb Z}
\newcommand{\N}{\mathbb N}
\newcommand{\Q}{\mathbb Q}
\newcommand{\Zp}{\mathbb Z_{\geq 0}}
\newcommand{\F}{\mathbb F}
\newcommand{\T}{\mathbb{T}}
\newcommand{\eps}{\varepsilon}

\DeclareMathOperator{\E}{\mathbb{E}}

\newcommand{\mG}{\mathcal {G}}

\newcommand{\mI}{\mathcal {I}}
\newcommand{\mZ}{\mathcal {Z}}
\newcommand{\mX}{\mathcal {X}}
\newcommand{\mY}{\mathcal {Y}}
\newcommand{\X}{\textbf{X}}
\newcommand{\Y}{\textbf{Y}}
\newcommand{\bfZ}{\textbf{Z}}
\newcommand{\UC}{\text{UC~-}\lim}
\newcommand{\ind}{\mathbbm{1}}
\newcommand{\norm}[2]{\left\| #2 \right\|_{#1}}
\newcommand{\innprod}[2]{\left\langle #1, #2 \right\rangle}
\newcommand{\es}{\emptyset}
\newcommand{\rk}{\text{rk}}
\newcommand{\epdd}{\text{epdd}}
\newcommand{\syndsup}{\text{synd-sup}}
\renewcommand{\tilde}{\widetilde}
\renewcommand{\hat}{\widehat}

	\title[Khintchine-type recurrence for 3-point configurations]{Khintchine-type recurrence for 3-point configurations}
	
	\author[E. Ackelsberg]{Ethan Ackelsberg}
	\address{Department of Mathematics, Ohio State University}
	\email{ackelsberg.1@osu.edu}
	
	\author[V. Bergelson]{Vitaly Bergelson}
	\address{Department of Mathematics, Ohio State University}
	\email{vitaly@math.ohio-state.edu}
	
	\author[O. Shalom]{Or Shalom}
	\address{Einstein Institute of Mathematics, Hebrew University of Jerusalem}
	\email{or.shalom@mail.huji.ac.il}
	
\date{\today}

\begin{document}
	\begin{abstract}
	The goal of this paper is to generalize, refine, and improve results on large intersections from \cite{BHK,ABB}.
	We show that if $G$ is a countable abelian group and $\varphi, \psi : G \to G$ are homomorphisms such that at least two of the three subgroups $\varphi(G)$, $\psi(G)$, and $(\psi-\varphi)(G)$ have finite index in $G$, then $\{\varphi, \psi\}$ has the \emph{large intersections property}.
	That is, for any ergodic measure preserving system $\X=(X,\mX,\mu,(T_g)_{g\in G})$, any $A\in \mX$, and any $\eps>0$, the set
	$$\{g\in G : \mu(A\cap T_{\varphi(g)}^{-1} A \cap T_{\psi(g)}^{-1} A)>\mu(A)^3-\varepsilon\}$$
	is syndetic (Theorem \ref{Khintchinefiniteindex}).
	Moreover, in the special case where $\varphi(g)=ag$ and $\psi(g)=bg$ for $a,b\in \mathbb{Z}$, we show that we only need one of the groups $aG$, $bG$, or $(b-a)G$ to be of finite index in $G$ (Theorem \ref{Khintchineab}), and we show that the property fails in general if all three groups are of infinite index (Theorem \ref{counterexample}).
	
	One particularly interesting case is where $G=(\mathbb{Q}_{>0},\cdot)$ and $\varphi(g)=g$, $\psi(g)=g^2$, which leads to a multiplicative version for the Khintchine-type recurrence result in \cite{BHK}.
	We also completely characterize the pairs of homomorphisms $\varphi,\psi$ that have the large intersections property when $G = \Z^2$.
	
	The proofs of our main results rely on analysis of the structure of the \emph{universal characteristic factor} for the multiple ergodic averages $$\frac{1}{|\Phi_N|} \sum_{g\in \Phi_N}T_{\varphi(g)}f_1\cdot T_{\psi(g)} f_2.$$
	In the case where $G$ is finitely-generated, the characteristic factor for such averages is the \emph{Kronecker factor}.
	In this paper, we study actions of groups that are not necessarily finitely-generated, showing in particular that, by passing to an extension of $\X$, one can describe the characteristic factor in terms of the \emph{Conze--Lesigne factor} and the $\sigma$-algebras of $\varphi(G)$ and $\psi(G)$ invariant functions (Theorem \ref{strongcf}).
	\end{abstract}

\maketitle

\tableofcontents

\section{Introduction}

Let $(G,+)$ be a countable abelian group. A \emph{probability measure-preserving $G$-system}, or simply \emph{$G$-system} for short, is a quadruple $\X=(X,\mX,\mu, (T_g)_{g\in G})$ where $(X,\mX,\mu)$ is a standard Borel probability space (that is, up to isomorphism of measure spaces, $X$ is a compact metric space, $\mX$ is the Borel $\sigma$-algebra, and $\mu$ is a regular Borel probability measure) and $T_g :X \to X$, $g \in G$, are measure-preserving transformations such that $T_{g+h}=T_g\circ T_h$ for every $g,h\in G$ and $T_0=Id$. The transformation $T_g:X \to X$ gives rise to a unitary operator on $L^2(\mu)$, which we also denote by $T_g$, given by the formula $T_g f(x) = f(T_g x)$. We say that a $G$-system is \emph{ergodic} if the only measurable $(T_g)_{g\in G}$-invariant functions are the constant functions.

\subsection{Khintchine-type recurrence and the large intersections property}

The starting point for the study of recurrence in ergodic theory is the Poincar\'{e} recurrence theorem, which states that, for any measure-preserving system $\left(X, \mX, \mu, T \right)$ and any set $A \in \mX$ with $\mu(A) > 0$, there exists $n \in \N$ such that $\mu(A \cap T^{-n}A) > 0$.

Khintchine's recurrence theorem strengthens and enhances Poincar\'{e}'s recurrence theorem by improving on the size of the intersections and the size of the set of return times.

\begin{thm}[Khintchine's recurrence theorem \cite{Kh}]
    For any measure-preserving system $\left(X, \mX, \mu, T \right)$, any $A \in \mX$, and any $\eps > 0$, the set
    \begin{align*}
        \left\{ n \in \N : \mu \left( A \cap T^{-n}A \right) > \mu(A)^2 - \eps \right\}
    \end{align*}
    has bounded gaps.
\end{thm}

Khintchine's recurrence theorem easily extends to general semigroups, where the appropriate counterpart of ``bounded gaps'' is the notion of syndeticity. In this paper, we deal with recurrence in countable abelian groups. A subset $A$ of a countable discrete abelian group $G$ is said to be \emph{syndetic} if there exists a finite set $F\subseteq G$ such that $A+F = \{a+f : a\in A, f\in F\} = G$.\\

It is natural to ask if recurrence theorems other than Poincar\'{e}'s recurrence theorem also have Khintchine-type enhancements.
For instance, it follows from the IP Szemer\'{e}di theorem of Furstenberg and Katznelson \cite{FK-IP} and also from \cite[Theorem B]{austin} that, for any abelian group $G$, any $k \in \N$, and any family of homomorphisms $\varphi_1, \dots, \varphi_k : G \to G$, the following holds: if $\left( X, \mX, \mu, (T_g)_{g \in G} \right)$ is a $G$-system and $A \in \mX$ has $\mu(A) > 0$, then the set
\begin{align*}
    \left\{ g \in G : \mu \left( A \cap T_{\varphi_1(g)}^{-1}A \cap \dots \cap T_{\varphi_k(g)}^{-1}A \right) > 0 \right\}
\end{align*}
is syndetic.\footnote{In fact, this set is an IP$^*$ set, which is a stronger notion of largeness that we do not address in this paper; see \cite{FK-IP}.}
With the goal of Khintchine-type enhancements in mind, this motivates the following definition:

\begin{defn}
    A family of homomorphisms $\varphi_1, \dots, \varphi_k : G \to G$ has the \emph{large intersections property} if the following holds: for any ergodic $G$-system $\left( X, \mX, \mu, (T_g)_{g \in G} \right)$, any $A \in \mX$ and any $\eps > 0$, the set
    \begin{align*}
        \left\{ g \in G : \mu \left( A \cap T_{\varphi_1(g)}^{-1}A \cap \dots \cap T_{\varphi_k(g)}^{-1}A \right) > \mu(A)^{k+1} - \eps \right\}
    \end{align*}
    is syndetic.
\end{defn}

The large intersections property is closely related to the phenomenon of popular differences in combinatorics; see, e.g., \cite{SSZ,Berger,Man,BSST,AB}.\\

Determining which families of homomorphisms have the large intersections property is a challenging problem with many surprising features.
In the case $G = \Z$ and $\varphi_i(n) = in$, the problem was resolved in \cite{BHK}.
\begin{thm}[\cite{BHK}, Theorems 1.2 and 1.3] \label{thm: bhk}
    The family $\{n, 2n, \dots, kn\}$ has the large intersections property in $\Z$ if and only if $k \le 3$.
\end{thm}

Later work of Frantzikinakis and of Donoso, Le, Moreira, and Sun generalized this picture for arbitrary homomorphisms $\Z \to \Z$, which take the form $n \mapsto an$ for some $a \in \Z$.
\begin{thm}[\cite{Fra}, special case of Theorem C; \cite{DLMS}, Theorem 1.5]~
    \begin{enumerate}
        \item   For any $a, b \in \Z$, the families $\{an,bn\}$ and $\{an, bn, (a+b)n\}$ have the large intersections property (in $\Z$).
        \item   For any $k \ge 4$ and any distinct and nonzero integers $a_1, \dots, a_k \in \Z$, the family $\{a_1n, \dots, a_kn\}$ does not have the large intersections property (in $\Z$).
    \end{enumerate}
\end{thm}

\begin{rem}
Finitary combinatorial work of \cite[Theorem 1.6]{SSZ} suggests that the family $\{a_1n,a_2n,a_3n\}$ has the large intersections property if and only if $a_i+a_j=a_k$ for some permutation $\{i,j,k\}$ of $\{1,2,3\}$.
\end{rem}

In \cite{BTZ}, Khintchine-type recurrence results are established in the infinitely-generated torsion groups $G = \bigoplus_{n=1}^{\infty}{\Z/p\Z}$.
\begin{thm}[\cite{BTZ}, Theorems 1.12 and 1.13] ~
    \begin{enumerate}
        \item Fix a prime $p > 2$. If $c_1, c_2 \in \Z/p\Z$ are distinct and nonzero, then $\{c_1g,c_2g\}$ has the large intersections property in $G = \bigoplus_{n=1}^{\infty}{\Z/p\Z}$.
        \item Fix a prime $p > 3$. If $c_1, c_2 \in \Z/p\Z$ are distinct and nonzero and $c_1 + c_2 \ne 0$, then $\{c_1g, c_2g, (c_1+c_2)g\}$ has the large intersections property in $G = \bigoplus_{n=1}^{\infty}{\Z/p\Z}$.
    \end{enumerate}
\end{thm}

\begin{rem}
    It is conjectured in \cite[Conjecture 1.14]{BTZ} that, if $c_1, c_2, c_3 \in \Z/p\Z$ are distinct and nonzero and $c_i+c_j \ne c_k$ for every permutation $\{i,j,k\}$ of $\{1,2,3\}$, then $\{c_1g, c_2g, c_3g\}$ does not have the large intersections property in $G = \bigoplus_{n=1}^{\infty}{\Z/p\Z}$.
\end{rem}

\bigskip

Khintchine-type recurrence in general abelian groups was addressed in \cite{ABB} and \cite{OS2}.
For 3-point linear configurations, the following was shown in \cite{ABB}:
\begin{thm}[\cite{ABB}, Theorem 1.10] \label{thm: abb}
    Let $G$ be a countable discrete abelian group.
    Let $\varphi,\psi:G \to G$ be homomorphisms.
    If all three of the subgroups $\varphi(G)$, $\psi(G)$, and $(\psi-\varphi)(G)$ have finite index in $G$, then $\{\varphi,\psi\}$ has the large intersections property.
\end{thm}

\begin{rem} \label{rem: chu}
Earlier work of Chu demonstrates that at least some finite index condition is necessary for large intersections. Namely, it follows from \cite[Theorem 1.2]{Chu} that the pair $\{(n,0), (0,n)\}$, does not have the large intersections property in $\Z^2$; see \cite[Example 10.2]{ABB}.
\end{rem}

For more restricted 4-point configurations, the following result was shown in \cite{ABB} and independently in \cite{OS2}:
\begin{thm}[\cite{ABB}, Theorem 1.11; \cite{OS2}, Theorem 1.3]
    Let $G$ be a countable discrete abelian group. Let $a, b \in \Z$ be distinct, nonzero integers such that all four of the subgroups $aG$, $bG$, $(a+b)G$, and $(b-a)G$ have finite index in $G$. Then $\{ag, bg, (a+b)g\}$ has the large intersections property.
\end{thm}

\subsection{Main results}

In this paper, we refine the understanding of Khintchine-type recurrence for 3-point configurations in abelian groups and make substantial progress towards characterizing the pairs of homomorphisms $\varphi, \psi : G \to G$ that have the large intersections property.

Our first result shows that the large intersections property holds for any pair of homomorphisms $\{\varphi,\psi\}$ so long as at least two of the three subgroups in Theorem \ref{thm: abb} have finite index in $G$. In particular, this shows that \cite[Conjecture 10.1]{ABB} is false.
\begin{thm}\label{Khintchinefiniteindex}
Let $G$ be a countable discrete abelian group. Let $\varphi, \psi : G \to G$ be homomorphisms such that at least two of the three subgroups $\varphi(G)$, $\psi(G)$ and $(\psi-\varphi)(G)$ have finite index in $G$. Then for any ergodic $G$-system $\left( X, \mX, \mu, (T_g)_{g \in G} \right)$, any $A \in \mX$, and any $\varepsilon>0$, the set $$\left\{g\in G : \mu\left(A\cap T_{\varphi(g)}^{-1} A \cap T_{\psi(g)}^{-1} A\right)  > \mu(A)^3-\varepsilon \right\}$$ is syndetic. 
\end{thm}

As mentioned above (see Remark \ref{rem: chu}), the work of Chu \cite{Chu} provides a counterexample to the large intersections property when all three subgroups $\varphi(G)$, $\psi(G)$, and $(\psi - \varphi)(G)$ have infinite index in $G$. In this paper, we give additional counterexamples for the group $G = \bigoplus_{n=1}^{\infty}{\Z}$ with homomorphisms $g \mapsto ag$ and $g \mapsto bg$ for some $a, b \in \Z$; see Theorem \ref{counterexample} below.
A natural question to ask, then, is what happens when only one of the subgroups $\varphi(G)$, $\psi(G)$, or $(\psi-\varphi)(G)$ has finite index. Namely:

\begin{quest} \label{Khintchine}
Let $G$ be a countable abelian group, and let $\varphi:G\rightarrow G$, $\psi:G\rightarrow G$ be homomorphisms such that at least one of the subgroups $\varphi(G)$, $\psi(G)$, or $(\psi-\varphi)(G)$ has finite index in $G$. Is it true that, for any ergodic $G$-system $\left( X, \mX, \mu, (T_g)_{g \in G} \right)$, any $A\in \mX$, and any $\varepsilon>0$, the set
$$\left\{g\in G : \mu\left(A\cap T_{\varphi(g)}^{-1} A \cap T_{\psi(g)}^{-1} A\right)  > \mu(A)^3-\varepsilon \right\}$$ is syndetic?
\end{quest}
Note that, by symmetry, it is enough to provide an answer to Question \ref{Khintchine} under the assumption that $(\psi-\varphi)(G)$ has finite index. Indeed, suppose $\psi(G)$ has finite index in $G$. Then since $(T_g)_{g \in G}$ is a measure-preserving action, we have the identity
\begin{align*}
    \mu \left( A \cap T_{\varphi(g)}^{-1}A \cap T_{\psi(g)}^{-1}A \right)
     = \mu \left( A \cap T_{-\varphi(g)}^{-1}A \cap T_{(\psi-\varphi)(g)}^{-1}A \right)
\end{align*}
Hence, the pair $\{\varphi,\psi\}$ has the large intersections property if and only if $\left\{\tilde{\varphi}, \tilde{\psi}\right\}$ has the large intersections property, where $\tilde{\varphi} = - \varphi$ and $\tilde{\psi} = \psi - \varphi$. Moreover, we have $(\tilde{\psi} - \tilde{\varphi})(G) = \psi(G)$, which is of finite index.
A similar argument applies when $\varphi(G)$ has finite index.\\

When $G = \Z^2$, we can use additional tools from linear algebra to classify all pairs of homomorphisms $\varphi$ and $\psi$,
which allows us to answer Question \ref{Khintchine} affirmatively in this setting.
In fact, we can give a precise description of the optimal size of intersections for all 3-point configurations in $\Z^2$;
see Subsection \ref{sec: Z^2 summary} below.
However, our results rely heavily on properties of $2\times2$ matrices,
and it appears that the full generality of Question \ref{Khintchine} for general abelian groups and general homomoprhisms is out of reach without developing new techniques.

On the other hand, in the special case $\varphi(g) = ag$ and $\psi(g) = bg$ for $a, b \in \Z$, we answer Question \ref{Khintchine} affirmatively:
\begin{thm} \label{Khintchineab}
Let $G$ be a countable abelian group. Let $a,b\in\mathbb{Z}$ be integers such that $(b-a)G$ has finite index in $G$. Then for any ergodic $G$-system $\left( X, \mX, \mu, (T_g)_{g \in G} \right)$, any $A \in \mX$, and any $\eps>0$, the set $$\left\{g\in G : \mu\left(A\cap T_{ag}^{-1} A \cap T_{bg}^{-1} A\right)  > \mu(A)^3-\varepsilon \right\}$$ is syndetic.
\end{thm}
We also show that the assumption that $(b-a)G$ has finite index in $G$ is necessary. To see this, we prove the following result:
\begin{thm}\label{counterexample}
Let $G = \bigoplus_{n=1}^{\infty}{\Z}$.
Let $l \in \N$. There exists $P = P(l)$ such that, for any $a, b \in \N$ with $p \mid \gcd(a,b)$ for some prime $p \ge P$, there is an ergodic $G$-system $\left( X, \mX, \mu, (T_g)_{g \in G} \right)$ and a set $A \in \mX$ with $\mu(A) > 0$ such that $$\mu(A\cap T_{ag}^{-1} A\cap T_{bg}^{-1} A)\leq \mu(A)^l$$ for every $g\ne0$.
\end{thm}

\begin{quest}
    Can $p$ in the statement of Theorem \ref{counterexample} be replaced by any natural number?
\end{quest}

\subsection{Applications to geometric progressions and other multiplicative patterns}

One particularly interesting corollary of Theorem \ref{Khintchineab} is a multiplicative version of Theorem \ref{thm: bhk}. Consider the group $G=(\mathbb{Q}_{>0},\cdot)$. This is a multiplicative counterpart of $(\mathbb{Z},+)$. Using an ergodic version of the Furstenberg correspondence principle (see \cite[Theorem 2.8]{BFe}) we deduce the following result:
\begin{thm}
Let $E\subseteq\mathbb{Q}_{>0}$ be a set of positive multiplicative upper Banach density and let $k\in\Z$. Then for any $\eps>0$, the sets \begin{align} \label{eq: gp k,k+1}\left\{q\in\mathbb{Q}_{>0} : d^*_{\text{mult}}\left(E\cap q^{-k}E\cap q^{-(k+1)}E\right)>d^*_{\text{mult}}(E)^3-\varepsilon\right\}\end{align} and \begin{align}\label{eq: gp 1,k}\left\{q\in\mathbb{Q}_{>0} : d^*_{\text{mult}}\left(E\cap q^{-1}E\cap q^{-k}E\right)>d^*_{\text{mult}}(E)^3-\varepsilon\right\}\end{align} are syndetic.
\end{thm}
\begin{rem}
The special case where $k=1$ in \eqref{eq: gp k,k+1} or $k = 2$ in \eqref{eq: gp 1,k} is related to the existence of length three geometric progressions in sets of positive multiplicative density. Heuristically, if $E$ were a random set, where each positive rational number $q\in \mathbb{Q}_{>0}$ is independently chosen to be inside $E$ with probability $\alpha$, then the expected number of geometric progressions of length three and quotient $q$ would be $\alpha^3$. Now fix any set $E$ with $d^*_{mult}(E) = \alpha$. Choosing $\varepsilon$ sufficiently small, our result implies that $E$ contains almost as many geometric progressions with quotient $q$ as a random set with the same density, $\alpha$, for a syndetic set of quotients.
\end{rem}


Theorem \ref{counterexample} shows that, if $n$ and $m$ share a large prime factor, then $\{q^n, q^m\}$ does not have the large intersections property in $(\Q_{>0}, \cdot)$. What happens in the case that $n$ and $m$ are coprime is an interesting question that we are unable to answer with our current methods:

\begin{quest} \label{q: coprime}
	Suppose $n, m \in \N$ are coprime.
	Does the pair $\{q^n, q^m\}$ have the large intersections property in $(\Q_{>0}, \cdot)$?
\end{quest}
Since every $\mathbb{Z}$-action can be lifted to a $(\mathbb{Q}_{>0}, \cdot)$-action (indeed, $(\Q_{>0}, \cdot)$ is torsion-free, so $\Z$ embeds as a subgroup), we see from Theorem \ref{thm: bhk} above that $\{q,q^2,\dots, q^k\}$ does not have the large intersections property for $k \ge 4$. However,
we can still ask about geometric progressions of length $4$.
\begin{quest} \label{q: 4-GP}
	Does the triple $\{q, q^2, q^3\}$ have the large intersections property in $(\Q_{>0}, \cdot)$?
\end{quest}

For a discussion of where our methods come up short for answering Questions \ref{q: coprime} and \ref{q: 4-GP}, see Subsection \ref{sec: seminorm combo} below.

\subsubsection{Patterns in $(\N,\cdot)$}

In Section \ref{sec: semigroups}, we transfer Theorems \ref{Khintchinefiniteindex} and \ref{Khintchineab} to the setting of cancellative abelian semigroups.
As a consequence, we obtain the following result about geometric configurations in the multiplicative integers:

\begin{thm}
    Let $E \subseteq \N$ be a set of positive multiplicative upper Banach density, and let $k\in\Z$. Then for any $\eps>0$, the sets \begin{align*}
        \left\{ m \in \N : d^*_{\text{mult}}\left(E\cap E/m^k \cap E/m^{k+1}\right)
        > d^*_{\text{mult}}(E)^3-\eps \right\}
    \end{align*}
    and
    \begin{align*}
        \left\{ m \in \N : d^*_{\text{mult}}\left(E\cap E/m \cap E/m^k \right) > d^*_{\text{mult}}(E)^3-\eps \right\}
    \end{align*}
    are (multiplicatively) syndetic in $(\N,\cdot)$.
\end{thm}


\subsection{Applications to patterns in $\Z^2$} \label{sec: Z^2 summary}

When $G = \Z^2$, we are able to give a complete picture of the phenomenon of large intersections for 3-point matrix patterns, i.e. patterns of the form $\{\vec{x}, \vec{x} + M_1\vec{n}, \vec{x} + M_2\vec{n}\}$, where $\vec{x}, \vec{n} \in \Z^2$ and $M_1, M_2$ are $2\times2$ matrices with integer entries.
(Note that any homomorphism $\varphi : \Z^2 \to \Z^2$ can be expressed as a $2\times2$ matrix with integer entries, so matrix patterns capture all possible configurations in $\Z^2$ that can be described within the framework of group homomorphisms.) \\

Following \cite{BHK}, we say that the \emph{syndetic supremum}
of a bounded real-valued $\Z^2$-sequence $\left( a_{n,m} \right)_{(n, m) \in \Z^2}$ is the quantity
\begin{align*}
	\syndsup_{(n,m) \in \Z^2}{a_{n,m}}
	 := \sup{\left\{ a \in \R : \left\{ (n,m) \in \Z^2 : a_{n,m} > a \right\}~\text{is syndetic in}~\Z^2 \right\}}.
\end{align*}
For $2\times2$ integer matrices $M_1$ and $M_2$ and $\alpha \in (0,1)$, we define the \emph{ergodic popular difference density} by
\begin{align*}
	\epdd_{M_1,M_2}(\alpha)
	 := \inf{\syndsup_{\vec{n} \in \Z^2} \mu \left( A \cap T_{M_1\vec{n}}^{-1}A \cap T_{M_2\vec{n}}^{-1}A \right)},
\end{align*}
where the infimum is taken over all ergodic $\Z^2$-systems $\left( X, \mX, \mu, (T_{\vec{n}})_{\vec{n} \in \Z^2} \right)$
and sets $A \in \mX$ with $\mu(A) = \alpha$.
This can be seen as an ergodic-theoretic analogue to the \emph{popular difference density} defined in \cite{SSZ}.
It is natural to ask if $\epdd(\alpha)$ coincides with the finitary combinatorial quantity $\text{pdd}(\alpha)$.
Standard tools for translating between ergodic theory and combinatorics, such as Furstenberg's correspondence principle,
are insufficient for resolving this question, and we do not know the answer in general.
However, in special cases where $\text{pdd}(\alpha)$ is known,
it is in agreement with the values of $\epdd(\alpha)$ displayed in Table \ref{table: epdd} below,
and we suspect that $\text{pdd}(\alpha) = \epdd(\alpha)$ in the remaining cases;
see Subsection \ref{sec: Z^2 finitary} below for additional remarks on (combinatorial) popular difference densities
for matrix patterns in $\Z^2$.

Theorem \ref{Khintchinefiniteindex} provides a sufficient condition on the matrices $M_1$ and $M_2$
to guarantee that $\epdd_{M_1, M_2}(\alpha) \ge \alpha^3$ for $\alpha \in (0,1)$.
We now seek to describe the quantity $\epdd_{M_1, M_2}(\alpha)$ for any pair of $2\times2$ integer matrices $M_1$ and $M_2$.
Table \ref{table: epdd} summarizes ergodic popular difference densities for all 3-point matrix configurations in $\Z^2$.
(For matrices $M_1, M_2$, we let $r(M_1,M_2)$ be a list of the ranks of $M_1$, $M_2$, and $M_2 - M_1$ in decreasing order.)

\begin{table}[h]
	\centering
	{\footnotesize
	\begin{tabular}{|c|c|c|c|} \hline
		$r(M_1, M_2)$ & other conditions & $\epdd_{M_1,M_2}(\alpha)$ & reason \\ \hline
		$(2,2,2)$ & - & $\alpha^3$ & \cite[Theorem 1.10]{ABB} \\ \hline
		$(2,2,1)$ & - & $\alpha^3$ & Theorem \ref{Khintchinefiniteindex} \\ \hline
		$(2,1,1)$ & - & $\alpha^3$ & ``Fubini'' for $\UC$ \cite{BL-cubic} \\ \hline
		$(1,1,1)$ & $[M_1, M_2] = 0$ & $< \alpha^{c\log(1/\alpha)}$ & Behrend-type construction \cite{Beh, BHK} \\ \hline
		$(1,1,1)$ & $[M_1, M_2] \ne 0$, ``row-like'' & $\alpha^3$ & ``Fubini'' for $\UC$ \cite{BL-cubic} \\ \hline
		$(1,1,1)$ & $[M_1, M_2] \ne 0$, ``column-like'' & $\alpha^{4-o(1)}$
		 & \cite[Theorem 1.1]{Chu}, \\
		 & & & \cite[Theorem 1.2]{DS} \\ \hline
	\end{tabular}
	}
	\caption{Ergodic popular difference densities for 3-point matrix patterns in $\Z^2$.}
	\label{table: epdd}
\end{table}

The cases $r(M_1, M_2) = (2,2,2)$ and $r(M_1, M_2) = (2,2,1)$ are covered directly by \cite[Theorem 1.10]{ABB}
and Theorem \ref{Khintchinefiniteindex} respectively.
Indeed, a matrix $M$ has full rank if and only if the subgroup $M(\Z^2) \subseteq \Z^2$ has finite index.
More precisely,
\begin{align*}
	[\Z^2 : M(\Z^2)] = \begin{cases}
		\left| \det(M) \right|, & \text{if}~\det(M) \ne 0; \\
		\infty, & \text{if} \det(M) = 0.
	\end{cases}
\end{align*}
The remaining cases are proved in Section \ref{sec: Z^2}.\\

\subsection{Preliminary remarks on characteristic factors}

In this paper, we approach multiple recurrence problems by determining and utilizing the so-called \emph{characteristic factors}, which are the factors that are responsible for the limiting behavior of the quantity
\begin{align*}
    \mu \left( A \cap T_{\varphi(g)}^{-1}A \cap T_{\psi(g)}^{-1}A \right)
\end{align*}
in ergodic $G$-systems (see Subsection \ref{sec: factors} for a discussion of factors in general and Definition \ref{defn: characteristic factor} for a definition of characteristic factors).
For $\Z$-actions, there are two different approaches to characteristic factors for linear averages, developed independently by Host and Kra \cite{HK} and by Ziegler \cite{Z}, giving rise to factors that coincide.
However, in the context of $G$-actions, where $G$ is an arbitrary (non-finitely generated) countable abelian group, the Host--Kra factors and Ziegler factors may differ; see Subsection \ref{sec: Ziegler} below for more details.

Our work thus leads to the general open question of how the Host--Kra factors are related to the actual characteristic factors of the corresponding multiple ergodic averages (the Ziegler factors). Discerning the relationship between the Host--Kra factors and the Ziegler factors  may lead to a better understanding of the quantities
$$\mu(A\cap T_{\varphi_1(g)}^{-1} A\cap... \cap T_{\varphi_k(g)}^{-1} A),$$
where $\X=(X,\mX,\mu,(T_g)_{g\in G})$ is a $G$-system, $A \in \mX$, and $\varphi_i:G\rightarrow G$ are homomorphisms or, more generally, polynomial maps.

\subsection{Structure of the paper}

The paper is organized as follows.
In Section \ref{notationsconventions}, we introduce notation and conventions that we use throughout the paper.

Proofs of the main results appear in Sections \ref{sec: finite index char factor}--\ref{proof}.
First, in Section \ref{sec: finite index char factor}, we establish characteristic factors for the multiple ergodic averages
\begin{align*}
    \UC_{g \in G} T_{\varphi(g)}f_1 \cdot T_{\psi(g)}f_2
\end{align*}
when $(\psi-\varphi)(G)$ has finite index in $G$ and prove Theorem \ref{Khintchinefiniteindex}.
Then, in Section \ref{Extensions}, we use an extension trick to simplify the characteristic factors, and in Section \ref{sec: limit formula} prove a new limit formula for the extension system, leading to a proof of Theorem \ref{Khintchineab}.
Finally, we prove Theorem \ref{counterexample} in Section \ref{proof}.

The final two sections contain applications of the main results.
Using Theorem \ref{Khintchinefiniteindex} together with additional tools from \cite{ABB, BHK, BL-cubic, Chu, DS}, we compute ergodic popular difference densities for three-point matrix patterns in $\Z^2$.
In Section \ref{sec: semigroups}, we extend the main results (Theorems \ref{Khintchinefiniteindex} and \ref{Khintchineab}) to the setting of cancellative abelian semigroups.

\section{Preliminaries}\label{notationsconventions}

The goal of this section is to introduce some notations and objects that will play an important role in this paper. Throughout this section we let $G$ denote an arbitrary countable abelian group and $\X=(X,\mX,\mu, (T_g)_{g\in G})$ a $G$-system.

\subsection{Uniform Ces\`{a}ro limits}

The large intersection property of a family $\{\varphi_1, \dots, \varphi_k\}$ is related to the limit behavior of the multiple ergodic averages
\begin{align} \label{eq: mult erg avg along Phi}
    \frac{1}{|\Phi_N|} \sum_{g \in \Phi_N} \prod_{i=1}^k{T_{\varphi_i(g)}f_i},
\end{align}
where $(\Phi_N)_{N \in \N}$ is a F{\o}lner sequence\footnote{A sequence $(\Phi_N)_{N\in\N}$ of finite subsets of $G$ is a \emph{F{\o}lner sequence} if, for any $x \in G$, $\frac{|(\Phi_N+x)\triangle\Phi_N|}{|\Phi_N|}\to0$ as $N \to \infty$.} in $G$ and $f_1, \dots, f_k \in L^{\infty}(\mu)$.
By \cite{austin} and \cite{zorin}, the quantity \eqref{eq: mult erg avg along Phi} converges in $L^2(\mu)$ as $N \to \infty$, and the limit is independent of the choice of F{\o}lner sequence $(\Phi_N)_{N \in \N}$.
For more concise notation, we define the \emph{uniform Ces\`{a}ro limit} $x = \UC_{g \in G}{x_g}$ if $\frac{1}{|\Phi_N|} \sum_{g \in \Phi_N}{x_g} \to x$ for every F{\o}lner sequence $(\Phi_N)_{N\in\N}$ in $G$.

One crucial tool for handling uniform Ces\`{a}ro limits is the following version of the van der Corput differencing trick:

\begin{lem} [van der Corput Lemma, cf. \cite{ABB}, Lemma 2.2] \label{vdc}
	Let $\mathcal{H}$ be a Hilbert space and $G$ an amenable group. Let $(u_g)_{g\in G}$ be a bounded sequence in $\mathcal{H}$. If $\UC_{g\in G} \innprod{u_{g+h}}{u_g}$ exists for every $h\in G$, and $$\UC_{h\in G} \UC_{g\in G} \innprod{u_{g+h}}{u_g}=0$$ then, $$\UC_{g\in G} u_g=0$$
		strongly.
\end{lem}

Another useful tool for computing uniform Ces\`{a}ro limits is the following ``Fubini'' trick, which we use extensively in Section \ref{sec: Z^2}:

\begin{lem}[\cite{BL-cubic}, special case of Lemma 1.1] \label{lem: fubini}
    Let $G$ and $H$ be countable discrete amenable groups, and let $(v_{h,g})_{(h,g) \in H \times G}$ be a bounded sequence.
    Suppose
    $$\UC_{(h,g) \in H \times G}{v_{h,g}}$$
    exists,
    and for every $g \in G$,
    $$\UC_{h \in H}{v_{h,g}}$$
    exists.
    Then
    \begin{align*}
        \UC_{g \in G} \UC_{h \in H} v_{h,g} = \UC_{(h,g) \in H \times G} v_{h,g}.
    \end{align*}
\end{lem}

\subsection{Factors} \label{sec: factors}

A \emph{factor} of $\X$ is a $G$-system $\textbf{Y}=(Y,\mY,\nu,(S_g)_{g\in G})$ together with a measurable map $\pi:X\to Y$ such that $\pi_*\mu = \nu$ and $\pi\circ T_g = S_g\circ \pi$ for all $g\in G$.
There is a natural one-to-one correspondence between factors and $(T_g)_{g \in G}$-invariant sub-$\sigma$-algebras of $\mX$.
Throughout the paper, we freely move between the system $\Y$ and the $\sigma$-algebra $\pi^{-1}(\mY)$ and refer to both of them as factors of $\X$.
Given $f\in L^2(\mu)$ we denote by $E(f|\mY)$ the conditional expectation of $f$ with respect to the $\sigma$-algebra $\pi^{-1}(\mY)$. We say that $f$ is \emph{measurable with respect to $\mY$} if $f=E(f|\mY)$.

\subsection{Factor of invariant sets}

Let $\X=(X,\mX,\mu, (T_g)_{g\in G})$ be a $G$-system. We write $\mathcal{I}_G(X)$ for the sub-$\sigma$-algebra of $G$-invariant sets. We say that $\X$ is \emph{ergodic} if $\mathcal{I}_G(X)$ is the $\sigma$-algebra comprised of null and co-null subsets of $(X,\mX,\mu)$.
For a subgroup $H \le G$, we denote by $\mathcal{I}_H(X)$ the sub-$\sigma$-algebra of $H$-invariant sets.
Given a homomorphism $\varphi:G\rightarrow G$, it is convenient to denote by $\mathcal{I}_\varphi(X)$ the $\sigma$-algebra $\mathcal{I}_{\varphi(G)}(X)$.

\subsection{Host--Kra factors}

The Gowers--Host--Kra seminorms are an ergodic-theoretic version of the uniformity norms introduced by Gowers in \cite{G}.
These seminorms were first introduced by Host and Kra in \cite{HK} in the case of ergodic $\mathbb{Z}$-systems and then generalized by Chu, Frantzikinakis, and Host to $\mathbb{Z}$-systems that are not necessarily ergodic in \cite{CFH}.
In \cite[Appendix A]{BTZ}, a general theory of Gower--Host--Kra seminorms is developed for (not necesssarily ergodic) $G$-systems, where $G$ is an arbitrary countable abelian group.
\begin{defn} \label{GHKseminorms}
Let $G$ be a countable abelian group, and let $\X=(X,\mX,\mu,(T_g)_{g\in G})$ be a $G$-system.
Let $f \in L^\infty (X)$, and let $k\geq 1$ be an integer. The \emph{Gowers--Host--Kra seminorm $\|f\|_{U^k(G)}$ of order $k$ of $f$} is defined recursively by the formula
	\[
	\|f\|_{U^1(G)}:= \|E(\phi|\mathcal{I}_G(X))\|_{L^2}
	\]
	for $k=1$, and
	\[
	\|f\|_{U^k(G)}:=\UC_{g\in G}\left(\|\Delta_gf\|_{U^{k-1}}^{2^{k-1}}\right)^{1/2^k}
	\]
	for $k>1$, where $\Delta_g f(x)=f(T_gx)\cdot\overline{f(x)}$.
\end{defn}
In \cite[Appendix A]{BTZ}, it is shown that the Gower--Host--Kra seminorms for general $G$-systems are indeed seminorms.
Moreover, these seminorms correspond to factors of $\X$.

\begin{prop}[Existence and uniqueness of the universal characteristic factors, cf. \cite{BTZ}, Proposition 1.10] \label{UCF} Let $G$ be a countable abelian group, let $\X$ be a $G$-system, and let $k\geq 0$. There exists a unique (up to isomorphism) factor $\mathbf{Z}^{k}(X)=\left( Z^{k}(X),\mathcal{Z}^{k}(X),\mu_k,(T^{(k)}_g)_{g \in G} \right)$ of $\X$ with the property that for every $f\in L^\infty (X)$, $\|f\|_{U^{k+1}(X)}=0$ if and only if $E(f|\mathcal{Z}^{k}(X))=0$.
\end{prop}
The factors $\mathbf{Z}^k$ guaranteed by Proposition \ref{UCF} are called the \emph{Host--Kra} factors of $\X$.

Let $\X=\left(X,\mX,\mu,(T_g)_{g \in G} \right)$ be a $G$-system. Then, $\mZ^0(X)$ is the same as the $\sigma$-algebra $\mI_G(X)$. In particular if $\X$ is ergodic, then $\mZ^0(X)$ is trivial.
In the literature, $\mathbf{Z}^1(X)$ is often called the \emph{Kronecker} factor, and $\mathbf{Z}^2(X)$ the \emph{Conze--Lesigne} or \emph{quasi-affine} factor of $\X$.

We summarize some basic results about the Host--Kra factors.
\begin{thm}\label{HKfactors}
Let $G$ be a countable abelian group, and let $\X = (X,\mX,\mu,(T_g)_{g\in G})$ be a ergodic $G$-system. Then,
\begin{itemize}
    \item[(i)] For every $k\geq 1$, $\mathcal{Z}^{k-1}(X)\preceq\mathcal{Z}^{k}(X)$. In other words, $\mathbf{Z}^{k-1}(X)$ is a factor of $\mathbf{Z}^k(X)$. In particular, $\mathcal{I}(X)\preceq \mathcal{Z}^k(X)$ for every $k\geq 0$.
    \item[(ii)] The Kronecker factor of $\X$ is isomorphic to a rotation on a compact abelian group. Namely, there exists a homomorphism $\alpha:G\rightarrow Z$ into a compact abelian group $(Z,+)$ such that $\mathbf{Z}^1(X)$ is isomorphic to $(Z, (R_g)_{g \in G})$, where $R_gz = z+\alpha(g)$.
    \item[(iii)] For every $k\geq 1$, if $\X$ is ergodic, then $\mathbf{Z}^k(X)$ is an extension of $\mathbf{Z}^{k-1}(X)$ by a compact abelian group $(H,+)$ and a cocycle $\rho:G\times Z^{k-1}(X)\rightarrow H$. Namely, $Z^k(X)=Z^{k-1}(X)\times H$ as measure spaces, and the action is given by $T^{(k)}_g (z,h) = (T^{(k-1)}_g z, h+\rho(g,z))$.
\end{itemize}
\end{thm}
\begin{proof}
The proof of $(i)$ is an immediate consequence of the monotonicity of the seminorms (see \cite[Corollary 4.4]{HK}). The proof of $(ii)$ and $(iii)$ in the generality of arbitrary countable abelian groups can be found in \cite[Lemma 2.4]{ABB}, and \cite[Theorem 5.3]{ABB}\footnote{The proof in \cite{ABB} applies to a special class of systems, called \emph{normal} systems, and it is shown that every system has a normal extension. However, passing to a normal extension is not necessary; see \cite[Proposition 6.3]{HK}.}.
\end{proof}

\subsection{Joins and meets of factors} \label{notations}

Let $G$ be a countable abelian group, let $\X=(X,\mX,\mu,(T_g)_{g\in G})$ be a $G$-system, and let $\varphi,\psi:G\rightarrow G$ be arbitrary homomorphisms.
\begin{enumerate}
\item{Let $\mZ^1_\varphi(X)$, or just $\mathcal{Z}_\varphi(X)$, denote the $\sigma$-algebra of the Kronecker factor of $X$ with respect to the action of $\varphi(G)$. That is, the $\sigma$-algebra of the factor $\mathbf{Z}^1_\varphi(X)$ obtained by applying Proposition \ref{UCF} for the $G$-system $(X,\mX,\mu,(T_{\varphi(g)})_{g\in G})$ and $k=1$. More generally, let $H$ be a subgroup of $G$ and $k\geq 1$, we let $\mathcal{Z}^k_H(X)$ denote the $\sigma$-algebra of the $k$-th Host--Kra factor $\mathbf{Z}^k_H(X)$ with respect to the action of $H$.}
\item{Let $\mathcal{A}$, $\mathcal{A}_1,\mathcal{A}_2$ be $\sigma$-algebras on $X$. Then, \begin{itemize}
\item{We write $\mathcal{A}\preceq \mX$ if the $\sigma$-algebra $\mathcal{A}$ is a sub-$\sigma$-algebra of $\mX$.}
\item{We let $\mathcal{A}_1\lor \mathcal{A}_2$ denote the \emph{join} of $\mathcal{A}_1$ and $\mathcal{A}_2$, i.e. the $\sigma$-algebra generated by $\mathcal{A}_1$ and $\mathcal{A}_2$ in $\mX$.}
\item{We let $\mathcal{A}_1\land \mathcal{A}_2$ denote the \emph{meet} of $\mathcal{A}_1$ and $\mathcal{A}_2$, i.e. the maximal $\sigma$-algebra which is also a sub $\sigma$-algebra of $\mathcal{A}_1$ and $\mathcal{A}_2$.}
\item{We say that $\mathcal{A}_1$ and $\mathcal{A}_2$ are \emph{$\mu$-independent} if their meet is trivial modulo $\mu$-null sets.}
\item{More generally, we say that $\mathcal{A}_1$ and $\mathcal{A}_2$ are \emph{relatively independent over the $\sigma$-algebra $\mathcal{A}$} if $\mathcal{A}_1\land \mathcal{A}_2\preceq \mathcal{A}$.}
\end{itemize}}
\item{We let $\mathcal{I}_{\varphi,\psi}(X)$ denote the meet of $\mathcal{I}_{\varphi}(X)$ and $\mathcal{I}_{\psi}(X)$ and $\mathcal{Z}_{\varphi,\psi}(X)$ the meet of $\mathcal{Z}_\varphi(X)$ and $\mathcal{Z}_\psi(X)$. We let $\mathbf{Z}_{\varphi,\psi}(X)$ denote the factor of $\X$ which corresponds to the $\sigma$-algebra $\mathcal{Z}_{\varphi,\psi}(X)$.}
\end{enumerate}

The next two lemmas give convenient alternative descriptions of independent and relatively independent $\sigma$-algebras.
These results are classical and can be found, e.g., in \cite[Proposition 1.4]{Zim};
we provide short proofs for the convenience of the reader.

\begin{prop}[Independent $\sigma$-algebras] \label{independent}
Let $\X=(X,\mX,\mu)$ be a probability space. Two $\sigma$-algebras $\mathcal{A}_1$ and $\mathcal{A}_2$ on $X$ are independent if and only if the following equivalent conditions hold:
\begin{enumerate}[(i)]
    \item{Any function $f\in L^\infty(X)$ measurable with respect to $\mathcal{A}_1$ and $\mathcal{A}_2$ simultaneously is a constant $\mu$-almost everywhere.}
    \item{If $f$ is measurable with respect to $\mathcal{A}_1$ and $g$ is measurable with respect to $\mathcal{A}_2$, then $$\int_X f\cdot g~d\mu = \int_X f~d\mu \cdot \int_X g~d\mu.$$}
\end{enumerate}
\end{prop}
\begin{proof}
The first definition of independence above is clearly equivalent to (i).
We prove the equivalence between (i) and (ii). \\

(i)$\Rightarrow$ (ii). 
$$ \int_X f\cdot g ~d\mu  = \int_X E(f|\mathcal{A}_2)\cdot g ~d\mu  = \int_X E(f|\mathcal{A}_2) ~d\mu \cdot \int_X g ~d\mu = \int_X f ~d\mu \cdot \int_X g ~d\mu.$$\\

For (ii)$\Rightarrow$ (i), let $\tilde{f} = f-\int f d\mu$. Then,
$$\|\tilde{f}\|_{L^2(\mu)}^2 = \int |\tilde{f}|^2 ~d\mu = \left|\int_X \tilde{f} ~d\mu\right| ^2=0.$$ We conclude that $f=\int fd\mu$.
\end{proof}

\begin{prop}[Relatively independent $\sigma$-algebras] \label{relativeind}
Let $\X=(X,\mX,\mu)$ be a probability space. Let $\mathcal{A}_1,\mathcal{A}_2$ be two $\sigma$-algebras on $X$ and let $\mathcal{A}$ be a third $\sigma$-algebra such that $\mathcal{A}\preceq \mathcal{A}_1\land \mathcal{A}_2$. Then, $\mathcal{A}_1$ and $\mathcal{A}_2$ are relatively independent with respect to $\mathcal{A}$ if the following equivalent conditions hold:
\begin{enumerate}[(i)]
    \item{Any function $f\in L^\infty(X)$ measurable with respect to $\mathcal{A}_1$ and $\mathcal{A}_2$ simultaneously, is measurable with respect to $\mathcal{A}$.}
    \item{If $f$ is measurable with respect to $\mathcal{A}_1$ and $g$ is measurable with respect to $\mathcal{A}_2$, then $$E(fg|\mathcal{A})=E(f|\mathcal{A})\cdot E(g|\mathcal{A}).$$}
\end{enumerate}
\end{prop}
\begin{proof}
Condition (i) is equivalent to the definition of relative independence above. Therefore it is enough to prove the equivalence of (i) and (ii).\\

(i)$\Rightarrow$(ii). We have $E(fg|\mathcal{A}_1) = f\cdot E(g|\mathcal{A}_1) = f\cdot E(g|\mathcal{A})$ where the last equality follows from $(i)$. Now by taking the conditional expectation over $\mathcal{A}$ we have $$E(fg|\mathcal{A})=E(f|\mathcal{A})\cdot E(g|\mathcal{A}).$$\\

(ii)$\Rightarrow$ (i).
Let $\tilde{f} = f-E(f|A)$. Then $E(|\tilde{f}|^2|A)=E(\tilde{f}|A)^2 = 0$. In particular $\int |\tilde{f}|^2 d\mu = 0$, thus $f=E(f|A)$.
\end{proof}

\subsection{Ziegler factors} \label{sec: Ziegler}

Let $\X=(X,\mX,\mu,T)$ be an invertible ergodic measure preserving system and $f_1,...,f_k\in L^\infty(X)$, $k\geq 0$. The convergence of the multiple ergodic averages 
\begin{align} \label{eq: AP erg avg}
    \frac{1}{N}\sum_{n=0}^{N-1} \prod_{i=0}^k T^{in} f_i
\end{align} in $L^2(\mu)$ for general $k$ was established by Host and Kra \cite{HK} and independently, though somewhat later, by Ziegler \cite{Z}.

Host and Kra proved convergence by showing that the averages \eqref{eq: AP erg avg} are controlled by the Gowers--Host--Kra seminorms defined above.
This reduces the general convergence problem to convergence under the additional assumption that each function $f_i$ is measurable with respect to the Host--Kra factor.

Ziegler, on the other hand, studied the universal (minimal) characteristic factors for the multiple ergodic averages
$$\frac{1}{N}\sum_{n=0}^{N-1} \prod_{i=0}^k T^{a_in} f_i$$ where $a_1,...,a_k\in\mathbb{Z}$ are distinct and non-zero.
These are the minimal factors $\mZ_k(X)$ such that
$$\lim_{N\rightarrow \infty} \frac{1}{N}\sum_{n=0}^{N-1} \prod_{i=0}^k T^{a_in} f_i=0$$
whenever $E(f_i|\mZ_k(X))=0$ for some $i$.\\

In \cite[Appendix A]{Leib}, Leibman proved that, for $\Z$-systems, the factors studied by Host and Kra coincide with the factors studied by Ziegler, thus giving these factors the name \textit{Host--Kra--Ziegler factors}.
Using F\o lner sequences in order to define averages, one can generalize the above to arbitrary countable abelian groups (or even more generally, to amenable groups).
However, in the setting of general abelian groups, Host--Kra factors may no longer coincide with Ziegler factors.
We give a very simple example.
Let $p$ be a prime number and $\mathbb{F}_p$ be the group with $p$ elements.
We denote by $\mathbb{F}_p^\infty$ the direct sum of countably many copies of $\mathbb{F}_p$.
In \cite{BTZ}, it is shown that there are many non-trivial ergodic $\mathbb{F}_p^\infty$-systems with non-trivial Host--Kra factors $\mZ^k(X)$ for any $k \ge 0$.
However, the only characteristic factor for the average
$$\UC_{g \in G} T_g f_1\cdot...\cdot T_{pg} f_p$$ is $\mX$.
Indeed, since $T_{pg}=Id$, the average is non-zero for every $f_p\not =0$, assuming that $f_1=...=f_{p-1}=1$ (say).
To overcome this technicality one may restrict to the case where $k<p$, but the situation is not that simple for arbitrary countable abelian groups, and in general Host--Kra factors may not coincide with the universal characteristic (Ziegler) factors.\\

This phenomenon was not studied previously in the literature, but it plays an important role in this paper. More specifically, we study how the Host--Kra factor $\mZ^1(X)$, which coincides with the classical Kronecker factor, is related to the Ziegler factor $\mZ_1(X)$ for arbitrary countable abelian groups, i.e. the universal characteristic factor for the average
$$\UC_{g \in G} T_g f_1 T_{2g} f_2,$$ where $f_1,f_2\in L^\infty(\mu)$. One of our main tools is a result which asserts, roughly speaking, that by adding eigenfunctions to the system $\X$, one has that the Ziegler factor $\mZ_1(X)$ is generated by the Host--Kra factor $\mZ^1(X)$ and the $\sigma$-algebra of $2G$-invariant functions. We also give an example that illustrates the necessity of adding eigenfunctions to the system (see Example \ref{eigenvalue2}).

\subsection{Seminorms for multiplicative configurations} \label{sec: seminorm combo}

We now give a brief explanation of where our methods come up short of fully answering Questions \ref{q: coprime} and \ref{q: 4-GP}.
As discussed above, our approach to the large intersections property is to study families of seminorms and their corresponding characteristic factors.
However, in the case of Question \ref{q: coprime} and Question \ref{q: 4-GP}, these seminorms have somewhat exotic behavior.

For example, Question \ref{q: coprime} is related to the averages \begin{align} \label{eq: q^n, q^m avg}
    \UC_{q\in\mathbb{Q}_{>0}} f_1(T_{q^n} x)  f_2(T_{q^m}x)
\end{align} for some ergodic $(\mathbb{Q}_{>0}, \cdot)$-system.
An application of the van der Corput lemma (Lemma \ref{vdc}) shows that \eqref{eq: q^n, q^m avg} is equal to zero if
$$\UC_{q\in \mathbb{Q}_{>0}} \left|\int \Delta_{q^m} f_1 \cdot E(\Delta_{q^m}f_2|\mathcal{I}_{q^{n-m}}(X))~d\mu \right|=0.$$
If the action of $T_{q^{n-m}}$, $q\in(\mathbb{Q}_{>0}, \cdot)$, were ergodic (e.g. if $n=m+1$), then the above expression is manageable as we will see in this paper. Presumably, if $n$ and $m$ are coprime, then this expression may also be manageable, but we do not see how.

Question \ref{q: 4-GP} is related to the average 
\begin{equation}\label{123} \UC_{q\in\mathbb{Q}_{>0}} T_qf_1 T_{q^2} f_2 T_{q^3} f_3.
\end{equation} Using the van der Corput lemma, the Cauchy--Schwarz inequality and then the van der Corput lemma again, we see that the average \eqref{123} is zero if
$$ \UC_{q_1\in\mathbb{Q}_{>0}} \left| \UC_{q_2\in\mathbb{Q}_{>0}} \int \Delta_{q_1^2} \Delta_{q_2^3} f_3~d\mu \right|=0$$
If in the expression above we had $q_1^2,q_2^2$, or $q_1^3,q_2^3$, then this expression would be related the Gowers--Host--Kra seminorm of $f_3$ with respect to the action of all squares or cubes of $(\mathbb{Q}_{>0}, \cdot)$. The above quantity is therefore some combination of the two. Again, presumably, the fact that $2$ and $3$ are coprime may be useful to analyse these seminorms. Studying the structure of these new peculiar seminorms is an interesting problem that we do not pursue in this paper.

\section{Theorem \ref{Khintchinefiniteindex}} \label{sec: finite index char factor}

We first give a brief overview of the proof of Theorem \ref{Khintchinefiniteindex}.
Let $\X=(X,\mX,\mu,(T_g)_{g\in G})$ be an ergodic $G$-system, and let $\varphi,\psi:G\rightarrow G$ be arbitrary homomorphisms such that $(\psi-\varphi)(G)$ has finite index in $G$. The key component in the proof of Theorem \ref{Khintchinefiniteindex} is the analysis of the limit behavior of the multiple ergodic averages \begin{equation}\label{avg}\UC_{g\in G} f_1(T_{\varphi(g)}x)\cdot f_2(T_{\psi(g)}x)
\end{equation}
for $f_1,f_2\in L^\infty(X)$. Standard arguments using the van der Corput lemma (Proposition \ref{partialcf1}) show that
\begin{equation}\label{avgZ}
\begin{split}\UC_{g\in G}& f_1(T_{\varphi(g)}x)\cdot f_2(T_{\psi(g)}x) = \\ &\UC_{g\in G} E(f_1|\mathcal{Z}_\varphi(X))(T_{\varphi(g)}(x)) E(f_2|\mathcal{Z}_\psi(X))(T_{\psi(g)}(x))
\end{split}
\end{equation}
where $\mathcal{Z}_\varphi(X)$ and $\mathcal{Z}_\psi(X)$ are the $\sigma$-algebras of the Kronecker factors of $X$ with respect to the actions of $\varphi(G)$ and $\psi(G)$, respectively (see Subsection \ref{notations}).

In Theorem \ref{Khintchinefiniteindex}, we assume furthermore that $\varphi(G)$ has finite index in $G$. In this case, the factor $\mZ_\varphi(X)$ coincides with $\mZ^1_G(X)$, the Kronecker factor of $X$ with respect to the action of $G$ (see Lemma \ref{IZk}). Our main observation is that one can replace $\mZ_\psi(X)$ in (\ref{avgZ}) with a smaller factor. As an illustration, we give the following example:
\begin{example} \label{Example:eigenvalue}
Consider the additive group $G=\bigoplus_{j=1}^\infty\mathbb{Z}/4\mathbb{Z}$. We use $i\in\mathbb{C}$ to denote the square root of $-1$, and for every natural number $n\in\mathbb{N}$, we let $C_n$ denote the group of roots of unity of degree $n$. We define an action of $G$ on $X=\left(\prod_{j\in\mathbb{N}} C_4\right)\times C_2$ by
$$T_g (\textbf{x},y) = \left( (i^{g_j}x_j)_{j\in \mathbb{N}},y\cdot \prod_{j\in\mathbb{N}} (x_j^{2g_j}\cdot i^{g_j^2-g_j})\right)$$ where $\textbf{x}=(x_1,x_2,...)\in \prod_{j\in\mathbb{N}} C_4$ and $g=(g_1,g_2,...)$ is any representation of $g$ in $\bigoplus_{j=1}^\infty\mathbb{Z}/4\Z$. The system $(X, (T_g)_{g \in G})$ is a group extension of its Kronecker factor $Z_G(X) = \prod_{j\in\mathbb{N}} C_4$ by the cocycle
\begin{align*}
   & \sigma : G \times \prod_{j\in\N}{C_4} \to C_2,\\
    \sigma&(g,\textbf{x}) = \prod_{j\in\mathbb{N}} (x_j^{2g_j}\cdot i^{g_j^2-g_j}).
\end{align*}
Let $\psi(g)=2g$. We observe that the function $f(\textbf{x},y)=y$ is orthogonal to $L^2(Z^1(X))$. On the other hand we have
$$T_{2g}f(\textbf{x},y) = \sigma(2g,\textbf{x})\cdot y = \prod_{j\in\mathbb{N}} (x_j^{4g_j}\cdot i^{4g_j^2-2g_j}) \cdot y = \prod_{j\in\mathbb{N}} (-1)^{g_j} y = \prod_{j\in\mathbb{N}} (-1)^{g_j} f(\textbf{x},y).$$ In other words, $f$ is an eigenfunction with respect to the action of $\psi(G)$ on $X$ with eigenvalue $\lambda(2g) =\prod_{j\in\mathbb{N}} (-1)^{g_j}$. Therefore, $f$ is measurable with respect to $\mathcal{Z}_\psi(X)$, and we see that $\mathcal{Z}^1(X)\not = \mathcal{Z}_\psi(X)$. Now let $\varphi(g)=g$. We claim that $f$ does not contribute to (\ref{avg}). Namely, we have that $$\UC_{g\in G} T_g f_1 T_{2g} f = 0$$ for every bounded function $f_1$. Indeed, by (\ref{avgZ}), it is enough to check this equality in the case where $f_1$ is an eigenfunction with respect to the action of $G$. Let $\chi(g)$ be the eigenvalue of $f_1$ we see that 
$$\UC_{g\in G} T_g f_1 T_{2g} f = f_1\cdot f \cdot \UC_{g\in G} \chi(g)\cdot \lambda(2g).$$ The eigenfunctions of $X$ take the form $h(\textbf{x},y) = \prod_{i=1}^n x_i^{l_i}$ for some $n\in\mathbb{N}$ and $l_1,...,l_n\in \{0,1,2,3\}$. Therefore, $g\mapsto \chi(g)\lambda(2g)$ is a non-trivial characters of $G$ and so $$\UC_{g\in G} \chi(g)\lambda(2g)=0.$$
\end{example}

\begin{rem}
    In the example above, the factor $\mathbf{Z}^1(X)$ is isomorphic to $\prod_{j\in\mathbb{N}} C_4$ equipped with the action $T^{(1)}_g x = (i^{g_j}\cdot x_j)_{j\in\mathbb{N}}$, while $\mathbf{Z}^2(X)=\X$. On the other hand, for the $2G$-system $\left( X, (T_g)_{g \in 2G} \right)$, we have $\mathbf{Z}_{2G}^1(X) = \X$.
\end{rem}

Example \ref{Example:eigenvalue} suggests that a $\psi(G)$-eigenfunction contributes to (\ref{avg}) if and only if its eigenvalue coincides with an eigenvalue of the $G$-action.
In practice, we use a result of Frantzikinakis and Host \cite{FranHost} to decompose $f_2$ into a linear combination of eigenfunctions (see Proposition \ref{FranHostdeco}).
However, since the action of $\psi(G)$ may not be ergodic, we have to include in our analysis the case where the $\psi(G)$-eigenvalue, $\lambda(\psi(g))$, is not a constant in $X$, but rather, a $\psi(G)$-invariant function.
We let $\tilde{\mathcal{Z}}_\psi(X)$ be the sub $\sigma$-algebra of $\mathcal{Z}_\psi(X)$ generated by all the $\psi(G)$-eigenfunctions with eigenvalues $\lambda(\psi(\cdot),x):X\rightarrow \widehat G$ that coincide with an eigenvalue with respect to the $G$-action for $\mu$-a.e. $x\in X$.
We show that one can replace $\mathcal{Z}_\psi(X)$ with $\tilde{\mathcal{Z}}_\psi(X)$ in (\ref{avgZ}).
After replacing $\mathcal{Z}_\psi(X)$ by $\tilde{\mathcal{Z}}_\psi(X)$, the remainder of the proof of Theorem \ref{Khintchinefiniteindex} follows by modifying previous arguments used for deducing Khintchine-type recurrence from knowledge of relevant characteristic factors (see, e.g., \cite[Section 8]{ABB}).

\subsection{Characteristic factors}

We start with a definition of characteristic factors (cf. \cite[Section 3]{F&W}).

\begin{defn} \label{defn: characteristic factor}
Let $G$ be a countable abelian group, let $\varphi,\psi:G\rightarrow G$ be arbitrary homomorphisms, and let $X=(X,\mX,\mu, (T_g)_{g\in G})$  be a $G$-system. A factor $\Y=(Y,\mathcal{Y},\nu, (S_g)_{g\in G})$ of $\X$ is called a \emph{partial characteristic factor for the pair $(\varphi,\psi)$ with respect to $\varphi$} if $$\UC_{g\in G} T_{\varphi(g)}f_1 T_{\psi(g)} f_2 = \UC_{g\in G} T_{\varphi(g)} E(f_1|\mathcal{Y}) T_{\psi(g)} f_2$$ for every $f_1,f_2\in L^\infty(X)$. We define a partial characteristic factor with respect to $\psi$ similarly, and say that $\Y$ is a \emph{characteristic factor} if it is a partial characteristic factor with respect to both $\varphi$ and $\psi$, i.e.
$$\UC_{g\in G} T_{\varphi(g)}f_1 T_{\psi(g)} f_2 = \UC_{g\in G} T_{\varphi(g)} E(f_1|\mathcal{Y}) T_{\psi(g)} E(f_2|\mathcal{Y})$$ for every $f_1,f_2\in L^\infty(X)$.
\end{defn}
In other words, a factor of a measure preserving system $\X=(X,\mX,\mu,(T_g)_{g\in G})$  is a characteristic factor for a certain multiple ergodic average, if the study of the limit behavior of the average can be reduce to this factor.
The following easy lemma is related to the well known result of Furstenberg which asserts that a system $\X=(X,\mX,\mu,T)$ is weakly mixing if and only if the Kronecker factor, $\mZ^1(X)$, is trivial.
\begin{lem}\label{weaklymixing}
    Let $\X=(X,\mX,\mu,(T_g)_{g \in G})$ be a $G$-system, let $\varphi:G\to G$ be a homomorphism and let $f\in L^2(X)$. If $E(f|\mZ_\varphi(X))=0$, then for every $h\in L^2(X)$ we have
    $$\UC_{g\in G}\left| \int_X T_{\varphi(g)} f \cdot h~d\mu \right|=0.$$
\end{lem}
\begin{proof}
Assume $E(f\mid\mZ_{\varphi}(X)) = 0$.
Then by Proposition \ref{UCF}, $\|f\|_{U^2(\varphi(G))}=0$.
That is,
$$\UC_{g\in G} \left|\int_X \Delta_{\varphi(g)} f d\mu \right| = 0.$$ Since $\UC_{g\in G} |a_g|=0 \iff \UC_{g\in G} a_g^2 = 0$ for every bounded complex-valued sequence $g\mapsto a_g$, we have
$$\UC_{g\in G} \int_{X\times X} \left(T_{\varphi(g)}\times T_{\varphi(g)}\right) f\otimes f \cdot \overline{f\otimes f}~ d(\mu\times \mu) = 0.$$
The mean ergodic theorem implies that
$$\int_{X^2} E(f\otimes f | \mathcal{I}_{\varphi\times\varphi}(X\times X))\cdot f\otimes f~ d(\mu\times \mu)=0$$ and $E(f\otimes f | \mathcal{I}_{\varphi\times\varphi}(X\times X))=0$. Therefore, for every $h\in L^2(X)$ we have,
$$\UC_{g\in G} \left(\int_X T_{\varphi(g)}f \cdot h~ d\mu\right)^2 = \int_{X^2} E(f\otimes f | \mathcal{I}_{\varphi\times\varphi}(X\times X))\cdot h\otimes h ~d\mu\times \mu = 0$$ which implies that $$\UC_{g\in G} \left|\int_X T_{\varphi(g)}f \cdot h~ d\mu\right|=0$$ as required.
\end{proof}

Using the van der Corput lemma (Lemma \ref{vdc}), we show that $\mZ_\varphi(X)$ and $\mZ_\psi(X)$ are partial characteristic factors for the pair $(\varphi,\psi)$ with respect to $\varphi$ and $\psi$ respectively.
\begin{prop}\label{partialcf1}
Let $\X = (X, \mX,\mu,(T_g)_{g\in G})$ be an ergodic $G$-system. Let $\varphi, \psi : G \rightarrow G$ be homomorphisms such that $(\psi-\varphi)(G)$ has finite index in $G$. Then for any $f_1,f_2\in L^\infty(\mu)$, one has
$$\UC_{g\in G} T_{\varphi(g)} f_1\cdot T_{\psi(g)} f_2 = \UC_{g\in G} T_{\varphi(g)} E(f_1|\mZ_{\varphi}(X)) \cdot T_{\psi(g)} E(f_2|\mZ_{\psi}(X)) 
$$
in $L^2(\mu)$.
\end{prop}
\begin{proof}
We follow the argument of Furstenberg and Weiss \cite{F&W}.
By linearity and symmetry, it is enough to show that $$\UC_{g\in G} T_{\varphi(g)} f_1\cdot T_{\psi(g)} f_2 = 0$$ whenever $E(f_1|\mZ_\varphi(X))=0$.
Dividing through by a constant, we may assume that $\|f_i\|_{\infty} \le 1$ for $i=1,2$.

We use the van der Corput lemma with $u_g = T_{\varphi(g)} f_1\cdot T_{\psi(g)} f_2$. For every $g,h\in G$, we have
\begin{align} \label{eq: vdc expression}
    \left<u_{g+h},u_g\right> = \int_X T_{\varphi(g+h)} f_1 \cdot T_{\psi(g+h)} f_2 \cdot T_{\varphi(g)}\overline{f_1}\cdot T_{\psi(g)}\overline{f_2}~d\mu.
\end{align}
Since the measure $\mu$ is $T_{\varphi(g)}$-invariant, \eqref{eq: vdc expression} is equal to $$\int_X T_{\varphi(h)}f_1\cdot \overline{f_1} \cdot T_{(\psi-\varphi)(g)} \left(T_{\psi(h)}f_1\cdot \overline {f_2} \right)~d\mu.$$
Hence, by the mean ergodic theorem we have
$$\UC_{g\in G}\left<u_{g+h},u_g\right> = \int_X T_{\varphi(h)} f_1\cdot\overline{f_1} \cdot E(T_{\psi(h)}f_2\cdot \overline{f_2} | \mathcal{I}_{\psi-\varphi}(X)).$$
Since $H:=(\psi-\varphi)(G)$ has finite index in $G$ and the action of $G$ on $X$ is ergodic, we can find a partition $X=\bigcup_{i=1}^{l} A_i$ to $H$-invariant sets, where $l$ is at most the index of $H$ in $G$. Since $f_2$ is bounded by $1$, $$\left|\UC_{g\in G}\left<u_{g+h},u_g\right>\right| \leq \sum_{i=1}^k \left|\int_{X} T_{\varphi(h)} f_1\cdot\overline{f_1} \cdot 1_{A_i}~ d\mu \right|.$$
Now, since $E(f_1|Z_{\varphi}(X))=0$, Lemma \ref{weaklymixing} implies that $\UC_{h\in G} \left|\int_X T_{\varphi(h)} f_1\cdot \overline{f_1}\cdot 1_{A_i}d\mu \right|=0$, for every $1\leq i \leq k$. The van der Corput lemma (Lemma \ref{vdc}) then implies that $$\UC_{g\in G} T_{\varphi(g)} f_1\cdot T_{\psi(g)} f_2 = 0,$$ and this completes the proof.
\end{proof}
In \cite[Appendix A]{Leib}, Leibman proved the following result in the special case where $G=\mathbb{Z}$. For the sake of completeness, we give a proof for arbitrary countable abelian $G$ in Appendix \ref{LeibProof}. 
\begin{lem} \label{IZk}
    Let $(X,\mX,\mu, (T_g)_{g\in G})$ be a $G$-system and let $H\leq G$ be a subgroup of finite index. Then for every $k\geq 1$, one has $\mathcal{Z}^k_H(X) = \mZ^k_G(X)$.
\end{lem}
In particular, if $\varphi(G)$ has finite index in $G$, then the factor $\mZ_\varphi(X)$ coincides with $\mZ(X)$.
\begin{cor}
Let $G$ be a countable abelian group, let $X=(X,\mX,\mu,(T_g)_{g\in G})$ be a $G$-system and let $\varphi,\psi:G\rightarrow G$ be arbitrary homomorphisms such that $\varphi(G)$ and $(\psi-\varphi)(G)$ have finite index in $G$. Then, for any bounded functions $f_1,f_2\in L^\infty(X)$, $$\UC_{g\in G} T_{\varphi(g)} f_1\cdot T_{\psi(g)} f_2 = \UC_{g\in G} T_{\varphi(g)} E(f_1|\mZ(X)) \cdot T_{\psi(g)} E(f_2|\mZ_{\psi}(X)) .
$$
\end{cor}
Let $G$ be a countable abelian group and $\X = (X,\mX,\mu,(T_g)_{g\in G})$ be an ergodic $G$-system. By Theorem \ref{HKfactors}(ii), the Kronecker factor of $\X$, $\mathbf{Z}^1(X)$, is isomorphic to an ergodic rotation. Therefore, it is convenient to identify the Kronecker factor with the system $\textbf{Z}=(Z,\alpha)$, where $Z$ is a compact abelian group and $\alpha:G\rightarrow Z$ is a homomorphism such that $T^1_g z = z + \alpha_g$, where $T^1$ is the $G$-action on $Z$. The following corollary of Proposition \ref{partialcf1} will be useful later on in this paper.
\begin{prop} \label{partialcfeta}
	Let $\X = \left(X, \mX, \mu, (T_g)_{g \in G} \right)$ be an ergodic $G$-system
	with Kronecker factor $\textbf{Z}=(Z, \alpha)$.
	Let $\varphi, \psi : G \to G$ be homomorphisms such that $(\psi - \varphi)(G)$ has finite index in $G$. Then for any $f_0, f_1, f_2 \in L^{\infty}(\mu)$ and any continuous function $\eta : Z^2 \to \mathbb{C}$, we have
	\begin{align*}
		\UC_{g \in G}&~{\eta \left( \alpha_{\varphi(g)}, \alpha_{\psi(g)} \right)
		 ~\int_X{f_0 \cdot T_{\varphi(g)}f_1 \cdot T_{\psi(g)}f_2~d\mu}} \\
		 & = \UC_{g \in G}{\eta \left( \alpha_{\varphi(g)}, \alpha_{\psi(g)} \right)
		 ~\int_X{f_0 \cdot T_{\varphi(g)}E(f_1|\mathcal{Z}_\varphi(X)) \cdot T_{\psi(g)}E(f_2|\mathcal{Z}_\psi(X))~d\mu}}.
	\end{align*}
\end{prop}

\begin{proof}
	By the Stone--Weierstrass theorem and linearity, we may assume $\eta(u,v) = \lambda_1(u)\lambda_2(v)$
	for some characters $\lambda_1, \lambda_2 \in \hat{Z}$.
	Let $\pi : X \to Z$ be the factor map, and let $\chi_i := \lambda_i \circ \pi$.
	Note that $T_g\chi_i = \lambda_i(\alpha_g)\chi_i$, so $\chi_i$ is a $G$-eigenfunction with eigenvalue $\lambda_i \circ \alpha$.
	
	Now set
	\begin{align*}
		h_0 & := \overline{\chi_1} \overline{\chi_2} f_0, \\
		h_1 & := \chi_1 f_1, \\
		h_2 & := \chi_2 f_2.
	\end{align*}
	\noindent Since $\chi_1$ and $\chi_2$ are measurable with respect to the Kronecker factor $\mZ(X)$,
	which is a sub-$\sigma$-algebra of $\mZ_\varphi(X)$ and $\mZ_\psi(X)$, we have the identities
	\begin{align*}
		E(h_1|\mZ_\varphi(X)) & = \chi_1 \cdot E({f_1}|{\mZ_\varphi(X)}), \\
		E(h_2|\mZ_\varphi(X)) & = \chi_2 \cdot E({f_2}|{\mZ_\varphi(X)}).
	\end{align*}
	
\noindent	Thus, applying Proposition \ref{partialcf1} for the functions $h_1, h_2$ and integrating against $h_0$, we have
	\begin{align*}
		\UC_{g \in G}&~{\eta \left( \alpha_{\varphi(g)}, \alpha_{\psi(g)} \right)
		 ~\int_X{f_0 \cdot T_{\varphi(g)}f_1 \cdot T_{\psi(g)}f_2~d\mu}} \\
		 & = \UC_{g \in G}{\int_X{h_0 \cdot T_{\varphi(g)}h_1 \cdot T_{\psi(g)}h_2~d\mu}} \\
		 & = \UC_{g \in G}{\int_X{h_0 \cdot T_{\varphi(g)} E(h_1|\mZ_\varphi(X))
		 \cdot T_{\psi(g)} E(h_2|\mZ_\psi(X))~d\mu}} \\
		 & = \UC_{g \in G}{\eta \left( \alpha_{\varphi(g)}, \alpha_{\psi(g)} \right)
		 ~\int_X{f_0 \cdot T_{\varphi(g)} E(f_1|\mZ_\varphi(X)) \cdot T_{\psi(g)} E(f_2|\mZ_\psi(X))~d\mu}}.
	\end{align*}
\end{proof}
In the next section we will study the factor $\mathcal{Z}_\psi(X)$ further.
\subsection{Relative orthonormal basis}
Let $G$ be a countable abelian group, and let $\X=(X,\mX,\mu,(T_g)_{g\in G})$ be a $G$-system. Under the assumption that the system is ergodic, it is well known that the factor $\mZ^1(X)$ admits an orthonormal basis of eigenfunctions. The following example demonstrates that this may fail for non-ergodic systems.
\begin{example}\label{basis}
Let $S^1 = \{z\in \mathbb{C} : |z|=1\}$. Consider $X=S^1\times S^1$ equipped with the Borel $\sigma$-algebra, the Haar probability measure $\mu$, and the measure-preserving transformation $T(x,y) = (x,y\cdot x)$. Any function $f\in L^2(X)$ takes the form $$f(x,y) = \sum_{n,m\in\mathbb{N}} a_{n,m} x^n y^m$$ for some $a_{n,m}\in\mathbb{C}$ with \begin{equation}\label{norm}
    \sum_{n,m\in\mathbb{N}} |a_{n,m}|^2 <\infty.
\end{equation}

Now suppose that there exists some constant $c\in S^1$ such that $Tf(x,y) = c\cdot f(x,y)$ for $\mu$-a.e. $(x,y)\in S^1\times S^1$. By the uniqueness of the Fourier series we deduce that $$a_{n+m,m} = c\cdot a_{n,m}$$ for every $n,m\in\mathbb{N}.$ If $m\not = 0$, this is a contradiction to \eqref{norm} unless $a_{n,m}=0$. We conclude that $f$ is an eigenfunction if and only if it is independent of the $y$ coordinate. In particular $L^2(X)$ is not generated by the eigenfunctions of $X$.

On the other hand, the functions $\{x^n\}_{n\in\mathbb{N}}$ are invariant and therefore measurable with respect to $\mathcal{Z}^1(X)$. Moreover, the functions $\{y^m\}_{n\in\mathbb{N}}$ satisfy $\Delta_n (y^m) = T^n (y^m) \cdot y^{-m} =  x^{n\cdot m}$, which is an invariant function. Hence, $y^m$ is also measurable with respect to $\mathcal{Z}^1(X)$. We thus conclude that $X$ coincides with $\mathcal{Z}^1(X)$.
\end{example}
In order to handle non-ergodic systems, Frantzikinakis and Host \cite{FranHost} came up with the following definition.
\begin{defn}
Let $H$ be a countable abelian group acting on a probability space $(X,\mX,\mu,(T_h)_{h\in H})$. A \textit{relative orthonormal system} is a countable family $(\phi_j)_{j\in\mathbb{N}}$ belonging to $L^2(\mu)$ such that
\begin{enumerate}[(i)]
    \item{$\mathbb{E}(\left|\phi_j\right|\text{ } | \mI_H(X))$ has value $0$ or $1$ $\mu$-a.e. for every $j\in \mathbb{N}$;}
    \item{$\mathbb{E}(\phi_j\overline{\phi_k} | \mI_H(X))=0$ $\mu$-a.e. for all $j,k\in\mathbb{N}$ with $j\not=k$.}\\
    
    The family $(\phi_j)_{j\in\mathbb{N}}$ is also a \emph{relative orthonormal basis} if it also satisfies
    \item{The linear space spanned by the set of functions $$\left\{\phi_j\psi : j \in \N, \psi \in L^{\infty}(\mu)~\text{is $H$-invariant} \right\}$$ is dense in $L^2(\mu)$.}
\end{enumerate}
\end{defn}
We also give a definition of eigenfunctions that applies to non-ergodic systems.
\begin{defn}[$H$-eigenfunctions]
Let $H$ be a countable abelian group and $X=(X,\mX,\mu,(T_h)_{h\in H})$ be an $H$-system. We say that $f:X\rightarrow \mathbb{C}$ is an \emph{$H$-eigenfunction} if there exists an $H$-invariant function $\lambda:X\to \hat{H}$ such that $T_h f(x) = \lambda(x,h) \cdot f(x)$ for all $h\in H$ and $\mu$-a.e. $x\in X$. In this case we also say that $\lambda$ is the \emph{eigenvalue} of $f$.
\end{defn}

Note that under the assumption that the $H$-action is ergodic, this definition coincides with the standard definition of an eigenfunction. Observe moreover that the functions $\{y^m\}_{m\in\mathbb{N}}$ from Example \ref{basis} are eigenfunctions according to this definition.

Frantzikinakis and Host proved the following result:
\begin{thm}[\cite{FranHost}, Theorem 5.2] \label{FranHostdeco}
Let $\X=(X,\mX,\mu,H)$ be an $H$-system. Then $\mZ_H(X)$ admits a relative orthonormal basis of eigenfunctions.
\end{thm}
The proof of Theorem \ref{FranHostdeco} is given for $\mathbb{Z}$-actions in \cite{FranHost}, but the same argument can be easily generalized for arbitrary group actions.
\subsection{Proof of Theorem \ref{Khintchinefiniteindex}} \label{sec: phi, psi proof}
In this subsection, we prove Theorem \ref{Khintchinefiniteindex}. Example \ref{Example:eigenvalue} is a good example to have in mind while reading this section.\\

Let $\X = \left( X, \mX, \mu, (T_g)_{g \in G} \right)$ be an ergodic $G$-system, and let $\bfZ=(Z,\alpha)$ be the Kronecker factor of $\X$. Let $A \in \mX$ and $f = \ind_A$.
We can write \begin{align*}
f_c:=E(f|\mathcal{Z}(X)) = \sum_{i\in\mathbb{N}}a_i\zeta_i
	\intertext{ where $\{\zeta_i\}_{i\in\mathbb{N}}$ is an orthonormal basis of eigenfunctions and $a_i\in\mathbb{C}$. Moreover, using Theorem \ref{FranHostdeco},}
f_\psi := E(f_\psi|\mathcal{Z}_\psi(X)) = \sum_{i\in\mathbb{N}} b_i \xi_i,
\end{align*}
where $\{\xi_i\}_{i\in\mathbb{N}}$ is a relative orthonormal basis of $\psi(G)$-eigenfunctions and $b_i = E(f\cdot\overline{\xi}_i | \mathcal{I}_\psi(X))$ are $\psi(G)$-invariant functions.

Choose $N_1 \in \mathbb{N}$ sufficiently large so that
\begin{align*}
	\norm{2}{f_c - \sum_{i=1}^{N_1}{a_i\zeta_i}} & < \frac{\eps}{8}
	\intertext{and}
	\norm{2}{f_{\psi} - \left( \sum_{i=1}^{N_1}{b_i\xi_i}\right)} & < \frac{\eps}{8}.
\end{align*}

For each $j \in \mathbb{N}$, the function $\xi_j$ is a $\psi(G)$-eigenfunction, so we can write
$\xi_j \left( T_{\psi(g)}x \right) = \mu_j(x, \psi(g)) \xi_j(x)$ for some $\psi(G)$-invariant function $\mu_j : X \to \widehat{\psi(G)}$.
The group $Z$ is compact, so $\hat Z$ is countable and we can write $\widehat{Z} = \bigcup_{n \in \mathbb{N}}{F_n}$,
where $F_1 \subseteq F_2 \subseteq \cdots$ are finite sets. Let
\begin{align*}
	C_n := \left\{ g \mapsto \chi_1(\alpha_{\varphi(g)}) \chi_2(\alpha_{\psi(g)})
	 : \chi_1, \chi_2 \in F_n \right\},
\end{align*}
and let $C = \bigcup_{n \in \mathbb{N}}{C_n}.$
Finally, let
\begin{align*}
	E_{j,n} := \left\{ x \in X : \mu_j(x, \cdot) \in C_n \cup \left( \hat{G} \setminus C \right) \right\}.
\end{align*}
Note that the complement of $E_{j,n}$ consists of all $x\in X$ such that $\mu_j(x,\cdot)$ belongs to a finite set. Since $\mu_j$ is measurable, we conclude that so is the complement of $E_{j,n}$. Hence, $E_{j,n}$ are measurable. Since
$\bigcup_{n=1}^{\infty}{E_{j,n}} = X$ for every $j \in \mathbb{N}$, there exists sufficiently large $N_2 \in \mathbb{N}$ such that
\begin{align*}
	\left( \int_{X \setminus E_{j,N_2}}{|b_j \xi_j|^2~d\mu} \right)^{1/2} < \frac{\eps}{16N_1}.
\end{align*}
for $j = 1, \dots, N_1$.
Then, let $N \ge \max\{N_1,N_2\}$ such that:
if $T_g\zeta_i = \chi(\alpha_g)\zeta_i$ for some $i = 1, \dots, N_1$, then $\chi \in F_N$.

Now let $B_0 \in Z$ be a small neighborhood of $0$ in $Z$ such that if $z \in B_0$ and $\chi \in F_N$, then
\begin{align*}
	|\chi(z) - 1| < \frac{\eps}{16N}.
\end{align*}
\noindent Let $\eta_0 : Z \to [0, \infty)$ be a continuous function supported on $B_0$ normalized so that
\begin{align*}
	\UC_{g \in G}{\eta_0(\alpha_{\varphi(g)}) \eta_0(\alpha_{\psi(g)})} = 1.
\end{align*}
Put $\eta(u,v) := \eta_0(u) \eta_0(v)$.
Then by Proposition \ref{partialcfeta}, we have
\begin{align*}
	\UC_{g \in G}&~{\eta \left( \alpha_{\varphi(g)}, \alpha_{\psi(g)} \right)
	 ~\mu \left( A \cap T_{\varphi(g)}^{-1}A \cap T_{\psi(g)}^{-1}A \right)} \\
	 & = \UC_{g \in G}{\eta \left( \alpha_{\varphi(g)}, \alpha_{\psi(g)} \right)
	 ~\int_X{f \cdot T_{\varphi(g)}f_c \cdot T_{\psi(g)}f_{\psi}~d\mu}} \\
	 & = \int_X{f \cdot \UC_{g \in G}{\eta_0(\alpha_{\varphi(g)}) T_{\varphi(g)} f_c
	 \cdot \eta_0(\alpha_{\psi(g)}) T_{\psi(g)} f_{\psi}}~d\mu}.
\end{align*}
From the definition of $B_0$, if $\alpha_{\varphi(g)} \in B_0$,
then $\norm{\infty}{T_{\varphi(g)}\zeta_i - \zeta_i} < \frac{\eps}{16N}$ for $i = 1, \dots, N_1$.
Hence, for every $g \in G$, since $\eta_0$ is supported on $B_0$, we have
\begin{align*}
	\norm{2}{\eta_0(\alpha_{\varphi(g)}) T_{\varphi(g)} f_c - \eta_0(\alpha_{\varphi(g)})f_c}
	 \le &~\norm{2}{\eta_0(\alpha_{\varphi(g)}) \left( T_{\varphi(g)} f_c - \sum_{i=1}^{N_1}{a_i T_{\varphi(g)} \zeta_i} \right)} \\
	 & + \norm{2}{\eta_0(\alpha_{\varphi(g)}) \left( \sum_{i=1}^{N_1}{a_i T_{\varphi(g)} \zeta_i}
	  - \sum_{i=1}^{N_1}{a_i\zeta_i} \right)} \\
	 & + \norm{2}{\eta_0(\alpha_{\varphi(g)}) \left( \sum_{i=1}^{N_1}{a_i\zeta_i} - f_c \right)} \\
	 \le &~\eta_0(\alpha_{\varphi(g)}) \left( \norm{2}{f_c - \sum_{i=1}^{N_1}{a_i\zeta_i}}
	 + N_1 \frac{\eps}{16N} + \norm{2}{f_c - \sum_{i=1}^{N_1}{a_i\zeta_i}} \right) \\
	 < &~\eta_0(\alpha_{\varphi(g)}) \left( \frac{\eps}{8} + \frac{\eps}{16} + \frac{\eps}{8} \right)
	 = \frac{5\eps}{16} \eta_0(\alpha_{\varphi(g)}).
\end{align*}
Therefore,
\begin{align*}
	\left| \int_X{f \cdot \eta_0(\alpha_{\varphi(g)}) T_{\varphi(g)} f_c} \right.
	 & \left. {\cdot~\eta_0(\alpha_{\psi(g)}) T_{\psi(g)} f_{\psi}~d\mu}
	 - \int_X{f_c \cdot f \cdot \eta \left( \alpha_{\varphi(g)}, \alpha_{\psi(g)} \right) T_{\psi(g)} f_{\psi}~d\mu} \right| \\
	 & = \left| \int_X{f \cdot \eta_0(\alpha_{\psi(g)}) T_{\psi(g)}f_{\psi} \cdot \left(
	 \eta_0(\alpha_{\varphi(g)}) T_{\varphi(g)} f_c - \eta_0(\alpha_{\varphi(g)})f_c \right)~d\mu} \right| \\
	 & \le \eta_0(\alpha_{\psi(g)}) \norm{1}{\eta_0(\alpha_{\varphi(g)}) T_{\varphi(g)} f_c - \eta_0(\alpha_{\varphi(g)})f_c} \\
	 & < \frac{5\eps}{16} \eta \left( \alpha_{\varphi(g)}, \alpha_{\psi(g)} \right).
\end{align*}
Taking a Ces\`{a}ro average, we have the inequality
\begin{align} \label{eq: compact almost-inv}
	\UC_{g \in G}&~{\eta \left( \alpha_{\varphi(g)}, \alpha_{\psi(g)} \right)
	 ~\mu \left( A \cap T_{\varphi(g)}^{-1}A \cap T_{\psi(g)}^{-1}A \right)} \nonumber \\
	 & > \int_X{f_c \cdot f \cdot \UC_{g \in G}{
	 \eta \left( \alpha_{\varphi(g)}, \alpha_{\psi(g)} \right)T_{\psi(g)} f_{\psi}}~d\mu} - \frac{5\eps}{16}.
\end{align}
Now we estimate the average
\begin{align*}
	\UC_{g \in G}{\eta \left( \alpha_{\varphi(g)}, \alpha_{\psi(g)} \right)T_{\psi(g)} f_{\psi}}.
\end{align*}
First, for each $i = 1, \dots, N_1$, we have

\begin{align*}
	\norm{\infty}{\eta_0(\alpha_{\psi(g)}) \left( T_{\psi(g)}(b_i\xi_i) - b_i\xi_i \right)}
	 = \norm{\infty}{b_i \cdot \eta_0(\alpha_{\psi(g)}) \left( T_{\psi(g)}\xi_i - \xi_i \right)}
	 < \frac{\eps}{16N} \eta_0(\alpha_{\psi(g)}).
\end{align*}
Next, let $1 \le j \le N_1$.
Write $T_{\psi(g)}(b_j\xi_j) = b_j \mu_j(x, \psi(g)) \psi_j$. If $\mu_j(x, \cdot) \notin C$,
then for any $\chi_1, \chi_2 \in \hat{Z}$, the character
$g \mapsto \chi_1(\alpha_{\varphi(g)}) \chi_2(\alpha_{\psi(g)}) \mu_j(x, \psi(g))$ is nontrivial, so
\begin{align*}
	\UC_{g \in G}{\chi_1(\alpha_{\varphi(g)}) \chi_2(\alpha_{\psi(g)}) \mu_j(x, \psi(g))} = 0.
\end{align*}
\noindent Hence, by the Stone--Weierstrass theorem,
\begin{align*}
	\UC_{g \in G}{\eta \left( \alpha_{\varphi(g)}, \alpha_{\psi(g)} \right) \mu_j(x, \psi(g))} = 0.
\end{align*}
\noindent Therefore,
\begin{align*}
	\UC_{g \in G}{\eta \left( \alpha_{\varphi(g)}, \alpha_{\psi(g)} \right)T_{\psi(g)} f_{\psi}}
	 = \UC_{g \in G}{\eta \left( \alpha_{\varphi(g)}, \alpha_{\psi(g)} \right)T_{\psi(g)} \tilde{f}_{\psi}},
\end{align*}
\noindent where $\tilde{f}_{\psi} = E(f|\tilde{\mathcal{Z}}_\psi(X))$
and $\tilde{\mZ}_\psi(X)$ is the factor generated by $\psi(G)$-eigenfunctions whose eigenvalues come from $C$.
Note that
\begin{align*}
	\tilde{f}_{\psi} = \sum_{i \in \mathbb{N}}b_i\tilde{\xi}_i ,
\end{align*}
\noindent where
\begin{align*}
	\tilde{\xi}_j(x) = \begin{cases}
		\xi_j(x), & \mu_j(x, \cdot) \in C; \\
		0, & \mu_j(x, \cdot) \notin C.
	\end{cases}
\end{align*}
We note that since $C$ is at most countable, $\tilde{\chi}_j$ is measurable. Moreover,
\begin{align*}
	\tilde{f}_{\psi} -  \sum_{i=1}^{N_1}b_i\tilde{\xi}_i
	 = E(f - \sum_{i=1}^{N_1} b_i \xi_i |\tilde{\mZ}_{\psi}(X)),
\end{align*}
so
\begin{align*}
	\norm{2}{\tilde{f}_{\psi} -  \sum_{i=1}^{N_1}{b_i\tilde{\xi}_i}} < \frac{\eps}{8}.
\end{align*}

If $x \in E_{j,N}$, then we must have $\mu_j(x, \cdot) \in C_N$.
That is, $\mu_j(x, \psi(g)) = \chi_1(\alpha_{\varphi(g)}) \chi_2(\alpha_{\psi(g)})$
for some $\chi_1, \chi_2 \in F_N$.
Thus,
\begin{align*}
	\left| \eta \left( \alpha_{\varphi(g)}, \alpha_{\psi(g)} \right) \left( \mu_j(x, \psi(g)) - 1 \right) \right|
	 = &~\left| \eta \left( \alpha_{\varphi(g)}, \alpha_{\psi(g)} \right)
	 \left( \chi_1(\alpha_{\varphi(g)}) \chi_2(\alpha_{\psi(g)}) - 1 \right) \right| \\
	 = &~\left| \eta \left( \alpha_{\varphi(g)}, \alpha_{\psi(g)} \right)
	 \left( \chi_1(\alpha_{\varphi(g)}) \chi_2(\alpha_{\psi(g)}) - \chi_2(\alpha_{\psi(g)}) \right) \right| \\
	 & + \left| \eta \left( \alpha_{\varphi(g)}, \alpha_{\psi(g)} \right)
	 \left( \chi_2(\alpha_{\psi(g)}) - 1 \right) \right| \\
	 < &~\eta \left( \alpha_{\varphi(g)}, \alpha_{\psi(g)} \right) \left( \frac{\eps}{16N} + \frac{\eps}{16N} \right)
	 = \frac{\eps}{8N} \eta \left( \alpha_{\varphi(g)}, \alpha_{\psi(g)} \right).
\end{align*}
Therefore,
\begin{align*}
	\left\| \eta \left( \alpha_{\varphi(g)}, \alpha_{\psi(g)} \right) \right.
	 & \left. \left( T_{\psi(g)}(b_j\tilde{\xi}_j) - b_j\tilde{\xi}_j \right) \right\|_2^2 \\
	 = &~\int_X{\left| \eta \left( \alpha_{\varphi(g)}, \alpha_{\psi(g)} \right)
	 \left( T_{\psi(g)}(b_j\tilde{\xi}_j) - b_j\tilde{\xi}_j \right) \right|^2~d\mu} \\
	 = &~\int_X{\left| b_j(x) \tilde{\xi}_j(x) \right|^2
	 \left| \eta \left( \alpha_{\varphi(g)}, \alpha_{\psi(g)} \right) \left( \mu_j(x, \psi(g)) - 1 \right) \right|^2~d\mu(x)} \\
	 \le &~\eta \left( \alpha_{\varphi(g)}, \alpha_{\psi(g)} \right)^2 \left(
	 \int_{E_{j,N}}{\left( \frac{\eps}{8N} \right)^2 \left| b_j \tilde{\xi}_j \right|^2~d\mu}
	 + 4 \int_{X \setminus E_{j,N}}{\left| b_j \tilde{\xi}_j \right|^2~d\mu} \right) \\
	 \le &~\eta \left( \alpha_{\varphi(g)}, \alpha_{\psi(g)} \right)^2 \left(
	 \left( \frac{\eps}{8N} \right)^2 + 4 \left( \frac{\eps}{16N_1} \right)^2 \right)
	 \le 2 \left( \frac{\eps}{8N_1} \eta \left( \alpha_{\varphi(g)}, \alpha_{\psi(g)} \right) \right)^2
\end{align*}
Putting together our estimates, we have
\begin{align*}
	\left\| \UC_{g \in G} \right.
	 & \left. {\eta \left( \alpha_{\varphi(g)}, \alpha_{\psi(g)} \right)T_{\psi(g)} f_{\psi}} - \tilde{f}_{\psi} \right\|_2 \\
	 = &~\norm{2}{\UC_{g \in G}{\eta \left( \alpha_{\varphi(g)}, \alpha_{\psi(g)} \right)T_{\psi(g)} \tilde{f}_{\psi}}
	 - \tilde{f}_{\psi}} \\
	 \le &~\norm{2}{\UC_{g \in G}{\eta \left( \alpha_{\varphi(g)}, \alpha_{\psi(g)} \right)
	  T_{\psi(g)}\tilde{f}_{\psi}
	 - T_{\psi(g)} \sum_{i=1}^{N_1}{b_i\tilde{\xi}_i}}} \\
	 & + \norm{2}{\UC_{g \in G}{\sum_{i=1}^{N_1}{
	 \eta \left( \alpha_{\varphi(g)}, \alpha_{\psi(g)} \right) \left(T_{\psi(g)}(b_i \tilde{\xi}_i) - b_j \tilde{\xi}_j \right)}}} \\
	 & + \norm{2}{ \sum_{i=1}^{N_1}{b_i\tilde{\xi}_i} - \tilde{f}_{\psi}} \\
	 < &~\frac{\eps}{8} + N_1 \frac{\sqrt{2}\eps}{8N_1} + \frac{\eps}{8}
	 \le \frac{(2\sqrt{2} + 5)\eps}{16} < \frac{\eps}{2}.
\end{align*}

Substituting back into \eqref{eq: compact almost-inv}, we have
\begin{align}
	\UC_{g \in G}{\eta \left( \alpha_{\varphi(g)}, \alpha_{\psi(g)} \right)
	 ~\mu \left( A \cap T_{\varphi(g)}^{-1}A \cap T_{\psi(g)}^{-1}A \right)}
	 & > \int_X{f_c \cdot f \cdot \tilde{f}_{\psi}~d\mu} - \frac{13\eps}{16} \nonumber \\
	 & \ge \mu(A)^3 - \frac{13\eps}{16}. \label{eq: mu^3 bound}
\end{align}
Since $\UC_{g \in G}\eta \left( \alpha_{\varphi(g)}, \alpha_{\psi(g)} \right)=1$, it follows that the set
\begin{align*}
	\left\{ g \in G : \mu \left( A \cap T_{\varphi(g)}^{-1}A \cap T_{\psi(g)}^{-1}A \right) > \mu(A)^3 - \eps \right\}
\end{align*}
is syndetic in $G$. If not, there exists a F{\o}lner sequence $(\Phi_N)_{N\in\N}$ such that $\mu(A\cap T_{\varphi(g)}^{-1}A\cap T_{\psi(g)}^{-1}A)\leq \mu(A)^3-\varepsilon$ for every $g\in\bigcup_{N\in\mathbb{N}}\Phi_N$. But then, $$\UC_{g \in G}{\eta \left( \alpha_{\varphi(g)}, \alpha_{\psi(g)} \right)
	 ~\mu \left( A \cap T_{\varphi(g)}^{-1}A \cap T_{\psi(g)}^{-1}A \right)} \leq \mu(A)^3-\varepsilon$$ which contradicts the inequality \eqref{eq: mu^3 bound}.
	 
\section{Extensions}\label{Extensions}

As we have observed in Subsection \ref{sec: phi, psi proof}, the partial characteristic factors obtained in Proposition \ref{partialcf1} are not the minimal characteristic factors. For example, in Subsection \ref{sec: phi, psi proof} we proved that one can replace $\mZ_\psi(X)$ with the smaller factor $\tilde{\mZ}_\psi(X)$. In this section we develop an extension trick that will be used to further simplify the characteristic factors. These results will be useful in the proof of Theorem \ref{Khintchineab}, where $\varphi(G)$ is no longer assumed to have finite index in $G$. In the example below we illustrate our main result in the simpler case where $\varphi(g)=g$, $\psi(g)=2g$. The following example is based on Example \ref{Example:eigenvalue}.
\begin{example}\label{eigenvalue2}
Let $G=\bigoplus_{j=1}^\infty \mathbb{Z}/4\mathbb{Z}$ and let $X=\left(\prod_{j\in\mathbb{N}}C_4\right) \times C_2 \times C_2$, where the action of $g\in G$ on $X$ is given by
\begin{align} \label{eq: G-action}
T_g(\textbf{x},x_\infty, y) = \left( (i^{g_j}x_j)_{j\in \mathbb{N}}, x_{\infty} \cdot \prod_{k=1}^\infty (-1)^{g_k}, y\cdot \prod_{j\in\mathbb{N}} (x_j^{2g_j}\cdot i^{g_j^2-g_j})\right)
\end{align}
for $\textbf{x} = (x_1, x_2, \dots) \in \prod_{j\in\N}{C_4}$, $x_{\infty} \in C_2$, and $y \in C_2$.
Note that for $g = (g_1, g_2, \dots) \in G$, only finitely many of the coordinates $g_j \in \Z/4\Z$ are nonzero, so \eqref{eq: G-action} is well-defined.

As in Example \ref{Example:eigenvalue}, the function $f(\textbf{x},x_\infty,y)=y$ is a $2G$-eigenfunction with eigenvalue $2g\mapsto \prod_{j=1}^\infty (-1)^{g_j}$. However, this time $f$ may have a non-trivial contribution for the average. 
Indeed, if we let $f_1(\textbf{x},x_\infty,y) = x_\infty$, then $f_1$ is a $G$-eigenfunction with eigenvalue $g\mapsto \prod_{k=1}^\infty (-1)^{g_k}$ and
$$\UC_{g\in G} T_g f_1(\textbf{x},x_\infty,y) T_{2g} f(\textbf{x},x_\infty,y) = x_\infty\cdot y$$ is nonzero. Let $\varphi(g)=g$ and $\psi(g)=2g$. The above computation shows that $f$ is measurable with respect to $\tilde{\mathcal{Z}}_\psi$ where $\tilde{\mathcal{Z}}_\psi$ is defined in Subsection \ref{sec: phi, psi proof}. As a result we deduce that $\mathcal{Z}(X)\lor \mathcal{I}_\psi(X)\prec \tilde{\mathcal{Z}}_\psi(X)$ is a strict inclusion.

Consider the homomorphism $\lambda:G\rightarrow S^1$, $\lambda(g)=\prod_{j=1}^\infty i^{g_j}$, and observe that $\lambda(2g)=\prod_{j=1}^\infty (-1)^{g_i}$ is the eigenvalue of $f_2$. We extend $X$ to a new system $\tilde{X}$, where $\lambda$ is an eigenvalue. Let $\tilde{X} = \left(\prod_{j\in\mathbb{N}} C_4\right) \times C_4\times C_2$, and let the action of $g\in G$ on $\tilde{X}$ be given by
$$S_g(\textbf{x},x_\infty,y) = \left( (i^{g_j}x_j)_{j\in \mathbb{N}}, \lambda(g) x_\infty, y\cdot \prod_{j\in\mathbb{N}} (x_j^{2g_j}\cdot i^{g_j^2-g_j})\right)$$
for $\textbf{x} = (x_1, x_2, \dots) \in \prod_{j\in\N}{C_4}$, $x_{\infty} \in C_4$, and $y \in C_2$.
It is easy to see that $\tilde{\X} = (\tilde{X},(S_g)_{g\in G})$ is an extension of $\X$ with respect to the factor map
$\pi(\textbf{x},x_\infty,y) = (\textbf{x},x_\infty^2,y)$.
Observe that now the function $h(\textbf{x},x_\infty,y)=x_\infty$ on $\tilde{X}$ is an eigenfunction with eigenvalue $\lambda$ and we deduce that $h\cdot \overline{f}\circ \pi$ is a $2G$-invariant function on $\tilde{X}$. This means that $\overline{f}\circ \pi$ is measurable with respect to the $\sigma$-algebra $\tilde{Z}(\tilde{X})\lor \mathcal{I}_\psi(\tilde{X})$. In fact, one can show that now we have an equality $\mathcal{Z}(X)\lor \mathcal{I}_\psi(\tilde{X})=\tilde{\mathcal{Z}}_\psi(\tilde{X})$.
\end{example}
The main result in this section is the following theorem.
\begin{thm}\label{extension}
Let $\X=(X,\mX,\mu,(T_g)_{g \in G})$ be an ergodic $G$-system, and let $C$ be a countable subgroup of $\hat G$. Let $\varphi,\psi : G \to G$ be homomorphisms. There exists an ergodic extension $\tilde{\X}$ of $\X$ with the following property: for any $\chi\in C$, there exist $G$-eigenvalues $\lambda,\mu$ of $\tilde{\X}$ such that $\lambda(\varphi(g))=\mu(\psi(g))=\chi(g)$.
\end{thm}
We will use the following elementary group-theoretic lemma, which is a special case of \cite[Theorem 2.1.4]{rudin}.
	\begin{lem} \label{lift} Let $G$ be a countable discrete abelian group, and let $H \le G$ be a subgroup. Then every character $\lambda\in \hat H$ has a lift $\tilde{\lambda}\in\hat G$ such that $\tilde{\lambda}(h)=\lambda(h)$ for every $h\in H$.
	\end{lem}

	The fact that $\tilde{\X}$ in Theorem \ref{extension} is ergodic will be important in our proof.
	In prepartion for proving that $\tilde{\X}$ is ergodic, we need the following defintion.
	\begin{defn}
	Let $(X,G)$ be an ergodic system and $U$ a compact abelian group. A \emph{cocycle} is a measurable map $\rho:G\times X\rightarrow U$ satisfying $\rho(g+g',x) = \rho(g,x)\cdot \rho(g',T_gx)$ for every $g,g'\in G$ and $\mu$-a.e. $x\in X$. Two cocycles $\rho,\rho':G\times X\rightarrow U$ are said to be \emph{cohomologous} if there exists a measurable map $F:X\rightarrow U$ such that $\rho(g,x)\cdot \rho'(g,x)^{-1} = \Delta_g F(x)$ for all $g\in G$ and $\mu$-a.e. $x\in X$. We let $V_\rho$ denote the minimal closed subgroup generated by $\{\rho(g,x) : g\in G, x\in X\}$. The cocycle $\rho$ is said to be \emph{minimal} if it is not cohomologous to any cocycle $\rho'$ with $V_{\rho'}\lneqq V_{\rho}$.
	\end{defn}
	In \cite{Zim}, Zimmer proved that every cocycle is cohomologous to a minimal cocycle and established the following criterion for ergodicity.
	\begin{lem}[\cite{Zim}, Corollary 3.8] \label{minimal}
	Let $\X=\left(X,\mX,\mu,(T_g)_{g \in G}\right)$ be an ergodic $G$-system, $U$ a compact abelian group, and $\rho : G \times X \to U$ a cocycle. Then, $\X\times_\rho U$ is ergodic if and only if $\rho$ is minimal and $U=U_\rho$.
	\end{lem}
	
	We are now set to prove Theorem \ref{extension}.
\begin{proof}[Proof of Theorem \ref{extension}]
Let $\{\chi_i : i\in \mathbb{N}\}$ be an enumerations of the elements in $C$. By Lemma \ref{lift}, we deduce that for every $i\in\mathbb{N}$, there exist homomorphisms $\chi_i^\varphi,\chi_i^\psi:G\rightarrow S^1$ such that $\chi_i^\varphi(\varphi(g)) = \chi_i^\psi(\psi(g))= \chi_i(g)$. Let $I=\mathbb{N}\times\{\varphi,\psi\}$ and let $\tilde{\chi}:G\rightarrow (S^1)^{I}$ be the homomorphism whose $(i,\varphi)$-coordinate is $\chi_i^\varphi$ and $(j,\psi)$-coordinate is $\chi_j^\psi$ for every $i,j\in\mathbb{N}$.
By Zimmer's theory, there exists a minimal cocycle $\rho:G\times X\rightarrow (S^1)^{I}$ which is cohomologous to $\tilde{\chi}$, where the latter is viewed as a $G\times X\rightarrow (S^1)^{I}$ function that is independent on $x\in X$.
This means that there exists a measurable map $F:X\rightarrow (S^1)^{I}$ such that $\rho_g = \tilde{\chi}(g)\cdot \Delta_g F$. Let $V$ be the image of $\rho$, then by Lemma \ref{minimal},  $\tilde{X}=X\times_\rho V$ is ergodic. Now, for every coordinate $t\in I$, consider the projection map $\pi_t:(S^1)^{I}\rightarrow S^1$. By restricting $\pi_t$ to $V$, we get a homomorphism $\tau_t :V\rightarrow S^1$. Then, the function $\phi_{i,\varphi}(x,v) := \tau_{i,\varphi}(v)\cdot \pi_{i,\varphi} F(x)$ is an eigenfunction with eigenvalue $\Delta_g \phi_{i,\varphi}(x,v) = \chi_i^\varphi(g)$ and $\phi_{j,\psi}(x,v)=\tau_{j,\psi}(v)\cdot \pi_{j,\psi} F(x)$ is an eigenfunction with eigenvalue $\Delta_g \phi_{j,\psi}(x,v) = \chi_j^\psi(g)$. This completes the proof. 
\end{proof}

\subsection{Characteristic factors related to Theorem \ref{Khintchineab}}
The goal of this subsection is to prove a stronger version of Proposition \ref{partialcf1} and Proposition \ref{partialcfeta} with smaller characteristic factors. We will use the above extension theorem in order to express these characteristic factors in terms of $\mathcal{Z}_{\varphi,\psi}(X)$ and the invariant $\sigma$-algebras, $\mathcal{I}_\varphi(X)$ and $\mathcal{I}_\psi(X)$. Then, using a result of Tao and Ziegler (see Theorem \ref{TZtheorem} below), we will reduce matters further to studying the Conze--Lesigne factor $\mathcal{Z}^2(X)$ with respect to the action of $G$, which is already well understood for arbitrary countable abelian groups (see \cite{ABB}, \cite{OS2}).\\

We start with a lemma.
\begin{lem}\label{TT}
Let $\X=(X,\mX,\mu, (T_g)_{g\in G})$ be an ergodic $G$-system. Let $\mI_{\varphi\times\psi}(X\times X)$ denote the $\sigma$-algebra of $(T_{\varphi(g)}\times T_{\psi(g)})_{g\in G}$-invariant sets in $X\times X$. Then, $$\mI_{\varphi\times\psi}(X\times X)\preceq \mathcal{Z}_{\varphi}(X)\times \mathcal{Z}_\psi(X).$$
\end{lem}
\begin{proof}
Let $f_1,f_2\in L^\infty(X)$ be arbitrary functions and $f(x,y)=f_1(x)f_2(y)$. Then, by the mean ergodic theorem we have that
$$E(f|\mI_{\varphi\times\psi}(X\times X))(x,y) = \UC_{g\in G} T_{\varphi(g)}f_1(x)\cdot T_{\psi(g)}f_2(y)$$ in $L^2(\mu\times\mu)$. By van der Corput lemma, $E(f|\mI_{\varphi\times\psi}(X\times X))=0$ if
$$\UC_{h\in G}\left| \UC_{g\in G} \int_{X\times X} T_{\varphi(g+h)}f_1(x)\cdot T_{\psi(g+h)}f_2(y) \cdot \overline{T_{\varphi(g)}f_1(x)}\cdot \overline{T_{\psi(g)}f_2(y)} d(\mu\times \mu)(x,y)\right|=0.$$
Since $\varphi(G)\times \psi(G)$ is measure-preserving the above is equal to 
$$\UC_{h\in G}\left(\left|\int_X \Delta_{\varphi(h)}f_1(x) d\mu(x)\right|\right) \left(\left|\int_X \Delta_{\psi(h)}f_2(y) d\mu(y) \right|\right)$$
which by the Cauchy--Schwarz inequality is bounded above by
$$\left(\|f_1\|_{U^2(\varphi(G))}\cdot \|f_2\|_{U^2(\psi(G))}\right)^{1/2}.$$
We deduce that if  $E(f|\mathcal{Z}_\varphi(X)\times \mathcal{Z}_\psi(X))=0$, then $E\left(f|\mathcal{I}_{\varphi\times\psi}(X\times X)\right)=0$. Since linear combinations of functions of the form $f_1\otimes f_2$ with $f_1,f_2\in L^\infty(X)$ are dense in $L^\infty(X\times X)$ we deduce that the same holds for every bounded function on $X\times X$, and this completes the proof. 
\end{proof}
Using Theorem \ref{extension} we can now prove the following useful result.
\begin{lem} \label{productsubalgebra}
    Let $G$ be a countable abelian group, and let $\X=(X,\mX,\mu,(T_g)_{g\in G})$ be an ergodic $G$-system. Suppose that $\varphi,\psi:G\rightarrow G$ are arbitrary homomorphisms such that $(\psi-\varphi)(G)$ has finite index in $G$. Then there exists an ergodic extension $\tilde{X}$ of $X$ such that $$\pi^{-1}(\mathcal{I}_{\varphi\times\psi}(X)) \preceq\left(\mZ(\tilde{X})\lor \mI_\varphi(\tilde{X})\right) \otimes \left(\mZ(\tilde{X})\lor \mI_\psi(\tilde{X})\right).$$ 
\end{lem}
\begin{proof}
Let $\{\zeta_i\}_{i\in\mathbb{N}}$ be a relative orthonormal basis of eigenfunctions for $\mZ_\varphi(X)$ and $\{\xi_i\}_{i\in\mathbb{N}}$ be the same for $\mZ_\psi(X)$. For every $i,j \in \N$, let $\lambda_i:\varphi(G)\times X\rightarrow \mathbb{C}$ and $\mu_j:\psi(G)\times X\rightarrow\mathbb{C}$ denote the eigenvalues of $\zeta_i$ and $\xi_j$ respectively. Our goal is to study the functions $f\in L^\infty(X^2)$ which are $(T_{\varphi(g)}\times T_{\psi(g)})_{g\in G}$-invariant. By Lemma \ref{TT}, we can write any such function as $$f(x,y) = \sum_{i,j\in\mathbb{N}} c_{i,j}(x,y) \zeta_i(x)\overline{\xi_j(y)}$$ where $c_{i,j}$ is a $\varphi(G)\times \psi(G)$-invariant function. Since $f$ is $T_{\varphi(g)}\times T_{\psi(g)}$-invariant we deduce that
$$c_{i,j}(x,y) \lambda_i(\varphi(g),x)\overline{\mu_j(\psi(g),y)} = c_{i,j}(x,y).$$\\

Hypothetically, if $c_{i,j}$ was a constant, then unless it is zero (and then can be removed from the summation), the equation above implies that  $\lambda_i(\varphi(g),\cdot) = \mu_j(\psi(g),\cdot)=\chi(g)$ for some character $\chi\in\hat G$. In this special case we can apply Theorem \ref{extension} in order to find an extension where $\lambda_i$ and $\mu_j$ are eigenvalues. This means that we can express the lift of $\zeta_i\otimes \xi_j$ to $\tilde{X}$ as a product of a tensor product of $G$-eigenfunctions (whose eigenvalues are $\lambda_i$ and $\mu_j$) and a $\varphi(G)\times \psi(G)$-invariant function, which completes the proof in this special case. Below we generalize the above to arbitrary $c_{i,j}$.\\

Let $C_{i,j} = \{(x,y)\in X\times X : c_{i,j}(x,y)\not = 0\}$. Then $\lambda_i(\varphi(g),x)\overline{\mu_j(\psi(g),y)}=1$ for every $(x,y)\in C_{i,j}$ and all $g\in G$. Hence, $g\mapsto \lambda_i(\varphi(g),x)$ and $g\mapsto\mu_j(\psi(g),y)$ are equal to the same character $\chi\in \hat G$ for all $(x,y)\in C_{i,j}$. Now, for every $\chi\in\hat G$ we let $$J_\chi=\{(i,j)\in\mathbb{N}^2 : (\mu\times\mu)(\{(x,y)\in X\times X:\forall g\text{ }\lambda_i(\varphi(g),x)=\mu_j(\psi(g),y)=\chi(g)\} >0\}$$ and set $$C:=\{\chi\in\hat G : J_\chi\not = \emptyset\} \text{ and } J:=\bigcup_{\chi\in C} J_\chi.$$
Our first observation is that 
\begin{equation}\label{linearcomb}
    f(x,y) = \sum_{(i,j)\in J} c_{i,j}(x,y) \zeta_i(x) \xi_j(y).
\end{equation}
Indeed, if $(i,j)\not\in J$, then for every $\chi$, $(i,j)\not\in J_\chi$, but then from the computation above $\mu(C_{i,j})=0$ and $c_{i,j}=0$ for $(\mu\times\mu)$-a.e. $(x,y)\in X\times X$.

\begin{claim}
The set $C$ is at most countable.
\end{claim}
\noindent \emph{Proof of the claim}.
We use the fact that in a probability space there can be at most countably many disjoint sets of positive measure. Assume by contradiction that $C$ is uncountable. Since there are only countably many $(i,j)\in\mathbb{N}^2$, we deduce that there exists some $(i_0,j_0)$ which belongs to $J_\chi$ for all $\chi$ in an uncountable subset of $\hat G$. But since the sets $$\{(x,y)\in X\times X:\forall g \in G,~\lambda_i(\varphi(g),x)=\mu_j(\psi(g),y)=\chi(g)\}$$ are disjoint for different $\chi$'s and of positive measure, we obtain a contradiction.
This proves the claim.\\

Now we return to the proof of the lemma.
Since $C$ is at most countable, we can apply Theorem \ref{extension}.
We see that there exists an ergodic extension $\pi:\tilde{X}\rightarrow X$, such that for every $\chi\in C$, there exist $G$-eigenvalues $\chi^\varphi,\chi^\psi:G\rightarrow S^1$ with $\chi^\varphi(\varphi(g))=\chi(g)$ and $\chi^\psi(\psi(g))=\chi(g)$.
Let $m_\chi^\varphi,m_\chi^\psi :\tilde{X}\rightarrow S^1$ be the corresponding eigenfunctions.
Now fix some $(i,j)\in J$ and let $\chi\in C$ be such that $\lambda_i(\varphi(g),x)=\mu_j(\psi(g),y)=\chi(g)$ whenever $c_{i,j}(x,y)\not=0$.
We deduce that $\left(c_{i,j}\cdot \zeta_i\otimes \xi_j\right)\circ\pi \cdot \overline{m_\chi^\varphi\otimes m_\chi^\psi}$ is a $\varphi(G)\times \psi(G)$-invariant function.
Since $c_{i,j}$ is also $\varphi(G)\times \psi(G)$-invariant, we deduce by equation (\ref{linearcomb}) that $f\circ \pi$ is a linear combination of products of eigenfunctions $m_\chi^\varphi\otimes m_\chi^\psi$ and some $\varphi(G)\times \psi(G)$-invariant functions.
Equivalently, the lift of $f$ to $\tilde{X}\times \tilde{X}$ is measurable with respect to the $\sigma$-algebra $$\left(\mZ^1(\tilde{X})\lor \mI_\varphi(\tilde{X})\right)\otimes\left(\mZ^1(\tilde{X})\lor \mI_\psi(\tilde{X})\right)$$
as required.
\end{proof}

The following result of Tao and Ziegler plays in important role in our work.
\begin{thm}[\cite{TZ}, Theorem 1.19]\label{TZtheorem}
Let $G$ be a countable abelian group, and let $\X=(X,\mX,\mu,(T_g)_{g \in G})$ be a $G$-system. Let $H_1,H_2$ be two subgroups of $G$, and denote by $H_1+H_2$ the subgroup of $G$ generated by $H_1$ and $H_2$. Then for every $d_1,d_2\in\mathbb{N}$, one has
$$\mathcal{Z}_{H_1}^{d_1}(X) \land \mZ^{d_2}_{H_2}(X)\preceq \mZ^{d_1+d_2}_{H_1+H_2}(X).$$
\end{thm}
 In particular, by setting $d_1=d_2=1$ and using Lemma \ref{IZk} we deduce:
 \begin{lem}\label{TZkron}
     Let $G$ be a countable abelian group and $(X,\mX,\mu,(T_g)_{g\in G})$ be a $G$-system, and let $\varphi,\psi:G\rightarrow G$ be homomorphisms such that $(\psi - \varphi)(G)$ has finite index in $G$. Then, $\mathcal{Z}_{\varphi,\psi}(X)\preceq \mathcal{Z}_G^2(X)$.
 \end{lem}
 We combine this with the results in Section \ref{sec: finite index char factor} to deduce the following version of Theorem \ref{partialcf1}.
\begin{thm}\label{strongcf}
Let $G$ be a countable abelian group and $\X = (X, \mX,\mu,(T_g)_{g\in G})$ be an ergodic $G$-system. Suppose that $\varphi, \psi : G \rightarrow G$ are arbitrary homomorphisms such that $(\psi-\varphi)(G)$ has finite index in $G$. Then for any $f_0,f_1,f_2\in L^\infty(\mu)$ there exists an ergodic extension $\pi:(\tilde{X},\tilde{\mu})\rightarrow (X,\mu)$ such that 
\begin{align*}&\UC_{g\in G} \int_{\tilde{X}} \tilde{f}_0\cdot  T_{\varphi(g)} \tilde{f}_1\cdot T_{\psi(g)} \tilde{f}_2~ d\tilde{\mu} =\\ &\UC_{g\in G} \int_{\tilde{X}} \tilde{f}_0\cdot  T_{\varphi(g)}  E(\tilde{f}_1|\mZ_{G}^2(\tilde{X})\lor \mI_\varphi(\tilde{X})) \cdot T_{\psi(g)} E(\tilde{f}_2|\mZ_{G}^2(\tilde{X})\lor \mI_\psi(\tilde{X}))~ d\tilde{\mu}
\end{align*}
in $L^2(\tilde{X})$, where $\tilde{f}_i:=f_i\circ \pi$ denotes the lift of $f_i$ to the extension $\tilde{X}$.
\end{thm}
Recall that the factors $\mZ_{\varphi}(X)$ and $\mZ_\psi(X)$ are relatively independent over $\mZ_{\varphi,\psi}(X)$.
To put this fact to use, we need to introduce a construction known as a fiber product:
\begin{defn}[The fiber product over a factor.]
For $i=1,2$, let $\Y_i=(Y_i, \mathcal{Y}_i,\mu_i,(S^{(i)}_g)_{g\in G})$ be $G$-systems. Suppose that $\Y=(Y,\mathcal{Y},\nu,(S_g)_{g\in G})$ is a common factor and let $\pi_i:Y_i\rightarrow Y$, $i=1,2$ denote the factor maps.
The \emph{fiber product of $\Y_1$ and $\Y_2$ over $\Y$} is the system $\Y_1 \times_{\Y} \Y_2 = \left( Y_1 \times_Y Y_2, \mY_1 \otimes \mY_2, \mu_1 \times_Y \mu_2, (S^{(1)}_g \times S^{(2)}_g)_{g \in G} \right)$, where
$$Y_1\times_Y Y_2 = \{(y_1,y_2)\in Y_1\times Y_2 : \pi_1(y_1)=\pi_2(y_2)\}$$ and
$$\mu_1\times_Y \mu_2 = \int_Y \mu_{1,y}\times \mu_{2,y} d\nu(y),$$
where $$\mu_i = \int_{Y} \mu_{i,y} d\nu(y)$$ is the disintegration of the measure $\mu_i$ over $Y$ for $i=1,2$.
\end{defn}

We will use the following result from \cite{Zim}:
\begin{thm}\label{RIZ}
Let $G$ be a countable abelian group, and let $\X=(X,\mX,\mu,(T_g)_{g\in G})$ be a $G$-system. Let $\textbf{Y}_1=(Y_1, \mathcal{A}_1,\mu_1,(T^{(1)}_g)_{g\in G})$ and $\textbf{Y}_2=(Y_2,\mathcal{A}_2,\mu_2,(T^{(2)}_g)_{g\in G})$ be two factors of $X$ with factor maps $\pi_i:X\rightarrow Y_i$ for $i=1,2$, and let $\textbf{Y}=(Y,\nu)$ be their meet. Then, the $\sigma$-algebra $\mathcal{A}_1\lor\mathcal{A}_2$ corresponds to the fiber product $\Y_1\times_{\Y} \Y_2$.
\end{thm}

\begin{rem}
In particular, Theorem \ref{RIZ} implies that $\Y_1\times_{\Y}\Y_2$ is a factor of $\X$.
We note that Zimmer also proved the other direction, namely that two factors $\mY_1$ and $\mY_2$ are relatively independent over a third factor $\mY$ if and only if the fiber product $\Y_1 \times_{\Y} \Y_2$ is a factor of $\X$; see \cite[Proposition 1.5]{Zim}.
\end{rem}

We also need the following result:
\begin{thm}[cf. \cite{HK}, Proposition 4.6] \label{liftHK}
Let $\pi:(Y,\mathcal{Y},\nu,(S_g)_{g\in G})\rightarrow (X,\mX,\mu,(T_g)_{g\in G})$ be a factor map between $G$-systems and let $k\geq 1$. Then, $\pi^{-1}(\mathcal{Z}^k(X)) = \mathcal{Z}^k(Y)\land \pi^{-1}(\mX)$.
\end{thm}
Host and Kra proved Theorem \ref{liftHK} for $\mathbb{Z}$-actions, but the argument extends easily to arbitrary countable abelian groups.

We now have all the requisite tools to prove Theorem \ref{strongcf}.
\begin{proof}[Proof of Theorem \ref{strongcf}]
By the previous result we see that if $f_0,f_1$ or $f_2$ are orthogonal to functions measurable with respect to the $\sigma$-algebra $\mZ_\varphi(X) \lor \mZ_\psi(X)$, then the averages above are zero. Therefore, by Theorem \ref{RIZ}, the factor $\mathbf{Z}_\varphi(X)\times_{\mathbf{Z}_{\varphi,\psi}(X)}\mathbf{Z}_\psi(X)$ is a characteristic factor. We may therefore assume without loss of generality that $\X=\mathbf{Z}_\varphi(X)\times_{\mathbf{Z}_{\varphi,\psi}(X)}\mathbf{Z}_\psi(X)$. For the sake of simplicity of notations we write $\mu_{\varphi,\psi}$ for the measure $\mu_{Z_\varphi(X)}\times_{Z_{\varphi,\psi}(X)} \mu_{Z_\psi(X)}$ on $Z_{\varphi(X)}\times_{Z_{\varphi,\psi}(X)}Z_\psi(X)$. By linearity it is suffices to prove the theorem in the case where $f_1 = f_1^\varphi\otimes f_1^\psi$ and $f_2 = f_2^\varphi\otimes f_2^\psi$ for some $f_1^\varphi,f_2^\varphi:Z_\varphi(X)\rightarrow\mathbb{C}$ and $f_1^\psi,f_2^\psi : Z_\psi(X)\rightarrow\mathbb{C}$.  Then,
\begin{align} \label{eq: rel ind join}
    \UC_{g\in G} & \int_X f_0 T_{\varphi(g)} f_1\cdot T_{\psi(g)} f_2 ~d\mu \nonumber \\ & =  \UC_{g\in G} \int_{Z_\varphi(X)\times Z_\psi(X)} f_0\cdot T_{\varphi(g)} \left(f_1^\varphi\otimes f_1^\psi\right)\cdot T_{\psi(g)} \left(f_2^\varphi\otimes f_2^\psi\right) ~d\mu_{\varphi,\psi}
\end{align}
By Proposition \ref{partialcf1}, \eqref{eq: rel ind join} is equal to
\begin{align} \label{eq: rel ind join factor}
    \UC_{g\in G} \int_{Z_\varphi(X)\times Z_\psi(X)} f_0(x,y) \cdot T_{\varphi(g)} \left(f_1^\varphi\cdot E(f_1^\psi|\mZ_{\varphi,\psi}(X))\right)(x)\cdot T_{\psi(g)} \left(E(f_2^\varphi|\mZ_{\varphi,\psi}(X))\cdot f_2^\psi\right)(y) d\mu_{\varphi,\psi}(x,y).
\end{align}
Note that we used the fact that $E(h|\mZ_{\varphi,\psi}(X))(x)=E(h|\mZ_{\varphi,\psi}(X))(y)$ for $\mu_{\varphi,\psi}$ a.e. $x,y$. By the mean ergodic theorem, applied to the transformation $T_{\varphi}\times T_{\psi}$, the limit \eqref{eq: rel ind join factor} converges to
$$\int_{Z_\varphi(X)\times Z_\psi(X)} f_0\cdot E\left(\left(f_1^\varphi\cdot E(f_1^\psi|\mZ_{\varphi,\psi}(X)) \otimes E(f_2^\varphi|\mZ_{\varphi,\psi}(X))\cdot f_2^\psi\right) \bigg| \mI_{\varphi\times \psi}(X)\right) d\mu_{\varphi,\psi}
$$
By Lemma \ref{productsubalgebra}, we can find an ergodic extension $\pi:\tilde{X}\rightarrow X$ such that $\pi^{-1}\left(\mI_{\varphi\times \psi}(X)\right)$ is a sub-$\sigma$-algebra of $\left(\mZ(\tilde{X})\lor \mI_\varphi(\tilde{X})\right) \otimes \left(\mZ(\tilde{X})\lor \mI_\psi(\tilde{X})\right)$. Now, by applying the same argument as above with $\tilde{f}_0,\tilde{f}_1$ and $\tilde{f}_2$ instead of $f_0,f_1$ and $f_2$, and using Theorem \ref{liftHK} in order to replace $\pi^{-1}(\mathcal{Z}_{\varphi,\psi}(X))$ with $\mathcal{Z}_{\varphi,\psi}(\tilde{X})$ we deduce that:
\begin{equation}\label{limit}
\begin{split}
   & \UC_{g\in G} \int_{\tilde{X}} \tilde{f}_0 T_{\varphi(g)} \tilde{f}_1\cdot T_{\psi(g)} \tilde{f}_2 ~d\mu = \\
    &\int_{\tilde{X}} \tilde{f}_0\cdot E\left(\left(\tilde{f}_1^\varphi\cdot E(\tilde{f}_1^\psi|\mZ_{\varphi,\psi}(\tilde{X})) \otimes E(\tilde{f}_2^\varphi|\mZ_{\varphi,\psi}(\tilde{X}))\cdot \tilde{f}_2^\psi\right) \bigg| \pi^{-1}\left(\mI_{\varphi\times \psi}(X)\right)\right) d\tilde{\mu}_{\varphi,\psi},
\end{split}   
\end{equation}
where $\tilde{\mu}_{\varphi,\psi}$ is the lift of $\mu_{\varphi,\psi}$ to $\tilde{X}$.

We return to the proof of the theorem. By linearity it is enough to show that if $E(\tilde{f}_1|\mZ_{G}^2(\tilde{X})\lor \mI_\varphi(\tilde{X}))=0$ or $E(\tilde{f}_2|\mZ_{G}^2(\tilde{X})\lor \mI_\psi(\tilde{X}))=0$, then (\ref{limit}) is zero. By symmetry and Lemma \ref{TZkron}, we may assume without loss of generality that $E(\tilde{f}_1|\mZ_{\varphi,\psi}(\tilde{X})\lor \mI_\varphi(\tilde{X}))=0$. Since $\mZ_{\varphi}(\tilde{X}),\mZ_\psi(\tilde{X})$ are relatively independent over $\mZ_{\varphi,\psi}(\tilde{X})$, they are also relatively independent over the larger $\sigma$-algebra $\mZ_{\varphi,\psi}(\tilde{X})\lor \mI_\varphi(\tilde{X})$. We deduce, by Proposition \ref{relativeind}, that \begin{equation} \label{eq1} E(\tilde{f}_1^\varphi|\mZ_{\varphi,\psi}(\tilde{X})\lor \mI_\varphi(\tilde{X}))\cdot E(\tilde{f}_1^\psi|\mZ_{\varphi,\psi}(\tilde{X})\lor \mI_\varphi(\tilde{X}))=0.
\end{equation}
\begin{claim} $E(\tilde{f}_1^\psi|\mZ_{\varphi,\psi}(\tilde{X})\lor \mI_\varphi(\tilde{X})) =E(\tilde{f}_1^\psi|\mZ_{\varphi,\psi}(\tilde{X})) $.
\end{claim}
\noindent \emph{Proof of the claim}.
$\mZ_{\varphi,\psi}(\tilde{X})\lor \mI_\varphi(\tilde{X})$ is a factor of $\mZ_{\varphi}(\tilde{X})$. By Theorem \ref{liftHK}, $\tilde{f}_1^\psi$ is measurable with respect to $\mZ_\psi(\tilde{X})$ and this and $\mZ_{\varphi}(\tilde{X})$ are relatively independent over $\mZ_{\varphi,\psi}(\tilde{X})$, so the claim follows. \\

Equation (\ref{eq1}) and the claim imply that $$\tilde{f}_1^\varphi\cdot E(\tilde{f}_1^\psi|\mZ_{\varphi,\psi}(\tilde{X}))= \left(\tilde{f}_1^\varphi - E(\tilde{f_1}^\varphi |\mZ_{\varphi,\psi}(\tilde{X})\lor \mI_\varphi(\tilde{X}))\right)E(\tilde{f}_1^\psi|Z_{\varphi,\psi}(\tilde{X}))$$ is orthogonal to all functions measurable with respect to $\mZ_{\varphi,\psi}(\tilde{X})\lor \mI_\varphi(\tilde{X})$ and so it is also orthogonal to those measurable with respect to $\mZ(\tilde{X})\lor \mI_\varphi(\tilde{X})$. Since $\pi^{-1}(\mathcal{I}_{\varphi\times \psi}(X))$ is a sub $\sigma$-algebra of $\left(\mZ(\tilde{X})\lor \mI_\varphi(\tilde{X})\right) \otimes \left(\mZ(\tilde{X})\lor \mI_\psi(\tilde{X})\right)$, this implies that \eqref{limit} is equal to zero as required.
\end{proof}
As a corollary we also have the following stronger counterpart of Proposition \ref{partialcfeta}.
\begin{cor}\label{strongcfeta}
In the settings of Theorem \ref{strongcf}. Let $\eta:Z(\tilde{X})\rightarrow\mathbb{C}$ be a continuous function and $f_0,f_1,f_2\in L^\infty(X)$. Let $\alpha_g$ denote the rotation of $g\in G$ on $Z(\tilde{X})$. If $a, b \in \Z$ are coprime, then
$$\UC_{g\in G} \eta(\alpha_{g})\int_{\tilde{X}} \tilde{f}_0\cdot  T_{ag} \tilde{f}_1\cdot T_{bg} \tilde{f}_2 ~d\tilde{\mu} =$$ $$\UC_{g\in G} \eta(\alpha_g)\int_{\tilde{X}} \tilde{f}_0\cdot  T_{ag}  E(\tilde{f}_1|\mathcal{Z}_G^2(\tilde{X})\lor \mI_a(\tilde{X})) \cdot T_{bg} E(\tilde{f}_2|\mZ_G^2(\tilde{X})\lor \mI_b(\tilde{X}))~ d\tilde{\mu}
$$
where $\tilde{f}_i = f_i \circ \pi$ is the lift of $f_i$ to $\tilde{X}$ for $i=0,1,2$. 
\end{cor}
\begin{proof}
    Since $\eta$ is measurable with respect to $\mZ(\tilde{X})$, it is a linear combination of characters. Therefore, it is enough to prove the equality in the special case where $\eta$ itself is a character. Then, since $a$ and $b$ are coprime we can find $t,s\in\mathbb{Z}$ such that $ta+sb=1$. Set $h_0 =\tilde{f}_0\cdot \eta^{-(t+s)}$, $h_1 = \tilde{f}_1\cdot \eta^s$ and $h_2=\tilde{f}_2\cdot \eta^t$. Arguing as in Theorem \ref{strongcf}, we have
   \begin{align} \label{eq: ag, bg limit}
   &\UC_{g\in G} \int_{\tilde{X}} h_0\cdot  T_{ag} h_1\cdot T_{bg} h_2 ~d\tilde{\mu} =\nonumber\\ &\UC_{g\in G} \int_{\tilde{X}} h_0\cdot  T_{ag}  E(h_1|\mathcal{Z}_G^2(\tilde{X}) \lor \mI_a(\tilde{X})) \cdot T_{bg} E(h_2|\mathcal{Z}_G^2(\tilde{X}) \lor \mI_b(\tilde{X}))~ d\tilde{\mu}.
\end{align}
Now since $\eta$ is measurable with respect to $\mZ(\tilde{X})$, it is also measurable with respect to $\mZ_G^2(\tilde{X}) \lor \mI_a(\tilde{X})$ and $\mZ_G^2(\tilde{X}) \lor \mI_b(\tilde{X})$, so the claim follows by rewriting $h_i$ in terms of $\eta$ and $\tilde{f}_i$ on both sides of the equation \eqref{eq: ag, bg limit}.
\end{proof}

\section{A limit formula for $\{ag, bg\}$} \label{sec: limit formula}

Let $G$ be a countable abelian group and $\X=(X,\mX,\mu,(T_g)_{g\in G})$ be an ergodic $G$-system. In this section we restrict ourselves to the homomorphisms $\varphi(g)=ag,\psi(g)=bg$ where $a,b\in\mathbb{Z}$. By Theorem \ref{strongcf}, we see that it is enough to analyse the ergodic average
\begin{align} \label{eq: ag, bg avg}
    \UC_{g\in G} T_{ag} f_1 \cdot T_{bg} f_2
\end{align} in the case where $X$ is a Conze--Lesigne system (i.e. $X=Z^2(X)$).

Under certain assumptions on $a$ and $b$, two different (but related) formulas were obtained previously in \cite{ABB} and in \cite{OS2}
(see Theorems \ref{thm: abb formula} and \ref{thm: s formula} below).
Neither of the previously-obtained formulas is sufficient for our purposes, so we prove a new one in this section.

\subsection{Previous limit formulas}
Assuming all of the subgroups $aG$, $bG$, $(a+b)G$, and $(b-a)G$ have finite index in $G$, a limit formula was obtained in \cite{ABB} for the multiple ergodic averages \eqref{eq: ag, bg avg} by analysing a Mackey group associated to the abelian extension corresponding to the Conze--Lesigne factor. (The relevant terminology is defined in the next subsection.)
For compact groups $Z$ and $H$, let $\mathcal{M}(Z,H)$ denote the space of measruable functions $f : Z \to H$ equipped with the topology of convergence in measure (with respect to the Haar probability measure).

\begin{thm}[\cite{ABB}, Theorem 7.1] \label{thm: abb formula}
Let $G$ be a countable abelian group. Let $a,b\in\mathbb{Z}$ such that $aG$, $bG$, $(a+b)G$, and $(b-a)G$
have finite index in $G$. Let $k_1' = -ab(a+b)$, $k_2'=ab(a+b)$ and $k_3' = -ab(b-a)$. Set $D=\text{gcd}(k_1',k_2',k_3')$ and $k_i=\frac{k_i'}{D}$ for $i=1,2,3$. Let $c_1,c_2,c_3\in\mathbb{Z}$ so that $\sum_{i=1}^3 k_ic_i=1$. Let $\X = \textbf{Z}\times_\sigma H$ be as in Theorem \ref{HKfactors}(iii). There is a functions $\psi:Z\times Z\rightarrow H$ such that $\psi(0,z)=0$ for every $z\in Z$ and $t\mapsto \psi(t,\cdot)$ is a continuous map from $Z$ to $\mathcal{M}(Z,H)$, and for every $f_1,f_2,f_3\in L^\infty(\mu)$, 
$$\UC_{g\in G} f_1(T_{ag}x)f_2(T_{bg}x)f_3(T_{(a+b)g}x) = \int_{Z\times H^2} \prod_{i=1}^3 f_i(z+a_it,h+d_iu+a_i^2v+c_i\psi(t,z)~ du~dv~dt,$$ in $L^2(\mu)$, where $x=(z,h)\in Z\times H$, and $a_1=a,a_2=b,a_3=a+b$. 
\end{thm}

Assuming that $(b-a)$ is even, the last author proved the following result.
\begin{thm}[\cite{OS2}, Corollary 6.2] \label{thm: s formula}
Let $G$ be a countable abelian group. Let $a,b\in\mathbb{Z}$ be such that $(b-a)$ is even and $(b-a)G$ has finite index in $G$. Let $\X=(X,\mX,\mu,(T_g)_{g\in G})$ be an ergodic $G$-system such that $\X=\mathbf{Z}^2(X)$. Then, there exists an ergodic extension $\pi:Y\rightarrow X$ which is isomorphic to a $2$-step nilpotent coset system\footnote{The exact definition is given in \cite{OS2}. We do not use this notion elsewhere in the paper.} and for every $f_1,f_2,f_3\in L^\infty(X)$,
\begin{align*}&\UC_{g\in G} \tilde{f}_1(T_{ag}y\Gamma) f_2(T_{bg}y\Gamma) f_3(T_{(a+b)g}y\Gamma) =\\& \int_{\mG/\Gamma} \int_{\mG_2} \tilde{f}_1(yy_1^ay_2^{\binom{a}{2}}) \tilde{f}_2(yy_1^by_2^{\binom{b}{2}}\Gamma)\tilde{f}_3(yy_1^{a+b\Gamma}y_2^{\binom{a+b}{2}}\Gamma)~ d\mu_{\mathcal{G}_2}(y_2)~d\mu_{\mathcal{G}/\Gamma}(y\Gamma).
\end{align*}
\end{thm}
The above formula fails if $b-a$ is odd; see \cite[Example 6.3]{OS2}.

Observe that in the formulas in Theorems \ref{thm: abb formula} and \ref{thm: s formula}, we can take $f_3\equiv 1$ and get a limit formula for the averages we are interested in.
However, for the sake of our argument we need a limit formula for every $a,b\in\mathbb{Z}$ regardless of the indices of the subgroups $aG$, $bG$, and $(a\pm b)G$ and the parity of $b-a$.
Below we remove the finite index assumptions in Theorem \ref{thm: abb formula}.

\subsection{Mackey group}
Let $G$ be a countable abelian group, and let $\X= (X,\mX,\mu,(T_g)_{g\in G})$ be an ergodic $G$-system. Suppose that $X=Z^2(X)$, then by Theorem \ref{HKfactors} we can write $\X = \textbf{Z} \times_{\sigma} H$, where $\textbf{Z} = (Z, \alpha)$ is the Kronecker factor, $H$ is a compact abelian group, and $\sigma:G\times Z\rightarrow H$ is a cocycle.

We now define a Mackey group associated to the cocycle $\sigma$. Let
\begin{align*}
	W = W(a,b) := \left\{ (z + at, z + bt) : z, t \in Z \right\},
\end{align*}
\noindent and define $S_gw = (w_1 + \alpha_{ag}, w_2 + \alpha_{bg})$ for $g \in G$, $w = (w_1, w_2) \in W$.
Let $\tilde{\sigma}_g(w) := \left( \sigma_{ag}(w_1), \sigma_{bg}(w_2) \right)$.
Then the \emph{Mackey group} $M = M(a,b)$ is the closed subgroup of $H$ with annihilator given by
\begin{align*}
	M^{\perp}
	 := \left\{ \tilde{\chi} \in \hat{H^2} : \tilde{\chi} \circ \tilde{\sigma}~\text{is a coboundary over}~(W, S) \right\}.
\end{align*}
We will show that the Mackey group is a product of subgroups of $H$.
For $c \in \Z$, let $M_c \le H$ be the closed subgroup with annihilator
\begin{align*}
	M_c^{\perp} := \left\{ \chi \in \hat{H}
	 : (g,z) \mapsto \chi \left( \sigma_{cg}(z) \right)~\text{is a coboundary over}~(Z, \alpha) \right\}.
\end{align*}
\begin{prop} \label{prop: Mackey prod}
	Let $a, b \in \Z$ be coprime, and let $M = M(a,b)$ be the Mackey group.
	Then $M = M_a \times M_b$.
\end{prop}

The proof of Proposition \ref{prop: Mackey prod} relies heavily on results from \cite[Section 7]{ABB}, which we restate here for ease of reference.

\subsection{Cocycle identities}

The following result gives a convenient characterization of coboundaries. (Recall that a cocycle $\rho : G \times Z \to S^1$ is a coboundary if $\rho_g = \Delta_g F$ for some measurable function $F : Z \to S^1$.)

\begin{prop}[\cite{ABB}, Proposition 7.12] \label{prop: abb7.12}
	Let $\mathbf{Z}$ be a Kronecker system and $\rho : G \times Z \to S^1$ a cocycle.
	The following are equivalent:
	\begin{enumerate}[(i)]
		\item	$\rho$ is a coboundary;
		\item	for any sequence $(g_n)_{n \in \mathbb{N}}$ in $G$ with $\alpha_{g_n} \to 0$ in $Z$,
		we have $\rho_{g_n}(z) \to 1$ in $L^2(Z)$.
	\end{enumerate}
\end{prop}

The next proposition gives three equivalent characterizations of Conze--Lesigne (or quasi-affine) cocycles.

\begin{prop}[\cite{ABB}, Proposition 7.15] \label{prop: abb7.15}
	Let $\mathbf{Z}$ be an ergodic Kronecker system and $\rho : G \times Z \to S^1$ a cocycle.
	The following are equivalent:
	\begin{enumerate}[(i)]
		\item	for any sequence $(g_n)_{n \in \mathbb{N}}$ in $G$ with $\alpha_{g_n} \to 0$ in $Z$,
		there is a sequence $(\omega_n)_{n \in \mathbb{N}}$ of affine functions such that
		$\omega_n\rho_{g_n}(z) \to 1$ in $L^2(Z)$;
		\item	for every $t \in Z$,
		\begin{align*}
			\frac{\rho_g(z+t)}{\rho_g(z)}
		\end{align*}
		is cohomologous to a character;
		\item	there is a Borel set $A \subseteq Z$ with $m_Z(A) > 0$ such that
		\begin{align*}
			\frac{\rho_g(z+t)}{\rho_g(z)}
		\end{align*}
		is cohomologous to a character for every $t \in A$.
	\end{enumerate}
\end{prop}

\begin{lem}[\cite{ABB}, Lemma 7.19] \label{lem: abb7.19}
	Let $\mathbf{Z}$ be an ergodic Kronecker system and $\rho : G \times Z \to S^1$ a cocycle.
	Suppose $\left( \alpha_{g_n} \right)$ converges (to $0$) in $Z$,
	and $\omega_n(z) = c_n\lambda_n(z)$ are affine functions such that
	$\left( \omega_n\rho_{g_n} \right)$ converges (to $1$) in $L^2(Z)$.
	Then for every $a \in \mathbb{N}$,
	\begin{align*}
		c_n^a \lambda_n \left( \binom{a}{2} \alpha_{g_n} \right) \lambda_n^a(z) \rho_{ag_n}(z)
	\end{align*}
	converges (to $1$) in $L^2(Z)$.
\end{lem}

\begin{lem}[\cite{ABB}, Lemma 7.23] \label{lem: abb7.23}
	Let $\mathbf{Z} \times_{\sigma} H$ be an ergodic Conze--Lesigne $G$-system.
	Suppose $a \in \Z$ and $aG$ has finite index in $G$.
	Then $aH = H$.
\end{lem}

\begin{lem}[\cite{ABB}, Lemma 7.25] \label{lem: abb7.25}
	Let $Z$ be a compact abelian group.
	Let $c_1, c_2 \in S^1$ and $\lambda_1, \lambda_2 \in \hat{Z}$.
	If $\lambda_1 \ne \lambda_2$, then
	\begin{align*}
		\norm{L^2(Z)}{c_1\lambda_1 - c_2\lambda_2} = \sqrt{2}.
	\end{align*}
\end{lem}

\subsection{Proof of Proposition \ref{prop: Mackey prod}}

We will prove Proposition \ref{prop: Mackey prod} via the next three lemmas.
Rather than proving directly that $M = M_a \times M_b$,
we will instead show the dual identity $M^{\perp} = M_a^{\perp} \times M_b^{\perp}$.
First, we show $M_a^{\perp} \times M_b^{\perp} \subseteq M^{\perp}$:
\begin{lem}
	In the setup of Proposition \ref{prop: Mackey prod}, $M_a^{\perp} \times M_b^{\perp} \subseteq M^{\perp}$.
\end{lem}
\begin{proof}
	Let $\chi_1 \in M_a^{\perp}$ and $\chi_2 \in M_b^{\perp}$.
	We want to show $\tilde{\chi} = \chi_1 \otimes \chi_2 \in M^{\perp}$.
	Let $(g_n)_{n \in \mathbb{N}}$ be a sequence in $G$ such that $(\alpha_{ag_n}, \alpha_{bg_n}) \to 0$ in $W$.
	By Proposition \ref{prop: abb7.12}, it suffices to show
	\begin{align} \label{eq: coboundary over W}
		\tilde{\chi} \circ \tilde{\sigma}_{g_n}(w) \to 1
	\end{align}
	\noindent in $L^2(W)$.
	Now, since $a$ and $b$ are coprime, we have $\alpha_{g_n} \to 0$ in $Z$.
	Since $\chi_1 \in M_a^{\perp}$, it follows that
	\begin{align} \label{eq: coboundary 1 over Z}
		\chi_1 \left( \sigma_{ag_n}(z) \right) \to 1
	\end{align}
	\noindent in $L^2(Z)$ by Proposition \ref{prop: abb7.12}.
	Similarly,
	\begin{align} \label{eq: coboundary 2 over Z}
		\chi_2 \left( \sigma_{bg_n}(z) \right) \to 1
	\end{align}
	\noindent in $L^2(Z)$.
	Combining \eqref{eq: coboundary 1 over Z} and \eqref{eq: coboundary 2 over Z}, we have
	\begin{align*}
		\chi_1 \left( \sigma_{ag_n}(z+at) \right) \chi_2 \left( \sigma_{bg_n}(z+bt) \right) \to 1
	\end{align*}
	\noindent in $L^2(Z \times Z)$.
	That is, \eqref{eq: coboundary over W} holds.
\end{proof}
Before establishing the reverse inclusion, $M^{\perp} \subseteq M_a^{\perp} \times M_b^{\perp}$,
we need the following result:
\begin{lem} \label{lem: Mackey torsion}
	In the setup of Proposition \ref{prop: Mackey prod},
	
	\begin{align*}
		M^{\perp} \subseteq \left\{ \chi_1 \otimes \chi_2 \in \hat{H^2} : \chi_1^a = \chi_2^b = 1 \right\}
	\end{align*}
\end{lem}
\begin{proof}
	Let $\tilde{\chi} = \chi_1 \otimes \chi_2 \in M^{\perp}$.
	By the argument in the proof of \cite[Theorem 7.26]{ABB}, we have $\chi_1^a \chi_2^b = \chi_1^{a^2} \chi_2^{b^2} = 1$.
	Therefore,
	
	\begin{align*}
		\chi_1^{a(b-a)} = \chi_1^{ab} \chi_1^{-a^2}
		 = \left( \chi_1^a \chi_2^b \right)^b \left( \chi_1^{a^2} \chi_2^{b^2} \right)^{-1} = 1.
	\end{align*}
	
	\noindent By assumption, $(b-a)G$ has finite index in $G$.
	It follows that $\hat{H}$ does not contain any $(b-a)$-torsion elements (see Lemma \ref{lem: abb7.23}),
	so $\chi_1^a = 1$.
	We immediately deduce $\chi_2^b = \chi_1^{-a} = 1$ as well.
\end{proof}

Now we can complete the proof of Proposition \ref{prop: Mackey prod}:

\begin{lem}
	In the setup of Proposition \ref{prop: Mackey prod}, $M^{\perp} \subseteq M_a^{\perp} \times M_b^{\perp}$.
\end{lem}
\begin{proof}
	Let $\tilde{\chi} = \chi_1 \otimes \chi_2 \in M^{\perp}$.
	We want to show $\chi_1 \in M_a^{\perp}$ and $\chi_2 \in M_b^{\perp}$.
	For notational convenience, let $a_1 = a$ and $a_2 = b$.
	Let $(g_n)_{n \in \mathbb{N}}$ be a sequence in $G$ such that $\alpha_{g_n} \to 0$ in $Z$.
	By Proposition \ref{prop: abb7.12}, it suffices to show
	
	\begin{align} \label{eq: coboundary over Z}
		\chi_i \left( \sigma_{a_ig_n}(z) \right) \to 1
	\end{align}
	
	\noindent in $L^2(Z)$ for $i = 1, 2$.
	
	Now, $(\alpha_{ag_n}, \alpha_{bg_n}) \to 0$ in $W$, so
	
	\begin{align} \label{eq: coboundary over W 2}
		\tilde{\chi} \circ \tilde{\sigma}_{g_n}(w) \to 1
	\end{align}
	
	\noindent in $L^2(W)$ by Proposition \ref{prop: abb7.12}.
	Moreover, since $\chi_i \circ \sigma$ is a Conze--Lesigne cocycle, we have
	\begin{align} \label{eq: QA cocycle}
		c_{i,n} \lambda_{i,n}(z) \chi_i \left( \sigma_{g_n}(z) \right) \to 1
	\end{align}
	
	\noindent in $L^2(Z)$ for some sequences $(c_{i,n})_{n \in \mathbb{N}}$ in $S^1$
	and $\left( \lambda_{i,n} \right)_{n \in \mathbb{N}}$ in $\hat{Z}$ (see Proposition \ref{prop: abb7.15}).
	
	It follows by Lemma \ref{lem: abb7.19} that
	
	\begin{align} \label{eq: binomial conv}
		c_{i,n}^{a_i} \lambda_{i,n}^{\binom{a_i}{2}}\left( \alpha_{g_n} \right) \lambda_{i,n}^a(z)
		 \chi_i \left( \sigma_{a_ig_n}(z) \right) \to 1
	\end{align}
	
	\noindent in $L^2(Z)$.
	On the other hand, by Lemma \ref{lem: Mackey torsion}, we have $\chi_i^{a_i} = 1$,
	so raising \eqref{eq: QA cocycle} to the $a_i$-th power gives
	\begin{align*}
		c_{i,n}^{a_i} \lambda_{i,n}^{a_i}(z) \to 1
	\end{align*}
	
	\noindent in $L^2(Z)$.
	Hence, by Lemma \ref{lem: abb7.25}, $\lambda_{i,n}^{a_i} = 1$ for all sufficiently large $n$, and $c_{i,n}^{a_i} \to 1$.
	Therefore, \eqref{eq: binomial conv} simplifies to
	
	\begin{align} \label{eq: cohom to char}
		d_{i,n} \chi_i \left( \sigma_{a_ig_n}(z) \right) \to 1
	\end{align}
	
	\noindent in $L^2(Z)$, where $d_{i,n} = \lambda_{i,n}^{\binom{a_i}{2}} \left( \alpha_{g_n} \right)$.
	
	The numbers $a$ and $b$ are coprime, so at least one of them is odd.
	Without loss of generality, assume $a$ is odd.
	Then $a$ divides $\binom{a}{2}$, so $\lambda_{1,n}^{\binom{a}{2}} = 1$.
	Hence, $d_{1,n} = 1$ for all large $n$, so \eqref{eq: coboundary over Z}
	follows from \eqref{eq: cohom to char} for $i = 1$.
	It remains to show \eqref{eq: coboundary over Z} holds for $i = 2$.
	
	Combining the identities \eqref{eq: cohom to char} for $i = 1, 2$ and using $d_{1,n} = 1$, we have
	\begin{align*}
		d_{2,n} \chi_1 \left( \sigma_{ag_n}(z+at) \right) \chi_2 \left( \sigma_{bg_n}(z+bt) \right) \to 1
	\end{align*}
	
	\noindent in $L^2(Z \times Z)$.
	That is,
	
	\begin{align*}
		d_{2,n} \tilde{\chi} \circ \tilde{\sigma}_{g_n}(w) \to 1
	\end{align*}
	
	\noindent in $L^2(W)$.
	Comparing with \eqref{eq: coboundary over W 2}, this implies $d_{2,n} \to 1$.
	Therefore, \eqref{eq: coboundary over Z} follows from \eqref{eq: cohom to char} for $i = 2$.
\end{proof}


\subsection{Limit formula}

With the help of Proposition \ref{prop: Mackey prod}, we will now
prove a limit formula for the averages $\UC_{g\in G} T_{ag} f_1 T_{bg} f_2$.
We need to define one more object related to the cocycle $\sigma$ before stating the limit formula.
For a compact space $K$, let $\mathcal{M}(Z, K)$ denote the space of measurable functions $Z \to K$ equipped
with the topology of convergence in measure.

\begin{prop} \label{prop: psi}
	Let $\X = \mathbf{Z} \times_{\sigma} H$ be an ergodic Conze--Lesigne system.
	Let $c \in \Z$.
	There exists a function $\psi_c : Z \times Z \to H/M_c$ such that
	
	\begin{enumerate}[(1)]
		\item	for every $g \in G$,
			
			\begin{align*}
				\psi_c(\alpha_g, z) \equiv \sigma_{cg}(z) \pmod{M_c},
			\end{align*}
			
			\noindent and
		\item	the map $Z \ni t \mapsto \psi_c(t, \cdot) \in \mathcal{M}(Z, H/M_c)$ is continuous.
	\end{enumerate}
\end{prop}

In order to prove Proposition \ref{prop: psi}, we use the following characterization of convergence in measure:

\begin{lem}[\cite{ABB}, 7.28] \label{lem: abb7.28}
	Let $(f_n)_{n \in \mathbb{N}}$ be a sequence of functions in $\mathcal{M}(Z,H)$.
	Then $f_n \to f$ in $\mathcal{M}(Z,H)$ if and only if $\chi \circ f_n \to \chi \circ f$ in $L^2(Z)$
	for every character $\chi \in \hat{H}$.
\end{lem}

\begin{proof}[Proof of Proposition \ref{prop: psi}]
	Given a sequence $(g_n)_{n \in \mathbb{N}}$ in $G$ such that $(\alpha_{g_n})_{n \in \mathbb{N}}$ is convergent in $Z$,
	we want to show that the sequence
	\begin{align*}
		\left( \sigma_{cg_n}(z) \right)_{n \in \mathbb{N}}
	\end{align*}
	converges in $\mathcal{M}(Z, H/M_c)$.
	Equivalently, by Lemma \ref{lem: abb7.28}, we must show that
	\begin{align*}
		\left( \chi \left( \sigma_{cg_n}(z) \right) \right)_{n \in \mathbb{N}}
	\end{align*}
	converges in $L^2(Z)$ for every $\chi \in \hat{H/M_c} = M_c^{\perp}$.
	
	Let $\chi \in M_a^{\perp}$.
	By the definition of $M_c$, the cocycle $\chi \left( \sigma_{cg}(z) \right)$ is a coboundary over $(Z, \alpha)$.
	Hence, by Proposition \ref{prop: abb7.12}, there is a continuous map $t \mapsto \varphi(t, \cdot) \in L^2(Z)$
	such that $\varphi(\alpha_g, z) = \chi \left( \sigma_{cg}(z) \right)$.
	Therefore,
	\begin{align*}
		\chi \left( \sigma_{cg_n}(z) \right) \to \varphi(t,z)
	\end{align*}
	in $L^2(Z)$, where $t = \lim_{n \to \infty}{\alpha_{g_n}} \in Z$.
\end{proof}

By the Kuratowski and Ryll-Nardzewski measurable selection theorem (see \cite[Section 5.2]{srivastava}), there exists a measurable map $\iota_a : H/M_a \to H$ such that $\pi_a(\iota_a(x)) = x$,
where $\pi_a$ is the canonical projection $\pi_a : H \to H/M_a$.
Let $\psi_1 = \iota_a \circ \psi_a$ and $\psi_2 = \iota_b \circ \psi_b$.
We can now state and prove a general limit formula for Conze--Lesigne systems:

\begin{thm} \label{thm: limit}
	Let $\X = \mathbf{Z} \times_{\sigma} H$ be an ergodic Conze--Lesigne system.
	Let $a, b \in Z$.
	Let $M = M(a,b) = M_a \times M_b$.
	Then for any $f_1, f_2 \in L^{\infty}(\mu)$, we have
	
	\begin{align} \label{eq: limit}
		\UC_{g \in G}&{f_1(T_{ag}(z,x)) f_2(T_{bg}(z,x))} \nonumber \\
		 & = \int_{Z \times M_a \times M_b}{f_1(z + at, x + u + \psi_1(t,z)) f_2(z + bt, x + v + \psi_2(t,z))~dt~du~dv}
	\end{align}
	
	\noindent in $L^2(Z \times H)$.
\end{thm}

\begin{rem}
    We have defined the functions $\psi_i$ by lifting $\psi_a$ and $\psi_b$ to the group $H$ from $H/M_a$ and $H/M_b$ respectively.
    If $\psi'_1$ is another functions with $\pi_a(\psi'_1) = \psi_a$, then for any $t, z \in Z$, we have $\psi'_1(t,z) - \psi_1(t,z) \in M_a$.
    Since the Haar measure on $M_a$ is invariant under shifts coming from $M_a$, the expression on the right hand side of \eqref{eq: limit} is unchanged when $\psi_1$ is replaced by $\psi'_1$.
    The same is true for replacing $\psi_2$ by $\psi'_2$, so it does not matter which lifts of $\psi_a$ and $\psi_b$ we choose.
\end{rem}

\begin{proof}
	For notational convenience, let $\psi = (\psi_1, \psi_2) : Z \times Z \to H^2$,
	and let $m_M$ denote the Haar measure on the Mackey group $M = M_a \times M_b$.
	
	It suffices to prove the formula \eqref{eq: limit} for functions of the form $f_i(z,x) = \omega_i(z) \chi_i(x)$
	with $\omega_i \in L^{\infty}(Z)$ and $\chi_i \in \hat{H}$.
	In this case, the right hand side of \eqref{eq: limit} is equal to
	
	\begin{align*}
		\int_Z{\omega_1(z+at) \omega_2(z+bt) \chi_1(x) \chi_2(x) \tilde{\chi}(\psi(t,z))~dt} \int_M{\tilde{\chi}~dm_M},
	\end{align*}
	
	\noindent where $\tilde{\chi} = \chi_1 \otimes \chi_2 \in \hat{H^2}$.
	
	We now consider two cases.
	First, if $\tilde{\chi} \notin M^{\perp}$, then $\int_M{\tilde{\chi}~dm_M} = 0$,
	so the right hand side of \eqref{eq: limit} is equal to zero.
	Moreover, for every $\lambda \in M^{\perp}$ and almost every $z, t \in Z$, we have
	\begin{align*}
		\int_{H^2}{f_1(z+at, x)f_2(z+bt,y)\lambda(x,y)~dx~dy}
		 = \omega_1(z+at)\omega_2(z+bt)~\int_{H^2}{\tilde{\chi}(x,y) \lambda(x,y)} = 0.
	\end{align*}
	\noindent Therefore, the left hand side of \eqref{eq: limit} is also zero (see \cite[Proposition 7.10]{ABB}). \\
	
	Now suppose $\tilde{\chi} \in M^{\perp}$ so that $\int_M{\tilde{\chi}~dm_M} = 1$.
	For $g \in G$ and $(z,x) \in Z \times H$, we can write
	\begin{align*}
		f_1(T_{ag}(z,x)) f_2(T_{bg}(z,x))
		 = \omega_1(z + \alpha_{ag}) \omega_2(z + \alpha_{bg})
		 \chi_1(x) \chi_2(x) \tilde{\chi}(\sigma_{ag}(z), \sigma_{bg}(z)).
	\end{align*}
	\noindent Thus, letting
	\begin{align*}
		\varphi_t(z,x) := \omega_1(z + at) \omega_2(z + bt) \chi_1(x) \chi_2(x) \tilde{\chi}(\psi(t,z)),
	\end{align*}
	\noindent we have
	\begin{align*}
		f_1(T_{ag}(z,x)) f_2(T_{bg}(z,x)) = \varphi_{\alpha_g}(z,x).
	\end{align*}
	\noindent By Proposition \ref{prop: psi}, the map $Z \ni t \mapsto \varphi_t \in L^2(Z \times H)$ is continuous.
	Therefore, for any $\xi \in L^2(Z \times H)$, since the system $(Z, \alpha)$ is uniquely ergodic, we have
	\begin{align*}
		\UC_{g \in G}{\innprod{\varphi_{\alpha_g}}{\xi}} = \int_Z{\innprod{\varphi_t}{\xi}~dt}.
	\end{align*}
	\noindent That is
	\begin{align} \label{eq: weak conv}
		\UC_{g \in G}{\varphi_{\alpha_g}(z,x)} = \int_Z{\varphi_t(z,x)~dt}
	\end{align}
	
	\noindent weakly in $L^2(Z \times H)$.
	By more general results on norm convergence on multiple ergodic averages (see \cite{austin, zorin}),
	it follows that \eqref{eq: weak conv} holds strongly.
	The right hand side of \eqref{eq: limit} is also equal to $\int_Z{\varphi_t(z,x)~dt}$,
	so the formula \eqref{eq: limit} holds when $\tilde{\chi} \in M^{\perp}$.
\end{proof}


\subsection{Proof of Theorem \ref{Khintchineab}}
We first prove the theorem in the special case where $a$ and $b$ are coprime.

Let $f = \ind_A$.
By Theorem \ref{strongcfeta}, there is an extension $\tilde{\X}$ of $\X$ such that
\begin{align*}
	\UC_{g \in G}&~{\eta(\alpha_g)~\int_{\tilde{X}}{\tilde{f} \cdot T_{ag} \tilde{f} \cdot T_{bg} \tilde{f}~d\tilde{\mu}}} \\
	 & = \UC_{g \in G}{\eta(\alpha_g)~\int_{\tilde{X}}{\tilde{f} \cdot
	 T_{ag} E(\tilde{f}|\mZ_{G}^2(\tilde{X})) \vee \mI_a(\tilde{X})) \cdot
	 T_{ag} E({\tilde{f}}|{\mZ_{G}^2(\tilde{X}) \vee \mI_b(\tilde{X})})~d\tilde{\mu}}},
\end{align*}

\noindent where $\tilde{f}$ is the lift of $f$ to $\tilde{X}$.
For notational convenience, let $\tilde{f}_a := E(\tilde{f}|\mZ_{G}^2(\tilde{X})) \vee \mI_a(\tilde{X}))$
and $\tilde{f}_b := E({\tilde{f}}|{\mZ_{G}^2(\tilde{X}) \vee \mI_b(\tilde{X})})$.
We can therefore write
\begin{align*}
	\tilde{f}_a & = \sum_{i\in\mathbb{N}}{c_ih_i}, \\
	\tilde{f}_b & = \sum_{j\in\mathbb{N}}{d_jk_j}, \\
\end{align*}
\noindent where each $c_i$ is $aG$-invariant, $d_j$ is $bG$-invariant,
and $h_i, k_j$ are $\mZ_G^2(\tilde{X})$-measurable. By Theorem \ref{HKfactors}(iii), we can write $\mathbf{Z}_G^2(\tilde{X}) = \tilde{\mathbf{Z}} \times_{\sigma} H$.
Then by Theorem \ref{thm: limit},
\begin{align*}
	 & \UC_{g \in G}{\eta(\alpha_g)~\mu \left( A \cap T_{ag}^{-1}A \cap T_{bg}^{-1}A \right)} \\
	 & = \UC_{g \in G}{\eta(\alpha_g)~\int_{\tilde{X}}{\tilde{f} \cdot
	 T_{ag} \tilde{f}_a \cdot T_{ag} \tilde{f_b}~d\tilde{\mu}}}\\
	 & = \sum_{i,j\in\mathbb{N}}{\int_{\tilde{X}}{c_i d_j \tilde{f} \cdot
	 \UC_{g \in G}{\eta(\alpha_g)~T_{ag}h_i \cdot T_{bg} k_j~d\tilde{\mu}}}} \\
	 & = \sum_{i,j\in\mathbb{N}}{\int_{\tilde{X} \times Z \times M_a \times M_b}{c_i(x) d_j(x) \tilde{f}(x)
	 \eta(t) h_i \left( \pi_Z(x)+at, \pi_H(x)+u+\psi_1(t,z) \right)}} \\
	 & \qquad \qquad \qquad {{k_j \left( \pi_Z(x)+bt, \pi_H(x)+v+\psi_2(t,z) \right)~d\tilde{\mu}(x)~dt~du~dv}},
\end{align*}
\noindent where $(\pi_Z(x), \pi_H(x)) \in Z \times H$ is the projection of $x \in \tilde{X}$
onto the Conze--Lesigne factor $Z \times H$.
By choosing $\eta : Z \to [0, \infty)$ concentrated on a small neighborhood of $0$ (as in the proof of Theorem \ref{Khintchinefiniteindex}; see Subsection \ref{sec: phi, psi proof}), it remains to show the inequality:
\begin{align} \label{eq: inequality}
	\sum_{i,j}{\int_{\tilde{X} \times M_a \times M_b}{c_i(x) d_j(x) \tilde{f}(x)
	 h_i \left( \pi_Z(x), \pi_H(x)+u \right) k_j \left( \pi_Z(x), \pi_H(x)+v \right)~d\tilde{\mu}(x)~du~dv}}
	 \ge \mu(A)^3.
\end{align}

Let $\mathcal{W}_1$ be the $\sigma$-algebra generated by functions $f \in L^{\infty}(Z \times H)$
such that $f(z, x + y) = f(z,x)$ for every $y \in M_a$.
Similarly, let $\mathcal{W}_2$ be the $\sigma$-algebra generated by functions $f \in L^{\infty}(Z \times H)$
such that $f(z, x + y) = f(z,x)$ for every $y \in M_b$.
Then the left hand side of \eqref{eq: inequality} is equal to

\begin{align} \label{eq: projections}
	\int_{\tilde{X}}{\tilde{f} \cdot E(\tilde{f}|\mathcal{W}_1 \vee \mI_a) \cdot E(\tilde{f}|\mathcal{W}_2 \vee \mI_b)~d\tilde{\mu}}.
\end{align}

By \cite[Lemma 1.6]{Chu}, the quantity \eqref{eq: projections} is bounded below by
$\left( \int_{\tilde{X}}{\tilde{f}~d\tilde{\mu}} \right)^3 = \mu(A)^3$, so \eqref{eq: inequality} holds.\\

Now suppose $a,b\in\mathbb{Z}$ are arbitrary integers and write $a=a'\cdot d$ and $b=b'\cdot d$ where $d=\text{gcd}(a,b)$ and $a',b'$ are coprime. Since $(b-a)G$ has finite index in $G$ we deduce that so does $dG$. Therefore, we can find finitely many ergodic $dG$-invariant measures $\{\mu_i\}_{i=1}^l$ such that $\mu = \frac{1}{l}\sum_{i=1}^l \mu_i$ and all of the systems $\X_i=(X,\mX,\mu_i,dG)$ admit the same Kronecker factor. By the argument above, we can find a suitable $\eta$ satisfying: $$\UC_{g\in dG} \eta(\alpha_g)\mu_i(A\cap T_{a'g}^{-1} A \cap T_{b'g}^{-1}A)>\mu_i(A)^3-\varepsilon$$
for all $i=1,...,l$, and $\UC_{g\in dG} \eta(\alpha_g)=1$.
Therefore, by Jensen's inequality we have
$$\UC_{g\in dG} \eta(\alpha_g)\mu(A\cap T_{a'g}^{-1} A \cap T_{b'g}^{-1}A)>\mu(A)^3-\varepsilon.$$
As in the proof of Theorem \ref{Khintchinefiniteindex}, we conclude that $$\{g\in dG : \mu(A\cap T_{a'g}^{-1} A \cap T_{b'g}^{-1}A)>\mu(A)^3-\varepsilon\}$$ is syndetic. Since $dG$ has finite index in $G$, this implies that $$\{g\in G : \mu(A\cap T_{ag}^{-1} A \cap T_{bg}^{-1}A)>\mu(A)^3-\varepsilon\}$$ is syndetic, as required. \qed

\section{Proof of Theorem \ref{counterexample}}\label{proof}

In this section, we prove Theorem \ref{counterexample}, restated here for the convenience of the reader:

\begin{thm}[Theorem \ref{counterexample}]
    Let $G = \bigoplus_{n=1}^{\infty}{\Z}$.
    Let $l \in \N$.
    There exists $P = P(l)$ such that, for any $a, b \in \N$ with $p \mid \gcd(a,b)$ for some prime $p \ge P$, there is an ergodic $G$-system $\left( X, \mX, \mu, (T_g)_{g \in G} \right)$ and a set $A \in \mX$ with $\mu(A) > 0$ such that
    \begin{align*}
        \mu(A\cap T_{ag}^{-1} A\cap T_{bg}^{-1} A)\leq \mu(A)^l
    \end{align*}
    for every $g\ne0$.
\end{thm}

Rather than constructing a $\bigoplus_{n=1}^{\infty}{\Z}$-system directly, we will instead construct a $\bigoplus_{n=1}^{\infty}{\Z/p^2\Z}$-system.
Since $\bigoplus_{n=1}^{\infty}{\Z/p^2\Z}$ is a quotient of $\bigoplus_{n=1}^{\infty}{\Z}$, the system we construct can be lifted to an ergodic $\bigoplus_{n=1}^{\infty}{\Z}$-system.
Hence, Theorem \ref{counterexample} follows from:

\begin{thm}\label{counterexample p}
    For any $a,b,l \in \N$, there exists a prime $p$ (sufficiently large), an ergodic $\bigoplus_{n=1}^{\infty}\mathbb{Z}/p^2\mathbb{Z}$-system $\X = \left( X, \mX, \mu, (T_g)_{g \in \bigoplus_{n=1}^{\infty}{\Z/p^2\Z}} \right)$, and a set $A \in \mX$ with $\mu(A) > 0$ such that
    \begin{align*}
        \mu(A\cap T_{pag}^{-1} A\cap T_{pbg}^{-1} A)\leq \mu(A)^l
    \end{align*}
    for every $g \ne 0$.
\end{thm}

The proof of Theorem \ref{counterexample p} is based on the following result of Behrend \cite{Beh}.
\begin{thm}\label{Beh-additive}
Let $a,b\in\mathbb{N}$ be distinct and non-zero. There is an absolute constant $c > 0$ such that: for every $N\in\mathbb{N}$, there is a subset $B\subseteq \{0,1,...,N-1\}$ such that $|B|>N\cdot e^{-c\sqrt{\log(N)}}$ and $B$ contains no configurations of the form $\{n,n+am,n+bm\}$ for $m\not = 0$.
\end{thm}
For every prime number $p$, let $C_p=\{z\in\mathbb{C} : z^p =1\}$ denote the group of all roots of unity of order $p$ and let $\omega_p=e^{2\pi i/{p}}$ be the first $p$-th root of unity in $\mathbb{C}$. The following is an immediate corollary of Behrend's theorem.
 \begin{lem}\label{Beh-multiplicative}
     Let $a,b\in\mathbb{N}$ be distinct, then for every $l$, there exists a sufficiently large prime $p$ and a subset $B\subseteq C_p$ of size $|B|>p^{1-\frac{1}{l-1}}$ which contains no configurations of the form $\{y,y\cdot x^a,y\cdot x^b\}$ for $x\not=1$.
 \end{lem}
Throughout this section, we let $\mathcal{T}_p := C_p^{\mathbb{N}}$ and $G_p:= \bigoplus_{i\in I} \mathbb{Z}/p\mathbb{Z}$.\\

We start by giving a proof that the large intersection property fails for non-ergodic systems.
 \begin{lem}\label{non-ergodic}
 Let $a,b \in \mathbb{Z}$ be distinct and nonzero. For every $L\in\mathbb{N}$, there is a $P=P(L)$, such that for every prime $p\geq P$, there is a $G_p$-system $(X,\mX,\mu,(T_g)_{g\in G_p})$ such that, for every $l\leq L$, there is a measurable set $A=A(l)$ with $\mu(A)>0$ and $$\mu(A\cap T_{ag} A \cap T_{bg} A)\leq \mu(A)^l$$ for every $g\not=0$.
 \end{lem}
This result was previously established in \cite[Proposition 10.11]{ABB}, but we give a different proof that will be useful later on.
 \begin{proof}
 Let $p$ be a prime number and let $X_p=\mathcal{T}_p\times C_p$. We equip $X_p$ with the Borel $\sigma$-algebra, the Haar measure $\mu$, and the action of $G_p$ by $$T_g(x,u) = (x,\prod_{i=1}^\infty x_i^{g_i} u).$$
 Now fix a subset   $B\subseteq C_p$ which avoids configurations of the form $\{y,y\cdot x^{a}, y\cdot x^{b}\}$ whenever $x\not =1$, and let $A=\mathcal{T}_p\times B$. It is easy to see that  $\mu(A) = \frac{|B|}{p}$ and we have
 \begin{align*}
    \mu(A\cap T_{ag} A\cap T_{bg} A) = \int_{\mathcal{T}_p^2} 1_B(y) 1_B\left(y\prod_{i\in I}x^{ag_i}\right) 1_B\left(y\prod_{i\in I}x^{bg_i}\right) dx dy &=\\ \int_{\mathcal{T}_p^2} 1_B(y) 1_B\left(y\cdot \left(\prod_{\{i~:~ g_i\not = 0\}} x_i\right)^a\right) 1_B\left(y\cdot \left(\prod_{\{i~:~ g_i\not = 0\}} x_i\right)^b\right) dx dy &=\\ \mu_{\mathcal{T}_p^2}\left(\left\{(y,x)\in\mathcal{T}_p^2 : \left\{y,y\cdot \left(\prod_{\{i~:~ g_i\not = 0\}} x_i\right)^a, y\cdot \left(\prod_{\{i~:~ g_i\not = 0\}} x_i\right)^b\right\}\subset B\right\}\right).
 \end{align*}
 But, $\left\{y,y\cdot \left(\prod_{\{i~:~ g_i\not = 0\}} x_i\right)^a, y\cdot \left(\prod_{\{i~:~ g_i\not = 0\}} x_i\right)^b\right\}\subset B$ if and only if $\prod_{\{i~:~ g_i\not = 0\}} x_i=1$. Since $g\not =0$, we deduce that $\mu(A\cap T_{ag} A\cap T_{bg} A) = \frac{|B|}{p^2} = \frac{p^{l-2}}{|B|^{l-1}} \mu(A)^l$. Now, choose $P$ sufficiently large for which there exists a set $B$ with $|B|>p^{1-\frac{1}{l-1}}$ (Lemma \ref{Beh-multiplicative}). Then $\mu(A\cap T_{ag} A\cap T_{bg} A)<\mu(A)^l$ as required.
 \end{proof}
 
 Roughly speaking, the idea in this section is to construct an ergodic $p$-th root for the system above.\\
 
    We fix some $P$ sufficiently large as in Lemma \ref{non-ergodic}, and let $p>P$ be a prime number. For convenience of notations we let $\omega= e^{2\pi i/p}$ and $\eta = e^{2\pi i/{p^2}}$. We define an action of $G=\bigoplus_{n\in\mathbb{N}}\mathbb{Z}/p^2\mathbb{Z}$ on $\mathcal{T}$ by setting $S_g x= \zeta(g) x$, where  $\zeta(g)=(\eta^{pg_i})_{i\in \mathbb{N}}=(\omega^{g_i})_{i\in\mathbb{N}}$. Since the image of $\zeta$ is dense in $\mathcal{T}$, the action is ergodic.\\

Now, we extend this action to the product space $X=\mathcal{T}\times C_{p^2}$. Let $\varphi:C_p\rightarrow C_{p^2}$ be the map $$\varphi(e^{\frac{2\pi i x}{p}}) = e^{\frac{2\pi i |x|_p}{p^2}}$$ where $|x|_p = x \mod p$. Then $\varphi$ is a cross-section of the canonical projection $C_{p^2}\rightarrow C_p$ and we have that $\varphi(x)^p = x$, and $\varphi(\omega)=\eta$. Our goal is to define an action $(T_g)_{g\in G}$ on $X$ such that $T_{pg}(t,u) = (t,\prod_{i\in\mathbb{N}} t_i^{pg_i}\cdot u)$. 

We do so in two steps. We define an action $T'_g$ on $X$ which satisfies that $T'_{e_i} (t,u) = (S_{e_i} t , \varphi(t_i) u)$, for every $i\in\mathbb{N}$, where $e_i\in\bigoplus_{n=1}^\infty \mathbb{Z}/p^n\mathbb{Z}$ is the $i$-th unit vector. Writing $g=\sum_{i\in\mathbb{N}}g_ie_i$ and using the group law, we get the following action:
\begin{equation} \label{preaction}
T'_g(t,u) = \left(S_g t , \prod_{j=1}^\infty \prod_{k=0}^{g_j-1} \varphi(\omega^k t_j) \cdot u\right)
\end{equation}
where an empty product $\prod_{k=0}^{-1}{x_k}$ is equal to 1.

Unfortunately, this action is not what we are looking for. Indeed, $$(T'_{e_j})^p(t,u) = (t, \prod_{k=0}^{p-1}\varphi(\omega^k\cdot t_j) u) =(t, t_j\cdot \eta^{\binom{p}{2}}\cdot u).$$ To fix that we let $\xi=\omega^\frac{1-p}{2}$ be a $p$-th root of $\overline{\eta}^{\binom{p}{2}}$ and change the action accordingly:
\begin{equation}\label{action}
T_g(t,u) = \left(S_g t , \prod_{j=1}^\infty\left( \prod_{k=0}^{g_j-1} \varphi(\omega^k t_j) \cdot \xi^{g_j}\right) \cdot u\right).
\end{equation}
\begin{lem} For every $t\in\mathcal{T}$, $u\in C_{p^2}$ and $g\in G$ we have
\begin{equation}\label{action'}T_{pg} (t,u) = (t, t^{pg}u).
\end{equation}
\end{lem}
\begin{proof}
The proof is a direct computation. Indeed, it suffices to prove that (\ref{action'}) holds for $g=e_j$ for every $j\in\mathbb{N}$.
Let $j\in\mathbb{N}$ be arbitrary. Since $\omega$ is of order $p$, $S_{pg}t = t$. As for the second coordinate observe that $$\prod_{k=0}^{p-1} \varphi(\omega^k t_j) \cdot \xi^p =  \xi^p\cdot \eta^{\binom{p}{2}}\cdot t_j = t_j.$$
The first equality follows because the product is independent on $t_j$ and always equals to $\varphi(\omega)\cdot...\cdot \varphi(\omega^{p-1})=\eta^{\binom{p}{2}}$, and the last equality follows from the definition of $\xi$. This completes the proof of the lemma.
\end{proof}
The main difficulty in the proof is showing that this action is ergodic.
\begin{lem}
The action (\ref{action}) on $X$ is ergodic.
\end{lem}
\begin{proof}
We use Zimmer criterion for ergodicity (Lemma \ref{minimal}). Since the action of $G$ on $\mathcal{T}$ is ergodic, it is enough to show that the cocycle $\sigma:G\times \mathcal{T}\rightarrow C_{p^2}$, $\sigma(g,t) = \prod_{i=1}^\infty \prod_{k=0}^{g_j-1} \varphi(\omega^kt_j)$ is minimal. Since $C_p$ is the largest proper subgroup of $C_{p^2}$, it is enough to show that $\sigma$ is not cohomologous to a cocycle taking values in $C_p$. Suppose by contradiction that there exists a cocycle $\tau:\mathcal{T}\rightarrow C_p$ cohomologous to $\sigma$. Since $\tau^p=1$, we deduce that $\sigma(g,t)^p =  \prod_{i=1}^\infty \omega^{\binom{g_i}{2}} t_i^{g_i} \xi^{pg_i}$ is a coboundary. Therefore, there exists $F:\mathcal{T}\rightarrow S^1$ such that \begin{equation}\label{coboundary}\sigma^p(g,t) = \frac{F(S_gt)}{F(t)}
\end{equation}
for every $g\in G$ and $t\in T$.
Observe that for every $g,h\in G$, $\Delta_h\sigma^p(g,t)$ is a constant in $t$. Therefore, by (\ref{coboundary}), $\Delta_{h_1}\Delta_{h_2}F$ is a constant for every $h_1,h_2\in G$. Let $s\in \mathcal{T}$ and define $\Delta_s F(x)  = \frac{F(sx)}{F(x)}$. We claim that $\Delta_s F(x)$ is an eigenfunction. Let $g_1,g_2\in G$, then $\Delta_{g_1}\Delta_{g_2} \Delta_s F(x) = \Delta_s \Delta_{g_1} \Delta_{g_2} F(x)=1$. Hence, by ergodicity $\Delta_{g_2}\Delta_{s}F$ is constant and $\Delta_s F$ is an eigenfunction for every $s\in Z$. Recall that translations by $s\in Z$ are continuous with respect to the $L^2$-norm. In particular, there exists an open subgroup $U\leq \mathcal{T}$ such that \begin{equation}\label{inequality}
\|\Delta_s F - 1\|_{L^2(\mu_\mathcal{T})} < \sqrt{2}.
\end{equation} By ergodicity, the multiplicity of each eigenvalue is $1$. Since eigenfunctions with different eigenvalues are orthogonal, it follows that $\Delta_s F$ is a constant for all $s\in U$. Otherwise, $\Delta_s F$ is orthogonal to $1$ and then $$\|\Delta_s F - 1\|_{L^2(\mu_\mathcal{T})}^2 = \|\Delta_s F\|_{L^2}^2 + \|1\|_{L^2}^2 = 2$$ which contradicts (\ref{inequality}). Now, choose $g\in G$ such that $\prod_{i=1}^\infty \omega^{g_i}\in U$ (such $g$ must exist by density). Then if we take $s=\omega^g$, equation (\ref{coboundary}) implies that $\sigma^p(g,\cdot)$ is a constant. As $\sigma^p(g,t)$ clearly depends on $t$, this is a contradiction.
\end{proof}
We now complete the proof of Theorem \ref{counterexample p}. Let $B\subseteq C_p$ be as in Lemma \ref{Beh-multiplicative}. Let $\pi:C_{p^2}\rightarrow C_p$ be the map $\pi_i(x) = x_1^p$ and let $A=\mathcal{T}\times \tilde{B}$ where $\tilde{B} = \pi^{-1}(B)$. Then $\mu_X(A) = \frac{|B|}{p}$, and as in the proof of Lemma \ref{non-ergodic}, 
$$\mu_X(A\cap T_{apg} A \cap T_{bpg} A)=\frac{|B|}{p^2}=\frac{p^{l-2}}{|B|^{l-1}}\mu_X(A)^l<\mu_X(A)^{l}.$$
This completes the proof. \qed

\section{3-point configurations in $\Z^2$} \label{sec: Z^2}

In this section, we establish ergodic popular difference densities for all 3-point matrix patterns in $\Z^2$. The results are summarized in Table \ref{table: epdd} in the introduction.

\subsection{Ergodic popular difference densities when $r(M_1, M_2) = (2,1,1)$}

The following Theorem gives an affirmative answer to Question \ref{Khintchine} for the group $G = \Z^2$:

\begin{thm} \label{thm: (2,1,1)}
	Suppose $M_1$ and $M_2$ are $2\times2$ matrices such that $r(M_1, M_2) = (2,1,1)$.
	Then for any $\alpha \in (0,1)$, $\epdd_{M_1, M_2}(\alpha) = \alpha^3$.
\end{thm}

An example of the configurations handled by Theorem \ref{thm: (2,1,1)}
is the class of all axis-aligned right triangles in $\Z^2$, $\{(a,b), (a+n,b), (a,b+m)\}$,
which corresponds to the choice of matrices
\begin{align*}
	M_1 = \left( \begin{array}{cc} 1 & 0 \\ 0 & 0 \end{array} \right) \qquad \text{and} \qquad
	M_2 = \left( \begin{array}{cc} 0 & 0 \\ 0 & 1 \end{array} \right).
\end{align*}

\begin{proof}[Proof of Theorem \ref{thm: (2,1,1)}]
	Without loss of generality, we may assume $\rk(M_1) = \rk(M_2) = 1$ and $\rk(M_2 - M_1) = 2$.
	Indeed, if $\rk(M_1) = 2$, we may rearrange the expression
	\begin{align*}
		\mu \left( A \cap T_{M_1\vec{n}}^{-1}A \cap T_{M_2\vec{n}}^{-1}A \right)
		 = \mu \left( A \cap T_{(M_1-M_2)\vec{n}}^{-1}A \cap T_{-M_2\vec{n}}^{-1}A \right)
	\end{align*}
	and the new matrices $N_1 = M_1-M_2$ and $N_2 = -M_2$ satisfy the desired conditions.
	
	We now break the proof into two cases depending on the diagonalizability of $M_1$ and $M_2$.
	Note that, since $M_i$ has rank 1, its characteristic polynomial is of the form $x(x-a)$ for some $a \in \Z$.
	Hence, if $M_i$ has a nonzero eigenvalue, then it has an integer eigenvalue (in this case, equal to $a$)
	and is diagonalizable. \\
	
	\underline{Case 1}: $M_1$ or $M_2$ has a nonzero eigenvalue.
	
	Without loss of generality, we may assume that $M_1$ has a nonzero eigenvalue
	and is therefore diagonalizable.
	Hence, there is a nonsingular $2\times2$ integer matrix $P$, an integer $a \in \Z$,
	and a rank 1 matrix $N_2$ with integer entries such that
	\begin{align*}
		M_1P = P\left( \begin{array}{cc} a & 0 \\ 0 & 0 \end{array} \right), \qquad
		M_2P = PN_2, \qquad \text{and} \qquad
		\rk \left( N_2 - \left( \begin{array}{cc} a & 0 \\ 0 & 0 \end{array} \right) \right) = 2.
	\end{align*}
	It is straightforward to check that, in order to satisfy the constraints on rank, $N_2$ must be of the form
	\begin{align*}
		N_2 = \left( \begin{array}{cc} cd & c \\ bd & b \end{array} \right)
	\end{align*}
	with $b \ne 0$.
	By changing to the basis $\binom{1}{-d}, \binom{0}{1}$, we may further assume $d = 0$.
	
	Suppose $\left( X, \mX, \mu, (T_{\vec{n}})_{\vec{n} \in \Z^2} \right)$ is a measure-preserving $\Z^2$-system
	(we do not need to assume that the system is ergodic here), and let $A \in \mX$ with $\mu(A) = \alpha$.
	Define a new $\Z^2$-action by $S_{\vec{n}} := T_{P\vec{n}}$.
	Then
	\begin{align*}
		\UC_{\vec{n} \in \Z^2}{\mu \left( A \cap T_{M_1P\vec{n}}^{-1}A \cap T_{M_2P\vec{n}}^{-1}A \right)}
		 = \UC_{\vec{n} \in \Z^2}{\mu \left( A \cap S_{(an_1,0)}^{-1}A \cap S_{(cn_2,bn_2)}^{-1}A \right)}.
	\end{align*}
	Now put $S_1 := S_{(a,0)}$ and $S_2 := S_{(c,b)}$.
	By Lemma \ref{lem: fubini} and the mean ergodic theorem, we have
	\begin{align*}
		\UC_{\vec{n} \in \Z^2}{\mu \left( A \cap T_{M_1P\vec{n}}^{-1}A \cap T_{M_2P\vec{n}}^{-1}A \right)}
		 & = \UC_{n_2 \in \Z}{\UC_{n_1 \in \Z}{\mu \left( A \cap S_1^{-n_1}A \cap S_2^{-n_2}A \right)}} \\
		 & = \int_X{\ind_A \cdot \E(\ind_A \mid \mI(S_1)) \cdot \E(\ind_A \mid \mI(S_2))} \\
		 & \ge \alpha^3.
	\end{align*}
	Therefore, for any $\eps > 0$, the set
	\begin{align*}
		R_{\eps}
		 := \left\{ \vec{n} \in \Z^2 : \mu \left( A \cap T_{M_1P\vec{n}}^{-1}A \cap T_{M_2P\vec{n}}^{-1}A \right) > \alpha^3 - \eps \right\}
	\end{align*}
	is syndetic.
	Noting that $P$ is nonsingular, it follows that the set $P(R_{\eps})$ is also syndetic in $\Z^2$.
	But for any $\vec{m} \in P(R_{\eps})$, we have
	\begin{align*}
		\mu \left( A \cap T_{M_1\vec{m}}^{-1}A \cap T_{M_2\vec{m}}^{-1}A \right) > \alpha^3 - \eps.
	\end{align*}
	This shows $\epdd_{M_1, M_2}(\alpha) \ge \alpha^3$. \\
	
	To see the upper bound $\epdd_{M_1,M_2}(\alpha) \le \alpha^3$,
	let $\left( X, \mX, \mu, (T_{\vec{n}})_{\vec{n} \in \Z^2} \right)$ be mixing of order 3.
	Then for any $A \in \mX$, we have $\mu \left( A \cap T_{\vec{n}}^{-1}A \cap T_{\vec{m}}^{-1}A \right) \to \mu(A)^3$
	as $\vec{n}, \vec{m}, \vec{m}-\vec{n} \to \infty$.
	Let $P$ be a nonsingular $2\times2$ matrix with integer entries and $a, b, c \in \Z$ with $a, b \ne 0$ such that
	\begin{align*}
		PM_1 = \left( \begin{array}{cc} a & 0 \\ 0 & 0 \end{array} \right)P, \qquad \text{and} \qquad
		PM_2 = \left( \begin{array}{cc} 0 & c \\ 0 & b \end{array} \right)P.
	\end{align*}
	The group of transformations $\tilde{T}_{\vec{n}} := T_{P\vec{n}}$ is still mixing of order 3.
	Write $\vec{m} = P\vec{n}$ for $\vec{n} \in \Z^2$.
	If $m_1 \to \infty$ and $m_2 \to \infty$, then
	\begin{align*}
		\mu \left( A \cap \tilde{T}_{M_1\vec{n}}^{-1}A \cap \tilde{T}_{M_2\vec{n}}^{-1}A \right)
		 = \mu \left( A \cap T_{(am_1,0)}^{-1}A \cap T_{(cm_2,bm_2)}^{-1}A \right) \to \mu(A)^3.
	\end{align*}
	Hence, for any $\eps > 0$, there is a finite set $F \subseteq \Z$ such that
	\begin{align*}
		\left\{ \vec{n} \in \Z^2 : \mu \left( A \cap \tilde{T}_{M_1\vec{n}}^{-1}A \cap \tilde{T}_{M_2\vec{n}}^{-1}A \right)
		 > \mu(A)^3 + \eps \right\}
		 \subseteq \left\{ \vec{n} \in \Z^2 : P\vec{n} \in (F \times \Z) \cup (\Z \times F) \right\}.
	\end{align*}
	A union of finitely many lines in $\Z^2$ is not syndetic, so
	\begin{align*}
		\syndsup_{\vec{n}\in\Z^2}{\mu \left( A \cap \tilde{T}_{M_1\vec{n}}^{-1}A \cap \tilde{T}_{M_2\vec{n}}^{-1}A \right)} \le \mu(A)^3.
	\end{align*} \\
	
	\underline{Case 2}: $M_1$ and $M_2$ have no nonzero eigenvalues.
	
	Since $M_1$ has rank 1, there is a nonsingular $2\times2$ integer matrix $P$, a nonzero integer $a \in \Z$,
	and a rank 1 matrix $N_2$ with integer entries and characteristic polynomial $x^2$ such that
	\begin{align*}
		M_1P = P\left( \begin{array}{cc} 0 & a \\ 0 & 0 \end{array} \right), \qquad
		M_2P = PN_2, \qquad \text{and} \qquad
		\rk \left( N_2 - \left( \begin{array}{cc} 0 & a \\ 0 & 0 \end{array} \right) \right) = 2.
	\end{align*}
	Write
	\begin{align*}
		N_2 = \left( \begin{array}{cc} s & t \\ u & v \end{array} \right).
	\end{align*}
	Since $N_2$ has characteristic polynomial $x^2$, we have $s + v = 0$ and $sv = tu$.
	Therefore, if $u = 0$, then $s = v = 0$.
	But then
	\begin{align*}
		N_2 - \left( \begin{array}{cc} 0 & a \\ 0 & 0 \end{array} \right)
		 = \left( \begin{array}{cc} 0 & t-a \\ 0 & 0 \end{array} \right)
	\end{align*}
	has rank at most 1.
	Thus, we must have $u \ne 0$.
	It follows that $N_2$ can be written in the form
	\begin{align*}
		N_2 = \left( \begin{array}{cc} db & -d^2b \\ b & -db \end{array} \right)
	\end{align*}
	for some $b, d$ with $b \ne 0$.
	Changing to the basis $\binom{1}{0}, \binom{d}{1}$, we may assume $d = 0$ so that
	\begin{align*}
		N_2 = \left( \begin{array}{cc} 0 & 0 \\ b & 0 \end{array} \right).
	\end{align*}
	
	Given a $\Z^2$-system $\left( X, \mX, \mu, (T_{\vec{n}})_{\vec{n} \in \Z^2} \right)$, note that
	\begin{align*}
		\mu \left( A \cap T_{N_1\vec{n}}^{-1}A \cap T_{N_2\vec{n}}^{-1}A \right)
		 = \mu \left( A \cap T_{(an_2,0)}^{-1}A \cap T_{(0,bn_1)}^{-1}A \right).
	\end{align*}
	Hence, replacing $(n_1,n_2)$ by $(n_2,n_1)$, we reduce to Case 1.
\end{proof}

\subsection{Ergodic popular difference densities when $r(M_1,M_2) = (1,1,1)$}

For matrix configurations with $r(M_1, M_2) = (1,1,1)$, we must distinguish between several cases.
First, when $M_1$ and $M_2$ commute, a construction based on Behrend's theorem shows that
the ergodic popular difference density decays faster than any polynomial:

\begin{thm} \label{thm: (1,1,1) commuting}
	Suppose $M_1$ and $M_2$ are commuting $2\times2$ matrices such that $r(M_1, M_2) = (1,1,1)$.
	Then for any sufficiently small $\alpha \in (0,1)$, $\epdd_{M_1, M_2}(\alpha) < \alpha^{c\log(1/\alpha)}$,
	where $c > 0$ is an absolute constant.
\end{thm}

Theorem \ref{thm: (1,1,1) commuting} applies to collinear three-point configurations up to scaling and translation.

\begin{proof}[Proof of Theorem \ref{thm: (1,1,1) commuting}]
	We first distinguish between two cases depending on diagonalizability of $M_1$ and $M_2$. \\
	
	\underline{Case 1}: $M_1$ or $M_2$ has a nonzero eigenvalue.
	
	Without loss of generality, assume $M_1$ has a nonzero eigenvalue and is therefore diagonalizable.
	Since $M_2$ and $M_2 - M_1$ are also rank 1 and commute with $M_1$,
	there exists a nonsingular $2\times2$ matrix $P$ with integer entries
	and $a, b \in \Z$ be distinct and nonzero such that
	\begin{align} \label{eq: rank1diag}
		PM_1 = \left( \begin{array}{cc} a & 0 \\ 0 & 0 \end{array} \right)P \qquad \text{and} \qquad
		PM_2 = \left( \begin{array}{cc} b & 0 \\ 0 & 0 \end{array} \right)P.
	\end{align} \\
	
	\underline{Case 2}: $M_1$ and $M_2$ have no nonzero eigenvalues.
	
	Using the condition $r(M_1,M_2) = (1,1,1)$, there is a nonsingular $2\times2$ integer matrix $P$,
	a nonzero integer $a \in \Z$, and a rank 1 matrix $N_2$ with integer entries and characteristic polynomial $x^2$
	such that
	\begin{align*}
		M_1P = P\left( \begin{array}{cc} 0 & a \\ 0 & 0 \end{array} \right), \qquad
		M_2P = PN_2, \qquad \text{and} \qquad
		\rk \left( N_2 - \left( \begin{array}{cc} 0 & a \\ 0 & 0 \end{array} \right) \right) = 1.
	\end{align*}
	Moreover, $N_2$ commutes with the matrix $\left( \begin{array}{cc} 0 & a \\ 0 & 0 \end{array} \right)$.
	Write
	\begin{align*}
		N_2 = \left( \begin{array}{cc} s & t \\ u & v \end{array} \right).
	\end{align*}
	Note that
	\begin{align*}
		\left[ \left( \begin{array}{cc} 0 & a \\ 0 & 0 \end{array} \right),
		 \left( \begin{array}{cc} s & t \\ u & v \end{array} \right) \right]
		 = \left( \begin{array}{cc} au & a(v-s) \\ 0 & au \end{array} \right),
	\end{align*}
	so $u = 0$ and $v = s$.
	On the other hand, since $N_2$ has characteristic polynomial $x^2$, we have $s + v = 0$ and $sv = tu$.
	Hence, $s = v = 0$, and $N_2$ is of the form
	\begin{align*}
		N_2 = \left( \begin{array}{cc} 0 & b \\ 0 & 0 \end{array} \right)
	\end{align*}
	with $b \notin \{0,a\}$.
	
	Now, replacing $(n_1,n_2) \in \Z^2$ by $(n_2,n_1) \in \Z^2$ and using the identity
	\begin{align*}
		\left( \begin{array}{cc} 0 & c \\ 0 & 0 \end{array} \right) \left( \begin{array}{c} n_2 \\ n_1 \end{array} \right)
		 = \left( \begin{array}{cc} c & 0 \\ 0 & 0 \end{array} \right) \left( \begin{array}{c} n_1 \\ n_2 \end{array} \right)
	\end{align*}
	for $c \in \Z$, we can reduce Case 2 to Case 1. \\
	
	Without loss of generality, let $P$ be a nonsingular $2\times2$ matrix with integer entries
	and $a, b \in \Z$ distinct and nonzero such that \eqref{eq: rank1diag} holds.
	Put $d := \left| \det(P) \right| \in \N$.
	
	Define $S : \T^2 \to \T^2$ by $S(x,y) := (x, y+x)$.
	Let $R : \T^2 \to \T^2$ be the transformation $R(x,y) = (2x, 2y + x)$.
	Both $S$ and $R$ preserve the Haar probability measure $\mu$ on $\T^2$.
	We claim that the $(\Z_{\ge 0})^2$-action generated by $S$ and $R$ is ergodic (with respect to $\mu$).
	To see this, suppose $f \in L^2(\T^2)$ is simultaneously $S$- and $R$-invariant, and expand $f$ as a Fourier series
	\begin{align*}
		f(x,y) = \sum_{n,m}{c_{n,m}e(nx+my)},
	\end{align*}
	where $e(t) := e(2\pi i t)$.
	Then
	\begin{align*}
		(Sf)(x,y) = \sum_{n,m}{c_{n,m}e((n+m)x+my)}
		 = \sum_{n,m}{c_{n-m,m} e(nx+my)}.
	\end{align*}
	Therefore, since $Sf = f$, we have $c_{n,m} = c_{n-m,m}$ for all $n,m \in \Z$.
	By Parseval's identity, $\sum_{n,m}{|c_{n,m}|^2} = \|f\|_2^2 < \infty$,
	so $c_{n,m} = 0$ whenever $m \ne 0$.
	That is, $f(x,y) = \sum_n{c_{n,0}e(nx)}$.
	Now,
	\begin{align*}
		(Rf)(x,y) = \sum_n{c_{n,0}e(2nx)}.
	\end{align*}
	Hence, since $Rf = f$, we have $c_{2n,0} = c_{n,0}$ for every $n \in \Z$.
	Applying Parseval's identity once again, we conclude that $c_{n,0} = 0$ for $n \ne 0$.
	Thus, $f(x,y) = c_{0,0}$ is a constant function.\\
	
	Fix $\alpha \in (0,1)$.
	By \cite[Theorem 1.3]{BHK}, there exists a set $A \subseteq \T^2$ with $\mu(A) = \alpha$
	such that $\mu \left( A \cap S^{-an}A \cap S^{-bn} A \right) < \alpha^{c\log(1/\alpha)}$ for $n \ne 0$,
	where $c > 0$ is an absolute constant.\footnote{The statement of \cite[Theorem 1.3]{BHK} only gives a bound
	of the form $\alpha^l$ rather than $\alpha^{c\log(1/\alpha)}$.
	However, as noted in \cite{BHK} immediately after the statement, the construction of the set $A$
	gives this stronger bound via Behrend's theorem on sets without three-term arithmetic progressions \cite{Beh}.
	Additionally, \cite[Theorem 1.3]{BHK} is only stated for the case $a=1, b = 2$,
	but the same method works for general $a, b$; see, e.g., \cite[Section 11]{ABB}.}
	
	Let $\left( X, \mX, \nu, (T_{\vec{n}})_{\vec{n} \in \Z^2} \right)$ be an ergodic $\Z^2$-system and $B \in \mX$ with $\nu(B) = \alpha$ such that
	\begin{align*}
		\nu \left( B \cap T_{\vec{n}}^{-1}B \cap T_{\vec{m}}^{-1}B \right)
		 = \mu \left( A \cap S^{-n_1}R^{-n_2}A \cap S^{-m_1}R^{-m_2}A \right)
	\end{align*}
	for every $\vec{n},\vec{m} \in \Z \times \Zp$.
	(Note that, because $R$ is non-invertible, we cannot simply take $X = \T^2$, $\nu = \mu$, $B = A$,
	and $T_{\vec{n}} = S^{n_1}R^{n_2}$.)
	Then let $\tilde{T}_{\vec{n}} := T_{P\vec{n}}$ for $\vec{n} \in \Z^2$.
	
	Since $[\Z^2 : P(\Z^2)] = \left| \det(P) \right| = d < \infty$, the system
	$\left( X, \mX, \nu, (\tilde{T}_{\vec{n}})_{\vec{n} \in \Z^2} \right)$ has at most $d$ ergodic components.
	Hence, we may write the ergodic decomposition as $\nu = \frac{1}{k} \sum_{i=1}^k{\nu_i}$
	for some $k \le d$ and some measure $\nu_i$.
	For some $1 \le i \le k$, we must have $\nu_i(B) \ge \alpha$.
	Without loss of generality, we may therefore assume $\nu_1(B) \ge \alpha$.
	
	Let $\vec{n} \in \Z^2 \setminus \{0\}$.
	Let $\vec{m} = P\vec{n} \in \Z^2$.
	Then
	\begin{align*}
		\nu_1 \left( B \cap \tilde{T}_{M_1\vec{n}}^{-1}B \cap \tilde{T}_{M_2\vec{n}}^{-1}B \right)
		 & = \nu_1 \left( B \cap T_{(am_1,0)}^{-1}B \cap T_{(bm_1,0)}^{-1}B \right) \\
		 & \le d \cdot \mu \left( A \cap S^{-am_1}A \cap S^{-bm_1}A \right).
	\end{align*}
	Hence, if $\nu_1 \left( B \cap \tilde{T}_{M_1\vec{n}}^{-1}B \cap \tilde{T}_{M_2\vec{n}}^{-1}B \right)
	 > d \cdot \alpha^{c\log(1/\alpha)}$, then $m_1 = 0$.
	But since $P$ is nonsingular,
	\begin{align*}
		\left\{ \vec{n} \in \Z^2 : P\vec{n} \in \{0\} \times \Z \right\} = \Q \vec{v} \cap \Z^2
	\end{align*}
	where $\vec{v}$ is the vector $P^{-1} \binom{0}{1} \in \Q^2$.
	Such a set is never syndetic, so $\epdd_{M_1,M_2}(\alpha) \le d \cdot \alpha^{c\log(1/\alpha)}$.
	For $c' < c$ and $\alpha$ sufficiently small, one has
	$d \cdot \alpha^{c\log(1/\alpha)} < \alpha^{c'\log(1/\alpha)}$, so this completes the proof.
\end{proof}

Now suppose $r(M_1, M_2) = (1,1,1)$, and $M_1$ and $M_2$ do not commute.
In this case, $M_1$ or $M_2$ must be diagonalizable,\footnote{If neither $M_1$ nor $M_2$ are diagonalizable,
then they both have characteristic polynomial $x^2$.
By a change of basis, we may assume $M_1$ is in its Jordan form
$M_1 = \left( \begin{array}{cc} 0 & 1 \\ 0 & 0 \end{array} \right)$.
Write $M_2 = \left( \begin{array}{cc} a & b \\ c & d \end{array} \right)$.
The condition $\rk(M_2) = \rk(M_2 - M_1) = 1$ implies that $ad - bc = ad - (b-1)c = 0$, so $c = 0$ and $ad = 0$.
Moreover, since $M_2$ has characteristic polynomial $x^2$, we have $a + d = 0$.
Hence, $M_2 = \left( \begin{array}{cc} 0 & b \\ 0 & 0 \end{array} \right)$.
But then $M_2$ commutes with $M_1$.}
so we assume without loss of generality that $M_1$ is diagonalizable.
We then distinguish between two cases, depending on the form of $M_2$ when $M_1$ is diagonalized.
Call the pair of matrices $(M_1, M_2)$ \emph{row-like} if there is a non-singular $2\times2$ matrix $P$ with rational entries
and rational numbers $a, b, c \in \Q$ with $a, b \ne 0$ such that
\begin{align*}
	PM_1P^{-1} = \left( \begin{array}{cc} a & 0 \\ 0 & 0 \end{array} \right) \qquad \text{and} \qquad
	PM_2P^{-1} = \left( \begin{array}{cc} c & b \\ 0 & 0 \end{array} \right).
\end{align*}
Similarly, call the pair $(M_1, M_2)$ \emph{column-like} if there is a non-singular $2\times2$ matrix $P$
with rational entries and rational numbers $a, b, c \in \Q$ with $a, b \ne 0$ such that
\begin{align*}
	PM_1P^{-1} = \left( \begin{array}{cc} a & 0 \\ 0 & 0 \end{array} \right) \qquad \text{and} \qquad
	PM_2P^{-1} = \left( \begin{array}{cc} c & 0 \\ b & 0 \end{array} \right).
\end{align*}

For row-like configurations, we can use the ``Fubini'' property of uniform Ces\`{a}ro limits (Lemma \ref{lem: fubini})
to show $\epdd(\alpha) = \alpha^3$:

\begin{thm} \label{thm: (1,1,1) row-like}
	Suppose $M_1$ and $M_2$ are $2\times2$ matrices with $r(M_1, M_2) = (1,1,1)$
	such that $(M_1, M_2)$ is row-like.
	Then for any $\alpha \in (0,1)$, $\epdd_{M_1,M_2}(\alpha) = \alpha^3$.
\end{thm}
\begin{proof}
	Let $P$ be a nonsingular $2\times2$ matrix with integer entries such that
	\begin{align*}
		M_1P = P\left( \begin{array}{cc} a & 0 \\ 0 & 0 \end{array} \right) \qquad \text{and} \qquad
		M_2P = P\left( \begin{array}{cc} c & b \\ 0 & 0 \end{array} \right).
	\end{align*}
	By changing to the basis $\binom{b}{-c}, \binom{0}{1}$, we may assume $c = 0$.
	
	Let $\left( X, \mX, \mu, (T_{\vec{n}})_{\vec{n} \in \Z^2} \right)$ be a measure-preserving system,
	and let $A \in \mX$ with $\mu(A) = \alpha > 0$.
	Define a new $\Z^2$-action by $\tilde{T}_{\vec{n}} := T_{P\vec{n}}$, and let $S := \tilde{T}_{(1,0)}$.
	Then
	\begin{align*}
		\mu \left( A \cap T_{M_1P\vec{n}}^{-1}A \cap T_{M_2P\vec{n}}^{-1}A \right)
		 = \mu \left( A \cap S^{-an_1}A \cap S^{-bn_2}A \right).
	\end{align*}
	Thus, by Lemma \ref{lem: fubini}, we have
	\begin{align*}
		\UC_{\vec{n} \in \Z^2}{\mu \left( A \cap T_{M_1P\vec{n}}^{-1}A \cap T_{M_2P\vec{n}}^{-1}A \right)} \ge \alpha^3.
	\end{align*}
	Since $P$ is nonsingular, it follows that
	\begin{align*}
		\syndsup_{\vec{n} \in \Z^2}{\mu \left( A \cap T_{M_1\vec{n}}^{-1}A \cap T_{M_2\vec{n}}^{-1}A \right)} \ge \alpha^3.
	\end{align*}\\
	
	Now we will show $\epdd_{M_1,M_2}(\alpha) \le \alpha^3$.
	Let $P$ be a nonsingular $2\times2$ matrix with integer entries and $a, b, c \in \Z$ with $a, b \ne 0$ such that
	\begin{align*}
		PM_1 = \left( \begin{array}{cc} a & 0 \\ 0 & 0 \end{array} \right)P, \qquad \text{and} \qquad
		PM_2 = \left( \begin{array}{cc} 0 & b\\ 0 & 0 \end{array} \right)P.
	\end{align*}
	Let $\left( X, \mX, \mu, S, R \right)$ be an ergodic $\Z^2$-system such that $S$ is mixing of order 3.
	Define $T_{\vec{n}} := S^{n_1}R^{n_2}$ and $\tilde{T}_{\vec{n}} := T_{P\vec{n}}$ for $\vec{n} \in \Z^2$.
	Then for $A \in \mX$ and $\vec{m} = P\vec{n} \in \Z^2$, we have
	\begin{align*}
		\mu \left( A \cap \tilde{T}_{M_1\vec{n}}^{-1}A \cap \tilde{T}_{M_2\vec{n}}^{-1}A \right)
		 = \mu \left( A \cap S^{-am_1}A \cap S^{-bm_2}A \right).
	\end{align*}
	Since $S$ is mixing of order 3, given $\eps > 0$, there exists a finite set $F \subseteq \Z$
	such that
	\begin{align*}
		& \left\{ \vec{n} \in \Z^2 : \mu \left( A \cap \tilde{T}_{M_1\vec{n}}^{-1}A \cap \tilde{T}_{M_2\vec{n}}^{-1}A \right)
		 > \mu(A)^3 + \eps \right\} \\
		 & \quad \subseteq P^{-1} \left( \left\{ \vec{m} \in \Z^2 :
		  m_1 \in F, m_2 \in F,~\text{or}~bm_2 - am_1 \in F \right\} \right).
	\end{align*}
	This set is a union of finitely many lines in $\Z^2$, so it is not syndetic.
	Hence,
	\begin{align*}
		\syndsup_{\vec{n}\in\Z^2}{\mu \left( A \cap \tilde{T}_{M_1\vec{n}}^{-1}A \cap \tilde{T}_{M_2\vec{n}}^{-1}A \right)} \le \mu(A)^3.
	\end{align*}
\end{proof}

The prototypical column-like configuration is the class of axis-aligned isosceles right triangles,
for which it is known by previous work of Chu \cite{Chu} and Donoso and Sun \cite{DS} that
$\alpha^4 \le \epdd(\alpha) \le \alpha^{4-o(1)}$.
We prove that these bounds extend to all column-like configurations:

\begin{thm}
	Suppose $M_1$ and $M_2$ are $2\times2$ matrices with $r(M_1, M_2) = (1,1,1)$
	such that $(M_1, M_2)$ is column-like.
	Then for any $\alpha \in (0,1)$, $\epdd_{M_1,M_2}(\alpha) \ge \alpha^4$.
	Moreover, for any $l < 4$ and all sufficiently small $\alpha$ (depending on $l$),
	one has $\epdd_{M_1,M_2}(\alpha) \le \alpha^l$.
\end{thm}
\begin{proof}
	Let $\left( X, \mX, \mu, (T_{\vec{n}})_{\vec{n} \in \Z^2} \right)$ be an ergodic $\Z^2$-system.
	Since the pair $(M_1, M_2)$ is column-like, there exists a nonsingular $2\times2$ matrix $P$ with integer entries
	and integers $a, b, c \in \Z$ with $a, b \ne 0$ such that
	\begin{align*}
		M_1P = P\left( \begin{array}{cc} a & 0 \\ 0 & 0 \end{array} \right) \qquad \text{and} \qquad
		M_2P = P\left( \begin{array}{cc} c & 0 \\ b & 0 \end{array} \right).
	\end{align*}
	Then for any $\vec{n} \in \Z^2$, we have
	\begin{align*}
		\mu \left( A \cap T_{M_1P\vec{n}}^{-1}A \cap T_{M_2P\vec{n}}^{-1}A \right)
		 = \mu \left( A \cap T_{P(an_1,0)}^{-1}A \cap T_{P(cn_1,bn_1)}^{-1}A \right).
	\end{align*}
	Letting $S := T_{P(a,0)}$ and $R := T_{P(c,b)}$, we therefore have the identity
	\begin{align*}
		\mu \left( A \cap T_{M_1P\vec{n}}^{-1}A \cap T_{M_2P\vec{n}}^{-1}A \right)
		 = \mu \left( A \cap S^{-n_1}A \cap R^{-n_1}A \right).
	\end{align*}
	Now, since $T$ is ergodic and $P$ is nonsingular,
	the $\Z^2$-action generated by $S$ and $R$ has finitely many ergodic components.
	Thus, by \cite[Theorem 1.1]{Chu},
	\begin{align*}
		\left\{ n \in \Z : \mu \left( A \cap S^{-n}A \cap R^{-n}A \right) \ge \mu(A)^4 \right\}
	\end{align*}
	is syndetic in $\Z$.\footnote{In \cite{Chu}, it is assumed that the system $(X, \mX, \mu, S, R)$ is ergodic.
	However, the proof easily extends to the case that the system has finitely many ergodic components
	by noting that all of the ergodic components will have the same Kronecker factor.}
	It follows that
	\begin{align*}
		\left\{ \vec{n} \in \Z^2 : \mu \left( A \cap T_{M_1\vec{n}}^{-1}A \cap T_{M_2\vec{n}}^{-1}A \right) \ge \mu(A)^4 \right\}
	\end{align*}
	is syndetic in $\Z^2$.
	Hence, $\epdd_{M_1,M_2}(\alpha) \ge \alpha^4$. \\
	
	Let $l < 4$.
	By \cite[Theorem 1.2]{DS}, there exists an ergodic $\Z^2$-system $\left( X, \mX, \mu, S, R \right)$
	and a set $A \in \mX$ such that $\mu \left( A \cap S^{-n}A \cap R^{-n}A \right) < \mu(A)^l$ for every $n \ne 0$.
	Since the pair $(M_1,M_2)$ is column-like, there is a nonsingular $2\times2$ matrix $P$ with integer entries
	and integers $a,b,c \in \Z$ with $a, b \ne 0$ such that
	\begin{align*}
		PM_1 = \left( \begin{array}{cc} a & 0 \\ 0 & 0 \end{array} \right)P \qquad \text{and} \qquad
		PM_2 = \left( \begin{array}{cc} c & 0 \\ b & 0 \end{array} \right)P.
	\end{align*}
	Define $T_{\vec{n}} := S^{bn_1}(R^aS^{-c})^{n_2}$, and let $\tilde{T}_{\vec{n}} := T_{P\vec{n}}$ for $n \in \Z^2$.
	Note that $\left( X, \mX, \mu, (\tilde{T}_{\vec{n}})_{\vec{n} \in \Z^2} \right)$ has finitely many ergodic components.
	To be more precise, the ergodic decomposition has the form $\mu = \frac{1}{k} \sum_{i=1}^k{\mu_i}$
	with $k \le d := \left| ab\det(P) \right|$.
	Without loss of generality, we may assume $\mu_1(A) \ge \mu(A)$.
	
	Now, for any $\vec{n} \ne 0$, we have
	\begin{align*}
		\mu_1 \left( A \cap \tilde{T}_{M_1\vec{n}}^{-1}A \cap \tilde{T}_{M_2\vec{n}}^{-1}A \right)
		 \le d \cdot \mu \left( A \cap S^{-abm_1}A \cap R^{-abm_1}A\right)
	\end{align*}
	where $\vec{m} = P\vec{n} \in \Z^2$.
	Therefore,
	\begin{align*}
		\left\{ \vec{n} \in \Z^2 : \mu_1 \left( A \cap \tilde{T}_{M_1\vec{n}}^{-1}A \cap \tilde{T}_{M_2\vec{n}}^{-1}A \right)
		 \ge d \cdot \mu_1(A)^l \right\}
		 \subseteq \left\{ \vec{n} \in \Z^2 : P\vec{n} \in \{0\} \times \Z \right\}
		 \subseteq \Q \vec{v} \cap \Z^2,
	\end{align*}
	where $\vec{v} = P^{-1}\binom{0}{1} \in \Q^2$.
	The set $\Q \vec{v} \cap \Z^2$ is not syndetic, so this shows $\epdd_{M_1,M_2}(\alpha) \le d \cdot \alpha^l$
	for $\alpha = \mu(A)$.
	Moreover, for any $l' < l$, we have the inequality $d \cdot \alpha^l < \alpha^{l'}$ for all $\alpha > 0$ sufficiently small.
\end{proof}

\subsection{Finitary combinatorial consequences and open questions} \label{sec: Z^2 finitary}

There are two cases in which our ergodic-theoretic results directly imply finitary combinatorial analogues.
Namely, when $r(M_1,M_2) = (2,1,1)$ and when $(M_1,M_2)$
is a row-like pair of non-commuting matrices with $r(M_1,M_2) = (1,1,1)$,
we establish the bound $\epdd_{M_1,M_2}(\alpha) \ge \alpha^3$
with the help of the ``Fubini'' property for uniform Ces\`{a}ro limits (Lemma \ref{lem: fubini}),
and this allows us to avoid assuming that the underlying $\Z^2$-system is ergodic.
For this reason, we can obtain the following combinatorial result:

\begin{thm} \label{thm: Z^2 finitary}
	Let $M_1, M_2$ be $2\times2$ matrices with integer entries.
	Suppose that either
	\begin{enumerate}[(i)]
		\item	$r(M_1,M_2) = (2,1,1)$, or 
		\item	$r(M_1,M_2) = (1,1,1)$, $M_1$ and $M_2$ do not commute, and $(M_1,M_2)$ is row-like.
	\end{enumerate}
	Then for any $\alpha, \eps > 0$, there exists $N_0 = N_0(\alpha, \eps) \in \N$ such that,
	if $N \ge N_0$ and $A \subseteq \{1, \dots, N\}^2$ has $|A| \ge \alpha N^2$, then
	there exists $\vec{n} \in \Z^2$ with $M_1\vec{n}, M_2\vec{n}, (M_2-M_1)\vec{n} \ne 0$ such that
	\begin{align*}
		 \left| \left\{ \vec{x} \in \Z^2 : \{\vec{x}, \vec{x} + M_1\vec{n}, \vec{x} + M_2\vec{n}\} \subseteq A \right\} \right|
		  > (\alpha^3 - \eps) N^2.
	\end{align*}
\end{thm}

\begin{proof}
	Let $\alpha, \eps > 0$ and suppose no such $N_0$ exists.
	Then there is an increasing sequence $(N_k)_{k \in \N}$ in $\N$
	and sets $A_k \subseteq \{1, \dots, N_k\}^2$
	with $|A_k| \ge \alpha N_k^2$ such that
	\begin{align*}
		 \left| A_k \cap \left( A_k - M_1\vec{n} \right) \cap \left( A_k - M_2\vec{n} \right) \right|
		 \le (\alpha^3 - \eps) N_k^2
	\end{align*}
	whenever $M_1\vec{n}, M_2\vec{n}, (M_2-M_1)\vec{n} \ne 0$.

	For notational convenience, let $A_{k,0} := \Z^2 \setminus A_k$ and $A_{k,1} := A_k$.
	By passing to a subsequence if necessary, we may assume without loss of generality that
	\begin{align} \label{eq: cylinder measure}
		\lim_{k \to \infty}{\frac{\left| \left( A_{k,i_1} - \vec{n}_1 \right) \cap \dots
		 \cap \left( A_{k,i_r} - \vec{n}_r \right) \cap \{1, \dots, N_k\}^2 \right|}{N_k^2}}
	\end{align}
	exists for all $r \in \N$, $\vec{n}_1, \dots, \vec{n}_r \in \Z^2$, and $i_1, \dots, i_r \in \{0,1\}$.
	Hence, we may define a measure $\mu$ on the sequence space $\{0,1\}^{\Z^2}$ by setting
	\begin{align*}
		\mu \left( \left\{ x \in X : x(\vec{n}_1) = i_1, \dots, x(\vec{n}_r) = i_r \right\} \right)
	\end{align*}
	equal to the limit \eqref{eq: cylinder measure} and extending with the use of Kolmogorov's extension theorem.
	Since $\left( \{1, \dots, N_k\}^2 \right)_{k \in \N}$ is a F{\o}lner sequence in $\Z^2$,
	the measure $\mu$ is invariant under the shift transformations $(T_{\vec{n}}x)(\vec{m}) := x(\vec{m}+\vec{n})$.
	
	Let $A := \{x \in X : x(\vec{0}) = 1\}$.
	Then $\mu(A) = \lim_{k \to \infty}{\frac{|A_k|}{N_k^2}} \ge \alpha$.
	On the other hand, if $M_1\vec{n}, M_2\vec{n}, (M_2-M_1)\vec{n} \ne 0$, then
	\begin{align*}
		\mu \left( A \cap T_{M_1\vec{n}}^{-1}A \cap T_{M_2\vec{n}}^{-1}A \right)
		 & = \mu \left( \{x \in X : x(\vec{0}) = x(M_1\vec{n}) = x(M_2\vec{n}) = 1\} \right) \\
		 & = \lim_{k \to \infty}{\frac{\left| A_k \cap \left( A_k - M_1\vec{n} \right) \cap \left( A_k - M_2\vec{n} \right) \right|}
		 {N_k^2}} \\
		 & \le \alpha^3 - \eps.
	\end{align*}
	Hence,
	\begin{align*}
		R_{\eps} & := \left\{ \vec{n} \in \Z^2 :
		 \mu \left( A \cap T_{M_1\vec{n}}^{-1}A \cap T_{M_2\vec{n}}^{-1}A \right) > \mu(A)^3 - \eps \right\} \\
		 & \subseteq \ker(M_1) \cup \ker(M_2) \cup \ker(M_2-M_1).
	\end{align*}
	But by the proofs of Theorems \ref{thm: (2,1,1)} and \ref{thm: (1,1,1) row-like},
	$R_{\eps}$ is a syndetic subset of $\Z^2$, so this is a contradiction.
\end{proof}

For general 3-point matrix patterns in $\Z^2$, it remains an open problem to fully determine (finitary combinatorial)
popular difference densities.
One particularly attractive case, which can be seen as a finitary version of Question \ref{Khintchine}
for the group $G = \Z^2$, is the following:

\begin{conj} \label{conj: Z^2 finitary}
	Let $M_1$ and $M_2$ be $2\times2$ matrices with integer entries such that $M_2 - M_1$ has full rank.
	Then for any $\alpha, \eps > 0$, there exists $N_0 = N_0(\alpha, \eps) \in \N$ such that,
	if $N \ge N_0$ and $A \subseteq \{1, \dots, N\}^2$ has cardinality $|A| \ge \alpha N^2$,
	then there exists $\vec{n} \in \Z^2$ with $M_1\vec{n}, M_2\vec{n} \ne 0$ such that
	\begin{align*}
		 \left| \left\{ \vec{x} \in \Z^2 : \{\vec{x}, \vec{x} + M_1\vec{n}, \vec{x} + M_2\vec{n}\} \subseteq A \right\} \right|
		  > (\alpha^3 - \eps) N^2.
	\end{align*}
\end{conj}

The special case when $M_1, M_2$, and $M_2 - M_1$ are all invertible,
Conjecture \ref{conj: Z^2 finitary} was verified by \cite[Theorem 1.1]{BSST}.
Moreover, Theorem \ref{thm: Z^2 finitary} shows that Conjecture \ref{conj: Z^2 finitary} holds
when $M_1$ and $M_2$ are both rank 1 matrices.
The most interesting remaining case is when $M_1$ has full rank and $M_2$ is a rank 1 matrix.

Finally, the column-like family of configurations $\{(a,b), (a+n,b), (a,b+n)\}$, known as \emph{corners},
has been well-studied from the perspective of popular differences in finitary combinatorics.
In particular, it is known that the popular difference density for corners is of the form $\alpha^{4-o(1)}$;
see \cite{Berger} and also \cite{Man, FSSSZ} for an analogous result in a finite characteristic setting.
To the authors' knowledge, such results are not known for general column-like matrix patterns,
but we anticipate that techniques for handling corners should apply in this generality
with only minor modifications needed.

\section{Khintchine-type recurrence for actions of semigroups} \label{sec: semigroups}

As a consequence of Theorem \ref{Khintchineab}, we obtain the following combinatorial result.
For any set $E \subseteq \Q_{>0}$ of positive multiplicative upper Banach density $d^*_{mult}(E) > 0$ and any $\eps > 0$, there exists $q \in \Q_{>0} \setminus \{1\}$ such that
\begin{align*}
	d^*_{mult} \left( E \cap q^{-1}E \cap q^{-2}E \right) > d^*_{mult}(E)^3 - \eps
\end{align*}
(in fact, the set of such $q$ is multiplicatively syndetic).
More generally, for any countable field $\F$, any set $E \subseteq \F^{\times}$
of positive multiplicative upper Banach density $d_{mult}^*(E) > 0$ and any $\eps > 0$,
the set of $x \in \F^{\times}$ such that
\begin{align*}
	d^*_{mult} \left( E \cap x^{-1}E \cap x^{-2}E \right) > d^*_{mult}(E)^3 - \eps
\end{align*}
is multiplicatively syndetic.\footnote{In fact, our results show that for any $k \in \N$, $d_{\text{mult}}^* \left( E \cap x^{-k}E \cap x^{-(k+1)}E \right)$ and $d_{\text{mult}}^*\left(E\cap x^{-1} E \cap x^{-k} E\right)$ can be made arbitrarily close to $d_{\text{mult}}^*(E)^3$ for a multiplicatively syndetic set of $x \in \F^{\times}$.
On the other hand, by Theorem \ref{counterexample},
there are $n, m \in \mathbb{N}$ such that $d_{\text{mult}}^* \left( E \cap x^{-n}E \cap x^{-m}E \right)$ is much smaller than $d_{\text{mult}}^*(E)^3$ for all $x \ne 1$.}
This is suggestive of the following problem.
Let $R$ be an integral domain.
(For example, $R$ can be the ring $\Z$, the ring of integers of a number field, or the polynomial ring $\F[t]$ over a finite field $\F$.)
Given a set $E \subseteq R^{\times}$ of positive multiplicative upper Banach density $d^*_{R, mult}(E) > 0$ and $\eps > 0$, does there exist $r \in R \setminus \{1\}$ such that
\begin{align*}
	d^*_{R, mult} \left( E \cap E/r \cap E/r^2 \right) > d^*_{R, mult}(E)^3 - \eps,
\end{align*}
where $E/r := \left\{ t \in R : rt \in E \right\}$ for $r \in R$?
The goal of this section is to transfer our results into the setting of cancellative abelian semigroups
in order to answer this question affirmatively.

\subsection{The group generated by a cancellative abelian semigroup}

Let $(S, +)$ be a countable cancellative abelian semigroup.
That is, $S$ is a countable set equipped with a commutative and associative binary operation $+$ such that
if $s + t = s + r$ for some $r, s, t \in S$, then $t = r$.

We can define a group $G_S$ as the set of formal differences $\left\{ s - t : s, t \in S \right\}$
where we identify $s - t$ and $s' - t'$ if $s+t' = s'+t$.
More formally, we may define an equivalence relation $\sim$ on $S^2$ by $(s,t) \sim (s',t')$ if $s+t' = s'+t$.
Then $G_S$ is the set of equivalence classes $S^2/\sim$ with the operation $[(s,t)] + [(s',t')] := [(s+s', t+t')]$.
It is easy to check that this operation is well-defined because $S$ is cancellative.
Moreover, $G_S$ has an identity $0 := [(s,s)]$, and for any $s, t \in S$, we have $[(s,t)] + [(t,s)] = 0$.
Thus, $G_S$ is a group.

\subsection{Notions of largeness}

For a set $E \subseteq S$ and an element $t \in S$, let $E - t := \{s \in S : s+t \in E\}$ and $E + t := \{s+t : s \in S\}$.
The following definition summarizes combinatorial notions of largeness that we will use, some of which are defined above in the setting of abelian groups.

\begin{defn}
	Let $(S, +)$ be a countable cancellative abelian semigroup.
	\begin{itemize}
		\item	A set $E \subseteq S$ is \emph{syndetic} if there are finitely many elements $t_1, \dots, t_k \in S$
			such that $\bigcup_{i=1}^k{(E - t_i)} = S$.
		\item	A set $T \subseteq S$ is \emph{thick} if for any finite set $F \subseteq S$, there exists $t \in S$
			such that $F + t \subseteq T$.
		\item	A set $P \subseteq S$ is \emph{piecewise syndetic} if there is a syndetic set $E \subseteq S$
			and a thick set $T \subseteq S$ such that $P = E \cap T$.
		\item	A sequence $(F_N)_{N \in \N}$ of finite subsets of $S$ is a \emph{F{\o}lner sequence} if,
			for any $t \in S$,
			\begin{align*}
				\frac{\left| (F_N+t) \triangle F_N \right|}{|F_N|} \to 0.
			\end{align*}
		\item	The \emph{lower Banach density} of a set $E \subseteq S$ is the quantity
			\begin{align*}
				d_*(E) := \inf{\left\{ \liminf_{N \to \infty}{\frac{\left| E \cap F_N \right|}{|F_N|}}
				 : (F_N)_{N \in \N}~\text{is a F{\o}lner sequence in}~S \right\}}.
			\end{align*}
		\item	The \emph{upper Banach density} of a set $E \subseteq S$ is the quantity
			\begin{align*}
				d^*(E) := \sup{\left\{ \limsup_{N \to \infty}{\frac{\left| E \cap F_N \right|}{|F_N|}}
				 : (F_N)_{N \in \N}~\text{is a F{\o}lner sequence in}~S \right\}}.
			\end{align*}
	\end{itemize}
\end{defn}

The following is a standard characterization of syndetic and thick sets; see, e.g. \cite[Section 2]{BHM}.

\begin{prop}
	Let $(S, +)$ be a countable cancellative abelian semigroup.
	\begin{enumerate}[1.]
		\item	$E$ is syndetic if and only if $d_*(E) > 0$ if and only if
			$E \cap T \ne \es$ for any thick set $T \subseteq S$;
		\item	$T$ is thick if and only if $d^*(T) = 1$ if and only if
			$T \cap E \ne \es$ for any syndetic set $E \subseteq S$.
	\end{enumerate}
\end{prop}

\begin{lem} \label{lem: thick}
	Let $(S,+)$ be a countable cancellative abelian semigroup.
	Then $S$ is thick in $G_S$.
\end{lem}
\begin{proof}
	Let $F \subseteq G_S$ be a finite set.
	Write $F = \{s_i - t_i : 1 \le i \le k\}$, where $s_i, t_i \in S$.
	Put $t = \sum_{i=1}^k{t_i} \in S$.
	Then
	\begin{align*}
		F + t = \left\{ s_i + \sum_{j \ne i}{t_j} : 1 \le i \le k \right\} \subseteq S.
	\end{align*}
\end{proof}

The fact that $S$ is thick in $G_S$ is closely related to the fact that any F{\o}lner sequence in $S$ is also a F{\o}lner sequence in $G_S$, from which we deduce the following density result:

\begin{prop} \label{prop: equal density}
	Let $E \subseteq S$.
	Then $d_S^*(E) = d_{G_S}^*(E)$.
\end{prop}

\begin{proof}
    To show the inequality $d_{G_S}^*(E) \ge d_S^*(E)$, it suffices to show that any F{\o}lner sequence in $S$ is a F{\o}lner sequence in $G_S$.
    Let $(F_N)_{N \in \N}$ be a F{\o}lner sequence in $S$, and let $x \in G$.
    We want to show
    \begin{align*}
        \frac{\left| (F_N+x) \triangle F_N \right|}{|F_N|} \to 0.
    \end{align*}
    Write $x = s-t$ with $s, t \in S$.
    Then
    \begin{align*}
        \frac{\left| (F_N+x) \triangle F_N \right|}{|F_N|}
         = \frac{\left| (F_N+s) \triangle (F_N+t) \right|}{|F_N|}
         \le \frac{\left| (F_N+s) \triangle F_N \right|}{|F_N|}
         + \frac{\left| F_N \triangle (F_N+t) \right|}{|F_N|}
         \to 0.
    \end{align*}
    Hence, $(F_N)_{N \in \N}$ is a F{\o}lner sequence in $G_S$ as claimed. \\
    
    Now we show the reverse inequality $d_S^*(E) \ge d_{G_S}^*(E)$.
    If $d_{G_S}^*(E) = 0$, there is nothing to show, so assume $d_{G_S}^*(E) > 0$.
    Let $m$ be an invariant mean on $G_S$ such that $m(E) = d_{G_S}^*(E)$.
    Put $c = m(S) \ge m(E) > 0$.
    Then $\tilde{m} := \frac{1}{c}m$ is an invariant mean on $S$.
    Moreover, $\tilde{m}(E) = \frac{1}{c}m(E) \ge m(E) = d_{G_S}^*(E)$.
    Therefore, $d_S^*(E) \ge \tilde{m}(E) \ge d_{G_S}^*(E)$.
\end{proof}

\begin{lem} \label{lem: synd implies rel synd}
	Suppose $E \subseteq G_S$ is syndetic in $G_S$.
	Then $E \cap S$ is syndetic in $S$.
\end{lem}
\begin{proof}
	Let $x_1, \dots, x_k \in G_S$ such that $\bigcup_{i=1}^k{(E - x_i)} = G_S$.
	By Lemma \ref{lem: thick}, $S$ is thick, so we may assume $x_i \in S$ for each $i = 1, \dots, k$.
	We claim
	\begin{align*}
		\bigcup_{i=1}^k{\left( (E \cap S) - x_i \right)} \supseteq S.
	\end{align*}
	
	It suffices to check $(E \cap S) - x_i \supseteq (E - x_i) \cap S$ for each $i = 1, \dots, k$.
	Suppose $y \in (E - x_i) \cap S$, and let $t \in E$ such that $t - x_i = y$.
	Then $t = y + x_i \in S + S \subseteq S$.
	Hence, $y \in (E \cap S) - x_i$ as desired.
\end{proof}

\subsection{Extending main results to actions of cancellative abelian semigroups}

Any homomorphism $\varphi : S \to S$ extends uniquely to a homomorphism $\tilde{\varphi} : G_S \to G_S$
via $\tilde{\varphi} \left( s - t \right) = \varphi(s) - \varphi(t)$.
To extend our Khintchine-type results to the semigroup setting,
we need a condition on $\varphi$ characterizing when $\tilde{\varphi}(G_S)$ has finite index in $G_S$.

\begin{prop} \label{prop: synd fin index}
	Let $(S, +)$ be a countable cancellative abelian semigroup.
	Let $\varphi : S \to S$ be a homomorphism, and let $\tilde{\varphi} : G_S \to G_S$ be the group homomorphism
	$\tilde{\varphi}(s-t) := \varphi(s) - \varphi(t)$.
	The following are equivalent:
	\begin{enumerate}[(i)]
		\item	$\varphi(S)$ is a piecewise syndetic subset of $S$;
		\item	$\tilde{\varphi}(G_S)$ has finite index in $G_S$.
	\end{enumerate}
\end{prop}

\begin{proof}
	Let $T := \varphi(S)$, and let $H := \tilde{\varphi}(G_S)$.
	Note that $H = T - T = G_T$.
	
	\bigskip
	
	(i)$\implies$(ii).
	Suppose $T$ is piecewise syndetic in $S$.
	Then $d_S^*(T) > 0$.
	Thus, by Proposition \ref{prop: equal density}, $d_{G_S}^*(H) \ge d_{G_S}^*(T) = d_S^*(T) > 0$.
	But in the group $G_S$, we have the identity
	\begin{align*}
		d_{G_S}^*(H) = \frac{1}{[G_S:H]},
	\end{align*}
	so $[G_S:H] < \infty$.
	
	\bigskip
	
	(ii)$\implies$(i).
	Suppose $H$ has finite index in $G_S$.
	Then $H$ is a syndetic subset of $G_S$, so $H \cap S$ is syndetic in $S$ by Lemma \ref{lem: synd implies rel synd}.
	Moreover, by Lemma \ref{lem: thick}, $T$ is a thick subset of $H$.
	Let $\tilde{T} := T \cup (S \setminus H)$ so that $T = \tilde{T} \cap (H \cap S)$.
	We claim that $\tilde{T}$ is thick in $S$.
	
	Let $F \subseteq S$ be a finite set.
	Put $F_1 = F \cap H$ and $F_2 = F \setminus H$.
	Since $T$ is a thick subset of $H$, there exists $x \in H$ such that $F_1 + x \subseteq T$.
	Write $x = s - t$ with $s, t \in T \subseteq H \cap S$.
	Then $F_1 + s = F_1 + x + t \subseteq T + t \subseteq T$.
	Now, since $s \in H \cap S$ and $H$ is a group, we have $F_2 + s \subseteq S \setminus H$.
	Thus, $F + s = (F_1 + s) \cup (F_2 + s) \subseteq T \cup (S \setminus H) = \tilde{T}$.
	
	This shows that $\tilde{T}$ is a thick subset of $S$, so $T = \tilde{T} \cap (H \cap S)$ is piecewise syndetic in $S$.
\end{proof}

Now we can extend Theorems \ref{Khintchinefiniteindex} and \ref{Khintchineab} to the semigroup setting:

\begin{thm} \label{thm: semigroup phi, psi}
	Let $(S, +)$ be a countable cancellative abelian semigroup.
	Let $\varphi, \psi : S \to S$ be homomorphisms.
	If at least two of the three subsemigroups $\varphi(S)$, $\psi(S)$, and $(\varphi + \psi)(S)$
	are piecewise syndetic in $S$,
	then for any set $E \subseteq S$ with positive upper Banach density $d_S^*(E) > 0$ and any $\eps > 0$, the set
	\begin{align*}
		\left\{ s \in S : d_S^* \left( E \cap \left( E - \varphi(s) \right) \cap \left( E - (\varphi + \psi)(s) \right) \right)
		 > d_S^*(E)^3 - \eps \right\}
	\end{align*}
	is syndetic in $S$.
\end{thm}

\begin{rem}
	We use the pair $\{\varphi, \varphi + \psi\}$ rather than $\{\varphi, \psi\}$
	since the difference $\psi - \varphi$ is not necessarily defined as a map into $S$.
\end{rem}

\begin{proof}
	By Proposition \ref{prop: equal density}, we have $\delta := d_{G_S}^*(E) = d_S^*(E) > 0$.
	Let $\tilde{\varphi}$ and $\tilde{\psi}$ be the extensions of $\varphi$ and $\psi$ to $G_S$.
	By Proposition \ref{prop: synd fin index}, at least two of the subgroups $\tilde{\varphi}(G_S)$, $\tilde{\psi}(G_S)$,
	and $\left( \tilde{\varphi} + \tilde{\psi} \right)(G_S)$ have finite index in $G_S$.
	Hence by Theorem \ref{Khintchinefiniteindex}, the set
	\begin{align*}
		R := \left\{ g \in G : d_{G_S}^* \left( E \cap \left( E - \tilde{\varphi}(g) \right)
		 \cap \left( E - \left( \tilde{\varphi} + \tilde{\psi} \right)(g) \right) \right)
		 > \delta^3 - \eps \right\}
	\end{align*}
	is syndetic in $G_S$.
	
	By Lemma \ref{lem: synd implies rel synd}, the set $R \cap S$ is syndetic $S$.
	But
	\begin{align*}
		R \cap S = \left\{ s \in S : d_S^* \left( E \cap \left( E - \varphi(s) \right) \cap \left( E - (\varphi + \psi)(s) \right) \right)
		 > \delta^3 - \eps \right\},
	\end{align*}
	so this completes the proof.
\end{proof}

\begin{thm} \label{thm: semigroup a, b}
	Let $(S, +)$ be a countable cancellative abelian semigroup.
	Let $a, b \in \N$.
	If at least one of the three subsemigroups $aS$, $bS$, or $(a+b)S$ is piecewise syndetic in $S$,
	then for any set $E \subseteq S$ with positive upper Banach density $d_S^*(E) > 0$ and any $\eps > 0$, the set
	\begin{align*}
		\left\{ s \in S : d_S^* \left( E \cap \left( E - as \right) \cap \left( E - (a+b)s \right) \right)
		 > d_S^*(E)^3 - \eps \right\}
	\end{align*}
	is syndetic in $S$.
\end{thm}
\begin{proof}
	The proof is identical to the proof of Theorem \ref{thm: semigroup phi, psi},
	except one must use Theorem \ref{Khintchineab} in place of Theorem \ref{Khintchinefiniteindex}.
\end{proof}

\subsection{Two combinatorial questions}

Applying Theorem \ref{thm: semigroup a, b} in the semigroup $(\N, \cdot)$,
for any $E \subseteq \N$ with positive multiplicative upper Banach density $d_{mult}^*(E) > 0$,
any $k \in \N$, and any $\eps > 0$, the set of $m \in \N$ such that
\begin{align*}
	d^*_{mult} \left( E \cap E/m^k \cap E/m^{k+1} \right) > d^*_{mult}(E)^3 - \eps
\end{align*}
is multiplicatively syndetic in $\N$.
It is natural to ask if a finitary variant of this result holds.

\begin{quest}
	Let $p_1, p_2, \dots$ be an enumeration of the positive prime numbers.
	Let $\delta, \eps > 0$, and let $k \in \N$.
	Does there exists $N = N(k, \delta, \eps) \in \N$ such that the following holds:
	for any $n \ge N$ and any set $A \subseteq \left\{ p_1^{r_1} \dots p_n^{r_n} : 0 \le r_i \le n \right\}$
	with $|A| \ge \delta n^n$, there exists $y \in \N \setminus \{1\}$ such that
	\begin{align*}
		\left| \left\{ x \in \N : \{x, xy^k, xy^{k+1}\} \subseteq A \right\} \right| > \left( \delta^3 - \eps \right) n^n.
	\end{align*}
\end{quest}

Now we describe an application of Theorem \ref{thm: semigroup phi, psi}.
Let $p_1, p_2, \dots$ and $q_1, q_2, \dots$ be enumerations of the positive prime numbers.
The map $\varphi : \N \to \N$ defined by $\varphi \left( \prod_{i=1}^n{p_i^{r_i}} \right) := \prod_{i=1}^n{q_i^{r_i}}$
is an automorphism of the semigroup $(\N, \cdot)$.
Hence, by Theorem \ref{thm: semigroup phi, psi}, if $E \subseteq \N$ has positive multiplicative upper Banach density
$d_{mult}^*(E) > 0$ and $\eps > 0$,
then there is a multiplicatively syndetic set of numbers $y = \prod_{i=1}^n{p_i^{r_i}} \in \N$ such that
\begin{align} \label{eq: prime permutation}
	d_{mult}^* \left( \left\{x \in \N
	 : \left\{ x, x\prod_{i=1}^n{p_i^{r_i}}, x\prod_{i=1}^n{q_i^{r_i}} \right\} \subseteq E \right\} \right)
	 > d_{mult}^*(E)^3 - \eps.
\end{align}

The IP Szemer\'{e}di theorem of Furstenberg and Katznelson \cite{FK-IP} implies that,
for any $k \in \N$ and any multiplicative automorphisms $\varphi_1, \dots, \varphi_k : \N \to \N$,
the set of $m \in \N$ such that
\begin{align*}
	d_{mult}^* \left( E \cap E/\varphi_1(m) \cap \dots \cap E/\varphi_k(m) \right) > 0
\end{align*}
is a multiplicative IP$^*$ set and hence multiplicatively syndetic.
It is therefore natural to ask if a large intersections variant holds for families of more than two multiplicative automorphisms:

\begin{quest}
	Let $p_1, p_2, \dots$ be the enumeration of the positive prime numbers in increasing order.
	For each $j \in \N$, let $q_{j,1}, q_{j,2}, \dots$ be a distinct enumeration of the positive prime numbers.
	For which $k \in \N$ does the following hold:
	for any $E \subseteq \N$ with $d_{mult}^*(E) > 0$ and any $\eps > 0$,
	there exists $y = \prod_{i=1}^n{p_i^{r_i}} \in \N \setminus \{1\}$ such that
	\begin{align} \label{eq: distinct factorizations}
		d_{mult}^* \left( \left\{x \in \N
		 : \left\{ x, x\prod_{i=1}^n{q_{1,i}^{r_i}}, x\prod_{i=1}^n{q_{2,i}^{r_i}} , \dots, x\prod_{i=1}^n{q_{k,i}^{r_i}} \right\}
		 \subseteq E \right\} \right)
		 > d_{mult}^*(E)^{k+1} - \eps.
	\end{align}
\end{quest}

Note that \eqref{eq: distinct factorizations} holds for $k \le 2$ (see \eqref{eq: prime permutation} and the discussion above).

\appendix

\section{Proof of Lemma \ref{IZk}}\label{LeibProof}
In this section we prove Lemma \ref{IZk}, restated here for the convenience of the reader:
\begin{lem}[Lemma 3.5] \label{IZk-appendix}
    Let $(X,\mX,\mu, (T_g)_{g\in G})$ be a $G$-system and let $H\leq G$ be a subgroup of finite index. Then for every $k\geq 1$, one has $\mathcal{Z}^k_H(X) = \mZ^k_G(X)$.
\end{lem}
We follow the arguments in \cite[Appendix A]{Leib} and generalize them to arbitrary countable abelian groups. We start with some background related to the Host--Kra parallelepipeds construction.
\begin{defn}
Let $G$ be a countable abelian group, and let $\X=(X,\mX,\mu,(T_g))$ be a $G$-system. For every $k\geq 0$, we define a $G$-system $\X_G^{[k]} = (X^{[k]},\mX^{[k]}, \mu^{[k]}, (T^{[k]}_g)_{g\in G})$ inductively by setting $X_G^{[0]}=X$, and $X_G^{[k+1]} = X_G^{[k]}\times_{\mathcal{I}(X_G^{[k]})}X_G^{[k]}$ where $\mathcal{I}(X_G^{[k]})$ is the $\sigma$-algebra of $(T_g^{[k]})_{g\in G}$-invariant functions.
\end{defn}
Host and Kra proved the following result.
\begin{thm}[\cite{HK}, Proposition 4.7]\label{invariantalgebra}
$\mZ^k_G(X)$ is the minimal $\sigma$-algebra with the property that $\mathcal{I}(X^{[k]})$ is a sub $\sigma$-algebra of $(\mZ^k_G(X))^{[k]}$.
\end{thm}
Let $X=\bigcup_{\alpha\in J}X_\alpha$ be a partition of $X$ to $G$-invariant sets. Then $X_G^{[k]} = \bigcup_{\alpha\in J} X_\alpha^{[k]}$, $\mathcal{I}(X^{[k]}) = \bigvee_{\alpha\in J} \mathcal{I}(X_\alpha^{[k]})$ and $\mathcal{Z}^k_G(X) =\bigvee_{\alpha\in J} \mathcal{Z}^k_G(X_\alpha)$. Therefore, by the ergodic decomposition, it is enough to prove Lemma \ref{IZk} in the case where the $G$-action is ergodic.\\

The following lemma gives the easy inclusion in Lemma \ref{IZk}.
\begin{lem}\label{easy}
In the setting of Lemma \ref{IZk}, $\mathcal{Z}_G^k(X)\preceq \mathcal{Z}_H^k(X)$.
\end{lem}
\begin{proof}
The proof is immediate by Theorem \ref{invariantalgebra} and since any $(T_g^{[k]})_{g\in G}$-invariant function is also a $(T_h^{[k]})_{h\in H}$-invariant function.
\end{proof}
We need the following observation.
\begin{lem}\label{IZ}
    Let $G$ be a countable abelian group, let $\X=(X,\mX,\mu,(T_g)_{g\in G})$ be an ergodic measure preserving $G$-system, and let $H\leq G$ be a subgroup of finite index. Then $\mI_H(X)\preceq \mZ_G(X)$.
\end{lem}
\begin{proof}
The group $G/H$ acts ergodically by unitary transformations on $\mathcal{H}=L^2(X,\mathcal{I}_H,\mu|_{\mathcal{I}_H})$. Since $G/H$ is a finite abelian group, the unitary representation splits into a direct sum of one-dimensional irreducible representations. In other words, $\mathcal{H}$ is generated by eigenfunctions of the action of $G/H$, which are measurable with respect to $\mZ_G(X)$. This completes the proof.
\end{proof}
Now we prove the $k=1$ case of Lemma \ref{IZk} under the additional assumption that the action of $H$ is ergodic.
\begin{lem} \label{IZ1}
Let $G$ be countable abelian groups, and let $H \le G$ be a finite index subgroup. Let  $\X=(X,\mX,\mu,(T_g)_{g\in G})$ be an ergodic $G$-system, and suppose the action of $H$ is ergodic. Then $\mZ_H(X)=\mZ_G(X)$.
\end{lem}
\begin{proof}
The group $G/H$ is finite, and therefore it is a direct product of finite cyclic groups. In particular, we can find $d \in \N$ and a sequence of subgroups $H_0 = H \le H_1 \le \dots \le H_d \le G$ such that $G/H_d$ and $H_i/H_{i-1}$, $1\leq i \leq d$, are cyclic groups of prime order. Using a proof by induction on $d$, we may assume without loss of generality that $G/H$ is cyclic and of prime order. Let $g_0\in G$ be a representative of a generator of $G/H$ and $l:=[G:H]$ be a prime number. By the ergodicity of $H$, the $\sigma$-algebra $\mathcal{Z}_H(X)$ is generated by $H$-eigenfunctions. Hence, it is enough to show that every $H$-eigenfunction $f$ is a linear combination of $G$-eigenfunctions. Let $\lambda:H\rightarrow S^1$ be the eigenvalue of $f$ and observe that for any $l$-th root $\omega \in S^1$ of $\lambda(lg_0)$ the function $$ f+\omega\cdot T_{g_0}f + ...+\omega^{l-1} \cdot T_{(l-1)g_0} f$$ is a $G$-eigenfunction. Now since $$f=\sum_{\omega \in S^1~:~\omega^l = \lambda(lg_0)} f+\omega\cdot T_{g_0}f + ...+\omega^{l-1} \cdot T_{(l-1)g_0} f,$$
$f$ is measurable with respect to $\mZ_G(X)$ and this completes the proof.
\end{proof}
Let $G$ be a countable abelian group, and let $\X=(X,\mX,\mu,(T_g)_{g\in G})$ be a $G$-system. If the system $\X$ is ergodic, it follows from the definition that $X_G^{[1]}$ is the Cartesian product of $X$ with itself, and the measure is the product measure. As a consequence of Lemma \ref{IZ1}, we have:
\begin{lem} \label{lem: induction step}
If the action of $H$ on $X$ is ergodic, then
 $$\mathcal{I}(X_H^{[1]}) = \mathcal{I}(X_G^{[1]}).$$
\end{lem}
\begin{proof}
The inclusion $\mathcal{I}(X_G^{[1]})\preceq\mathcal{I}(X_H^{[1]}))$ is trivial. Now let $f:X\times X\rightarrow\mathbb{C}$ be a $(T_h\times T_h)_{h\in H}$ invariant function. By Lemma \ref{IZ1}, we can find an orthonormal basis of $G$-eigenfunctions $\{f_i\}_{i\in\mathbb{N}}$ for $\mathcal{Z}_H(X)$. By Lemma \ref{TT}, there exist constants $a_{i,j}\in\mathbb{C}$ for all $i,j\in\mathbb{N}$ such that $$f(x,y) = \sum_{i=1}^\infty a_{i,j} f_i(x) \overline{f_j}(y).$$
Applying the $H$-action and using the uniqueness of the decomposition, we see that $a_{i,j}=0$ unless $i=j$. In particular $f$ is spanned by the $G$-invariant functions $f_i\otimes \overline{f_i}$. Thus, $f$ is measurable with respect to $\mathcal{I}(X_G^2)$ and the claim follows.
\end{proof}
We use Lemma \ref{lem: induction step} to prove the following:
\begin{prop} \label{H ergodic case}
If the action of $H$ on $X$ is ergodic, then for $k \ge 0$, one has $$\mathcal{I}(X_H^{[k]}) = \mathcal{I}(X_G^{[k]}) \qquad \text{and} \qquad \mu_G^{[k]}= \mu_H^{[k]}.$$
\end{prop}
\begin{proof}
We prove the claim by induction on $k$. The case $k=0$ is trivial.

Assume that for some $k \ge 0$, $\mathcal{I}(X_H^{[k]}) = \mathcal{I}(X_G^{[k]})$ and $\mu_G^{[k]}= \mu_H^{[k]}$. It is immediate that $$\mu_G^{[k+1]} = \mu_G^{[k]}\times_{\mathcal{I}(X_G^{[k]})} \mu_G^{[k]} = \mu_H^{[k]}\times_{\mathcal{I}(X_H^{[k]})} \mu_H^{[k]} = \mu_H^{[k+1]}.$$

By the ergodic decomposition theorem, applied with respect to the $\sigma$-algebra $\mathcal{I}(X_G^{[k]})$ we can find a partition $X_G^{[k]}=\bigcup_{\alpha\in J} X_\alpha$ of $X_G^{[k]}$ to $(T_g^{[k]})_{g\in G}$ invariant sets. Let $S_g^\alpha$ be the restriction of $T_g^{[k]}$ to the set $X_\alpha$. By the induction hypothesis the action of $(S_h^\alpha)_{h\in H}$ on $X_\alpha$ is ergodic. Hence, by Lemma \ref{lem: induction step}, we have
$$ 
\mathcal{I}(X_H^{[k+1]}) = \bigcup_{\alpha\in J} \mathcal{I}_H(X_\alpha^{[1]})) =\bigcup_{\alpha\in J} \mathcal{I}_G(X_\alpha^{[1]})  = \mathcal{I}(X_G^{[k+1}),
$$
as required.
\end{proof}
Proposition \ref{H ergodic case} establishes Lemma \ref{IZk} in the case where the action of $H$ is ergodic. Now we assume that the $H$-action is non-ergodic. As in the proof of Lemma \ref{IZ1}, we may assume without loss of generality that $G/H$ is cyclic of order $l$ for some prime $l$. In particular, there exists a partition $X=\bigcup_{i\in \Z/l\Z} X_i$ into $H$-invariant sets and some $g_0\in G$ such that $T_{g_0} X_i = X_{i+1}$, $i \in \Z/l\Z$.

We need the following technical lemma.
\begin{lem} \label{lem: fiber product partition}
Let $G$ be a countable abelian group, and let $\textbf{Y}=(Y,\mY,\nu,(T_g)_{g\in G})$ be an ergodic $G$-system.
Suppose that there exists some $g_0\in G$ and $H$-invariant subsets $Y_i$ such that $Y=\bigcup_{i\in\Z/l\Z} Y_i$, and $T_{g_0} Y_i = Y_{i+1}$ for $i \in \Z/l\Z$.
Then, $Y\times_{\mathcal{I}_G(Y)} Y = \bigcup_{i,j\in\Z/l\Z} Y_{i,j}$ where $Y_{i,i}=Y_i\times_{\mathcal{I}_H(Y_i)} Y_i$ and $T_{sg_0}\times T_{tg_0}$ is an isomorphism between $Y_{i,i}$ and $Y_{i+s, i+t}$, $i \in \Z/l\Z$.
\end{lem}
\begin{proof}
Let $A\in \mathcal{I}_G(Y)$ be a measurable $G$-invariant subset of $Y$. For each $0\leq i \leq l-1$, $A_i = A\cap Y_i$ is an $H$-invariant set. In particular, $A_0$ is $H$-invariant and $A_i=T_{ig_0}A_0$. We deduce that the mapping $A\mapsto A\cap Y_0$ is an isomorphism between $\mathcal{I}_G(Y)$ and $\mathcal{I}_H(Y_0)$. Using the ergodic decomposition we can find a partition $$Y_0 = \bigcup_{\alpha\in I} Y_{0,\alpha}$$ of $Y_0$ to $H$-invariant sets. For every $\alpha\in I$, and $i\not =0$, let $Y_{i,\alpha} = T_{ig_0} Y_{0,\alpha}$ and $Y_\alpha = \bigcup_{i\in\Z/l\Z} Y_{i,\alpha}$. Then, $Y=\bigcup_{\alpha\in I} Y_\alpha$ is the ergodic decomposition of $Y$ with respect to the factor $\mathcal{I}_G(Y)$. Thus, if we let $Y_{i,j}=\bigcup_{\alpha\in I} Y_{i,\alpha}\times Y_{j,\alpha}$ we have,
$$Y_G^{[1]} = \bigcup_{\alpha\in I} (Y_\alpha\times_{\mathcal{I}_G(Y_\alpha)} Y_\alpha) =  \bigcup_{\alpha\in I}\bigcup_{i,j\in\Z/l\Z} (Y_{i,\alpha}\times Y_{j,\alpha}) =\bigcup_{i,j\in\Z/l\Z}\bigcup_{\alpha\in I}\ (Y_{i,\alpha}\times Y_{j,\alpha}) =\bigcup_{i,j\in\Z/l\Z} Y_{i,j}.$$ In particular, $Y_{i,i} = \bigcup_{\alpha\in I} (Y_{i,\alpha}\times Y_{i,\alpha}) =  Y_i\times Y_i$, as required.
\end{proof}
Recall that $G=\bigcup_{i=0}^{l-1} ig_0+H$. It follows from Lemma \ref{lem: fiber product partition} that for $i,j \in \Z/l\Z$, $$(T_{g_0}\times T_{g_0}) (Y_i\times_{\mathcal{I}_H(Y)} Y_j) = Y_{i+1, j+1}.$$ Therefore, the subsets $V_i = \bigcup_{j\in\Z/l\Z} Y_{j,j+i}$, $i \in \Z/l\Z$ form a partition of $Y\times_{\mathcal{I}_G(Y)} Y$ into $(T_g\times T_g)_{g\in G}$-invariant sets. Furthermore, $\text{Id}\times T_{ig_0}$ is an isomorphism between $V_0$ and $V_i$.\\

We use Lemma \ref{lem: fiber product partition} to show the following:
\begin{lem} \label{lem: parallelopiped partition}
Let $\X=(X,\mX,\mu,(T_g)_{g\in G})$ be an ergodic $G$-system. Let $X=\bigcup_{i\in\Z/l\Z} X_i$ be a partition into $H$-invariant sets, and let $g_0\in G$ be as above. Then for any $k\geq 0$, there exists a partition $X_G^{[k]} = \bigcup_{j\in(\Z/l\Z)^k} W_j$, into $(T_g^{[k]})_{g\in G}$-invariant sets, such that $W_0 =\bigcup_{i\in\Z/l\Z} (X_i)_H^{[k]} $ and $T_{g_0}^{[k]}\left((X_i)_H^{[k]} \right)=(X_{i+1})_H^{[k]}$. Furthermore, for every $j\in(\Z/l\Z)^k$, there exists an isomorphism of measure spaces $\tau_j:W_0\rightarrow W_j$, which in every coordinate of $X^{[k]}$ is a power of $T_{g_0}$.
\end{lem}
\begin{proof}
We induct on $k$. The case $k=0$ is trivial.

Assume that the claim holds for some $k\geq 0$. Then $$X_G^{[k+1]} = X_G^{[k]}\times_{\mathcal{I}(X_G^{[k]})}X_G^{[k]} =\bigcup_{j\in(\Z/l\Z)^k} (W_j\times_{\mathcal{I}(W_j)} W_j).$$
Fix $j \in (\Z/l\Z)^k$.
Since the isomorphism $\tau_j:W_0\rightarrow W_j$ commutes with $(T_g^{[k]})_{g\in G}$, it induces an isomorphism $\tau_j\times \tau_j:W_0\times_{\mathcal{I}(W_0)} W_0\rightarrow W_j\times_{\mathcal{I}(W_j)} W_j$. By assumption $W_0 = \bigcup_{i\in\Z/l\Z}(X_i)_H^{[k]}$ and by Lemma \ref{lem: fiber product partition}, $W_0\times_{\mathcal{I}(W_0)}W_0$ can be partitioned into $(T_g^{[k+1]})_{g\in G}$-invariant sets $\{V_i\}_{i\in\Z/l\Z}$ such that $$V_0 = \bigcup_{i\in\Z/l\Z}\left( (X_i)_H^{[k]}\times_{\mathcal{I}\left((X_i)_H^{[k]}\right)} (X_i)_H^{[k]}\right) =\bigcup_{i\in\Z/l\Z} (X_i)_H^{[k+1]}.$$
Moreover, $V_0$ is isomorphic to $V_j$ via an isomorphism whose projections are powers of $T_{g_0}^{[k]}$. Since $W_0$ is isomorphic to $W_j$, this completes the proof.
\end{proof}
We recall that it suffices to establish the proof of Lemma \ref{IZk} in the case where the $G$-action is ergodic and $G/H$ is a cyclic group of order $l$ for some $l>0$. As before, we find a partition $X=\bigcup_{i\in\Z/l\Z} X_i$ of $X$ into $H$-invariant sets and some $g_0\in G$ such that $T_{g_0} (X_i) = X_{i+1}$ for $i \in \Z/l\Z$.
\begin{proof}[Proof of Lemma \ref{IZk}]
Let $k\geq 0$, and let $\{W_i\}_{i\in(\Z/l\Z)^k}$ be as in Lemma \ref{lem: parallelopiped partition}.
Since $X_0,...,X_{l-1}$ are disjoint $(T_h)_{h\in H}$-invariant subsets of $X$, we have $\mathcal{I}(X_H^{[k]}) = \prod_{i\in\Z/l\Z} \mathcal{I}\left((X_i)_H^{[k]}\right)$ and $Z^k_H(X)=\prod_{i\in\Z/l\Z} Z^k_H(X_i)$.
Let $B$ be a $(T_h^{[k]})_{h\in H}$-invariant subset of $(X_i)_H^{[k]}$. 
For every $j\in\Z/l\Z$, let $A_j=(T_{(j-i)g_0}^{[k]})(B)$ and $A=\bigcup_{j\in\Z/l\Z} A_j$.
By definition $A\subseteq W_0$ is a $(T_g^{[k]})_{g\in G}$-invariant set. 
Therefore, by Theorem \ref{invariantalgebra}, $A\in \left(\mZ^k_G(X)\right)^{[k]}$.
Since $X_i$ is $(T_h^{[k]})$-invariant, by Lemma \ref{IZ}, $X_i\in\mZ^1_G(X)$.
Therefore, $B=A_i = A\cap \left(X_i\right)_H^{[k]}$ is an element of $\left(\mZ^k_G(X)\right)^{[k]}$.
Since $B$ is arbitrary, and this holds for all $i\in\Z/l\Z$, we deduce that $\mathcal{I}(X_H^{[k]})\preceq \mathcal{Z}^k_G(X)$.
By Theorem \ref{invariantalgebra}, we have $\mathcal{Z}^k_H(X)\preceq \mathcal{Z}_G^k(X)$.
Lemma \ref{easy} provides the other inclusion, and this completes the proof.
\end{proof}


\begin{thebibliography}{9}

\bibitem[AB21]{AB}
\newblock	E.~Ackelsberg and V.~Bergelson.
\newblock	Popular differences for polynomial patterns in rings of integers.
\newblock	arXiv:2107.07626.

\bibitem[ABB21]{ABB}
\newblock E.~Ackelsberg, V.~Bergelson, and A.~Best.
\newblock	Multiple recurrence and large intersections for abelian group actions.
\newblock	\emph{Discrete Anal.} 2021:18, 91 pp.

\bibitem[AWM60]{AWM}
\newblock	F. V. Atkinson, G. A. Watterson and P. A. P. Moran.
\newblock	A matrix inequality.
\newblock \emph{Quart. J. Math. Oxford Ser. (2)}, 11:137--140, 1960.

\bibitem[Au16]{austin}
\newblock	T. Austin.
\newblock	Non-conventional ergodic averages for several commuting actions of an amenable group.
\newblock	\emph{J. Anal. Math.}, 130:243--274, 2016.

\bibitem[Beh46]{Beh}
\newblock	F. A. Behrend.
\newblock	On sets of integers which contain no three terms in arithmetical progression.
\newblock	\emph{Proc. Nat. Acad. Sci. U.S.A.}, 32:331--332, 1946.

\bibitem[B87]{Be}
\newblock	V. Bergelson.
\newblock	Weakly mixing PET.
\newblock	\emph{Ergodic Theory Dynam. Systems}, 7:337--349, 1987.

\bibitem[B06]{Leib}
\newblock	V. Bergelson.
\newblock	Combinatorial and Diophantine applications of ergodic theory. Appendix A by A. Leibman and Appendix B by A. Quas and M. Wierdl.
\newblock	In \emph{Handbook of dynamical systems.}, Vol. 1B, 745--869, Elsevier B. V., Amsterdam 2006.

\bibitem[BF21]{BFe}
\newblock	V. Bergelson and A. Ferr\'e Moragues.
\newblock	An ergodic correspondence principle, invariant means and applications.
\newblock	\emph{Israel J. Math.}, to appear. arXiv:2003:03029.

\bibitem[BHM98]{BHM}
\newblock	V.~Bergelson, N.~Hindman, and R.~McCutcheon.
\newblock	Notions of size and combinatorial properties of quotient sets in semigroups.
\newblock	In \emph{Proceedings of the 1998 {T}opology and {D}ynamics {C}onference ({F}airfax, {VA})}, Topology Proc., 23:23--60, 1998.

\bibitem[BHK05]{BHK}
\newblock	V. Bergelson, B. Host and B. Kra.
\newblock	Multiple recurrence and nilsequences. With an appendix by Imre Ruzsa.
\newblock	\emph{Invent. Math.}, 160(2):261--303, 2005.

\bibitem[BL15]{BL-cubic}
\newblock	V.~Bergelson and A.~Leibman.
\newblock	Cubic averages and large intersections.
\newblock	In \emph{Recent Trends in Ergodic Theory and Dynamical Systems}, volume 631 of Contemp. Math., 5--19, Amer. Math. Soc., Providence, RI, 2015.

\bibitem[BLL08]{BLL}
\newblock   V. Bergelson, A. Leibman, and E. Lesigne.
\newblock   Intersective polynomials and the polynomial Szemer\'edi theorem.
\newblock   \emph{Adv. Math.}, 219(1):369--388, 2008.

\bibitem[BTZ10]{BTZ}
\newblock	V. Bergelson, T. Tao, and T. Ziegler.
\newblock	An inverse theorem for the uniformity seminorms associated with the action of $\mathbb{F}_p^\infty$
\newblock	\emph{Geom. Funct. Anal.}, 19(6):1539--1596, 2010.

\bibitem[Ber21]{Berger}
\newblock	A.~Berger.
\newblock	Popular differences for corners in Abelian groups.
\newblock	\emph{Math. Proc. Cambridge Philos. Soc}, 171(1):207--225, 2021.

\bibitem[BSST21]{BSST}
\newblock	A.~Berger, A.~Sah, M.~Sawhney, and J.~Tidor.
\newblock	Popular differences for matrix patterns.
\newblock	\emph{Trans. Amer. Math. Soc.}, to appear. arXiv:2102.01684.

\bibitem[C11]{Chu}
\newblock	Q. Chu.
\newblock	Multiple recurrence for two commuting transformations.
\newblock \emph{Ergodic Theory Dynam. Systems}, 31(3):771–792, 2011.

\bibitem[CFH11]{CFH}
\newblock	Q. Chu, N. Frantzikinakis, and B. Host.
\newblock	Commuting averages with polynomial iterates of distinct degrees.
\newblock	\emph{Proc. Lond. Math. Soc. (3)} 102(5):801--842, 2011.

\bibitem[DLMS21]{DLMS}
\newblock   S. Donoso, A. Le, J. Moreira, and W. Sun.
\newblock   Optimal lower bounds for multiple recurrence.
\newblock   \emph{Ergodic Theory Dynam. Systems}, 41:379–407, 2021.

\bibitem[DS18]{DS}
\newblock	S.~Donoso and W.~Sun.
\newblock	Quantitative multiple recurrence for two and three transformations.
\newblock	\emph{Israel J. Math.}, 226(1):71--85, 2018.

\bibitem[FSSSZ20]{FSSSZ}
\newblock	J.~Fox, A.~Sah, M.~Sawhney, D.~Stoner, and Y.~Zhao.
\newblock	Triforce and corners.
\newblock	\emph{Math. Proc. Cambridge Philos. Soc.}, 169(1):209--223, 2020.

\bibitem[F08]{Fra}
\newblock   N. Frantzikinakis.
\newblock   Multiple ergodic averages for three polynomials and applications.
\newblock   \emph{Trans. Amer. Math. Soc.}, 360(10):5435--5475, 2008.

\bibitem[FH18]{FranHost}
\newblock	N. Frantzikinakis, B. Host.
\newblock	Weighted multiple ergodic averages and correlation sequences.
\newblock	\emph{Ergodic Theory Dynam. Systems}, 38(1):81--142, 2018.

\bibitem[FK85]{FK-IP}
\newblock	H. Furstenberg and Y. Katznelson.
\newblock	An ergodic {S}zemer\'{e}di theorem for {IP}-systems and combinatorial theory.
\newblock	\emph{J. Analyse Math.}, 45:117--168, 1985.

\bibitem[FW96]{F&W}
\newblock	H.Furstenberg and B. Weiss.
\newblock	A Mean ergodic theorem for $\frac{1}{N}\sum_{n=1}^{N}f(T^n(x))g(T^{n^2}(x))$
\newblock	In \emph{Convergence in Ergodic Theory and Probability, Columbus, OH 1993 (Bergelson, March, and Rosenblatt, eds.)}, Ohio State Univ. Math. Res. Inst. Publ. 5, de Gruyter, Berlin (1996), 193-227.

\bibitem[G01]{G}
\newblock	T. Gowers.
\newblock	A new proof of Szemeredi’s theorem.
\newblock	\emph{Geom. Func. Anal.}, 11:465--588, 2001.

\bibitem[HK05]{HK}
\newblock	B. Host and B. Kra.
\newblock	Nonconventional ergodic averages and nilmanifolds.
\newblock	\emph{Ann. Math.}, 161:397--488, 2005.

\bibitem[JST21]{JST}
\newblock   A. Jamneshan, O. Shalom, and T. Tao.
\newblock   The structure of arbitrary Conze--Lesigne systems.
\newblock   arXiv:2112.02056.

\bibitem[Kh35]{Kh}
\newblock	A. Khintchine.
\newblock	Eine Versch\"arfung des Poincar\'eschen “Wiederkehrsatzes”.
\newblock	\emph{Compositio Math.}, 1:177–179, 1935.

\bibitem[M21]{Man}
\newblock	M.~Mandache.
\newblock	A variant of the {C}orners theorem.
\newblock	\emph{Math. Proc. Cambridge Philos. Soc.}, 171(3):607--621, 2021.

\bibitem[MS80]{MS}
\newblock   C. Moore and K. Schmidt.
\newblock Coboundaries and homomorphisms for non-singular actions and a problem of H. Helson.
\newblock	\emph{Proc. London Math. Soc. (3)}, 40:443--475, 1980.

\bibitem[R90]{rudin}
\newblock   W. Rudin.
\newblock   \emph{Fourier Analysis on Groups},
\newblock   Wiley Classics Library, John Wiley \& Sons, Inc., New York, 1990.

\bibitem[SSZ21]{SSZ}
\newblock	A.~Sah, M.~Sawhney, and Y.~Zhao.
\newblock	Patterns without a popular difference.
\newblock	\emph{Discrete Anal.}, 2021:8, 30 pp.

\bibitem[S21]{OS2}
\newblock	O. Shalom.
\newblock	Multiple ergodic averages in abelian groups and Khintchine type recurrence.
\newblock	\emph{Trans. Amer. Math. Soc.}, to appear. arXiv:2102.07273.

\bibitem[Sri98]{srivastava}
\newblock S. M. Srivastava.
\newblock \emph{A Course on Borel Sets}, volume 180 of Graduate Texts in Mathematics, Springer-Verlag, New York, 1998.

\bibitem[TZ16]{TZ}
\newblock	T. Tao and T. Ziegler.
\newblock	Concatenation theorems for anti-Gowers-uniform functions and Host--Kra characteristic factors.
\newblock	\emph{Discrete Anal.}, 2016:13, 60 pp.

\bibitem[Zie07]{Z}
\newblock   T. Ziegler.
\newblock   Universal characteristic factors and Furstenberg averages.
\newblock   \emph{J. Amer. Math. Soc.}, 20:53--97, 2007.

\bibitem[Zim76]{Zim}
\newblock	R. Zimmer.
\newblock	Extensions of ergodic group actions.
\newblock	\emph{Illinois J. Math.} 20:373--409, 1976.

\bibitem[Z-K16]{zorin}
\newblock	P. Zorin-Kranich.
\newblock	Norm convergence of multiple ergodic averages on amenable groups.
\newblock	\emph{J. Anal. Math.}, 130:219–241, 2016.

\end{thebibliography}
\end{document}